\documentclass[11pt,a4paper]{article}

\usepackage[utf8]{inputenc}
\usepackage[T1]{fontenc}
\usepackage{amsmath,amssymb,amsthm,mathtools}
\usepackage[margin=1in]{geometry}
\usepackage{enumitem}
\usepackage{booktabs}
\usepackage{graphicx}
\usepackage[colorlinks=true,linkcolor=blue,citecolor=blue,urlcolor=blue]{hyperref}
\newtheorem{mainthm}{Theorem}

\newtheorem{maincor}{Corollary}[mainthm]
\newtheorem{lemma}{Lemma}[section]
\newtheorem{proposition}[lemma]{Proposition}

\newtheorem{blackbox}[lemma]{Theorem}
\theoremstyle{definition}
\newtheorem{definition}[lemma]{Definition}
\theoremstyle{remark}
\newtheorem{remark}[lemma]{Remark}

\newcommand{\R}{\mathbb{R}}
\newcommand{\dd}{\,\mathrm{d}}
\newcommand{\norm}[1]{\lVert #1\rVert}
\newcommand{\abs}[1]{\lvert #1\rvert}
\newcommand{\Lone}{L^1(\Omega)}
\newcommand{\Lzero}{L^1_0(\Omega)}
\newcommand{\Ustar}{U^{*}}
\newcommand{\rmax}{r_{\max}}
\newcommand{\rmin}{r_{\min}}

\title{Fitness is asymptotically irrelevant under uniform competition\\ in phenotype-structured populations}

\author{Artur C\'esar Fassoni\textsuperscript{1,2}\\[2pt]
\small \textsuperscript{1}Instituto de Matem\'atica e Computa\c{c}\~ao, Universidade Federal de Itajub\'a, Itajub\'a, Brazil\\
\small \textsuperscript{2}Carl Gustav Carus School of Medicine, Technische Universit\"at Dresden, Dresden, Germany\\
\small \texttt{fassoni@unifei.edu.br}}

\date{}

\begin{document}
\maketitle

\begin{abstract}
We prove that, in phenotype-structured populations under uniform competition, the long-term phenotypic distribution does not depend on the proliferation rates: it is the stationary distribution of the switching dynamics alone. The density $u(x,t)$ of cells with phenotype $x$ evolves by advection and diffusion in phenotype space (switching) and by proliferation at a phenotype-dependent rate $r(x,t)$, modulated by a factor $g(U)$ common to all phenotypes, where $U$ is the total population. This extends to continuous phenotypes a recent result for compartmental models. On a bounded interval with no-flux boundary conditions, the total population converges monotonically to the carrying capacity $U^*$, the density converges in $L^1$ to $U^*\psi$, where $\psi$ is the stationary density of the switching dynamics, and the nonlinear model is asymptotically equivalent to the linear one; by an explicit estimate, the imprint of fitness decays at the spectral gap of the switching dynamics. On the whole real line we prove the same convergence under two natural conditions, confinement and no escape to infinity, covering Ornstein--Uhlenbeck dynamics, multi-well landscapes, and heavy-tailed densities. If $g'(U^*)<0$ and $\psi$ satisfies a Poincar\'e inequality, convergence is exponential in a weighted $L^2$ norm, while no uniform $L^1$ rate holds for Ornstein--Uhlenbeck dynamics. The proof follows the compartmental one: under uniform competition the reaction term has one sign, so its size is the growth rate of the total population, integrable in time, and the switching dynamics contracts the zero-mass perturbations it produces. Since it uses only the Markov semigroup generated by switching, the proof extends to general mechanisms; we apply it to nonlocal switching by jumps and to diffusions with non-gradient drift in bounded domains of $\R^d$, whose stationary density is not of Boltzmann form and carries a circulating flux.
\end{abstract}




\section{Introduction}\label{sec:intro}

\subsection{The question}

Mathematical models of phenotypic plasticity usually combine two mechanisms. A \emph{switching dynamics} moves cells between phenotypic states: in compartmental models it is a linear system of ODEs $u'=Au$ with a transition-rate matrix $A$, and when the phenotype $x$ is a continuous variable it is a linear advection--diffusion equation
\begin{equation}\label{eq:linear}
\partial_tu=\mathcal Lu,\qquad\mathcal Lu:=-\partial_x\bigl(v(x)u\bigr)+\partial_x\bigl(D(x)\partial_xu\bigr),
\end{equation}
in which the diffusivity $D$ describes unbiased random switching and the velocity $v$ a directional bias, for instance the gradient $v=-P'$ of an epigenetic potential \cite{Review}.\footnote{We write the diffusion term in the divergence (Fickian) form $\partial_x(D\partial_xu)$. If switching is described by the It\^o equation $\mathrm dX=b(X)\,\mathrm dt+\sqrt{2D(X)}\,\mathrm dW$, the Fokker--Planck operator $-\partial_x(bu)+\partial_{xx}(Du)$ is of the form \eqref{eq:linear} with $v=b-D'$; for constant $D$ the two forms coincide. All results below apply to either convention, with $v$ interpreted accordingly.} A \emph{vital dynamics} makes cells proliferate and compete. We study the case of \emph{uniform competition}, in which all phenotypes compete for the same resource: the proliferation rate $r(x,t)\ge0$ depends on the phenotype, but the density-dependent brake $g$ is the same for all phenotypes and depends only on the total population. The model is
\begin{equation}\label{eq:model}
\partial_tu=\mathcal Lu+g\bigl(U(t)\bigr)\,r(x,t)\,u,\qquad U(t)=\int u(x,t)\dd x ,
\end{equation}
with no-flux boundary conditions, so that switching neither creates nor destroys cells. Examples of $g$ are the logistic $g(U)=1-U/K$ and the Gompertz $g(U)=\ln(K/U)$ laws; the carrying capacity $U^*$ is the zero of $g$.

Equation~\eqref{eq:model} belongs to the family of phenotype-structured population models, in which the growth of each phenotype is coupled to the whole population through a nonlocal term. These models have been studied for the selection they produce and for the concentration of the population on a few phenotypes when phenotypic changes are small \cite{Perthame2007,DesvillettesJabinMischlerRaoul2008,JabinRaoul2011,LorzMirrahimiPerthame2011}, and as models of non-genetic drug resistance and tumour heterogeneity \cite{Chisholm2015,LorenziChisholmClairambault2016,PoucholClairambaultLorzTrelat2018,ClairambaultPouchol2019,AlvarezCarrilloClairambault2022}; see \cite{LorenziPainterVilla2025} for a recent tutorial. Density dependence often enters additively, as a death rate that grows with crowding and is the same for all phenotypes, in the net growth rate $r(x)-d(U)$, or, more generally, through a competition kernel, in $a(x)-\int b(x,y)u(y)\dd y$ \cite{DesvillettesJabinMischlerRaoul2008}. In general, the net growth rates of different phenotypes do not vanish at the same population size; selection then acts at all times, and the long-term phenotypic distribution depends on the proliferation rates. Uniform competition is the special case in which crowding multiplies the proliferation rate by a common factor, $r(x)\,g(U)$; for the logistic law it is the kernel $b(x,y)=r(x)/K$, proportional to $a=r$. At the carrying capacity the net growth rate $r(x)\,g(U^*)$ vanishes for every phenotype: all phenotypes are selectively neutral, and there is no fittest phenotype to select. The two cases are compared in Remark~\ref{rem:replicator}, Figure~\ref{fig:numerics} and Section~\ref{sec:disc}.

Faster-proliferating phenotypes are nevertheless favoured by selection while the population grows, so one might expect them to be over-represented in the long run. For compartmental models, Giaimo, Shah, Raatz and Traulsen \cite{Giaimo2025} showed that, on the contrary, under uniform competition the nonlinear model is asymptotically equivalent to the linear switching system, so that the long-term composition is determined by the transition rates alone. An elementary proof, with an explicit estimate, is given in \cite[Appendix~A]{Review}. For continuously structured populations, the following counterpart was stated in the recent review \cite{Review}, and illustrated there by numerical simulations of \eqref{eq:model}, with the proof deferred to the present paper:

\begin{quote}
\emph{Under uniform competition, the solution of \eqref{eq:model} approaches a solution of the linear equation \eqref{eq:linear} with total mass $U^*$; in particular, the long-term phenotypic distribution is the stationary distribution of \eqref{eq:linear}, independently of the proliferation rates.}
\end{quote}

We prove this statement on a bounded phenotype interval (Theorem~\ref{thm:A}), extend it to the whole real line (Theorem~\ref{thm:B}), to general switching mechanisms (Theorem~\ref{thm:C}) and to bounded phenotype domains in $\R^d$ (Corollary~\ref{cor:Rd}), and quantify how fast fitness is forgotten.

\subsection{Results in brief}

\emph{Bounded phenotype space (Theorem~\ref{thm:A}).} Let $\Omega=(0,\ell)$, with smooth $D>0$ and $v$, and no-flux boundary conditions. Let $\psi(x)=Z^{-1}\exp\bigl(\int_0^xv/D\bigr)$ be the stationary probability density of \eqref{eq:linear}. If proliferation is bounded, and persistent on some sub-interval of phenotypes, then for every nonnegative initial density: $U(t)\to U^*$ monotonically; $u(\cdot,t)\to U^*\psi$ in $L^1(\Omega)$; the nonlinear solution shadows \emph{every} solution of the linear equation with mass $U^*$; and
\begin{equation}\label{eq:star-intro}
\norm{u(t)-U(t)\psi}_{L^1}\le Ce^{-\lambda t}\norm{u(0)-U(0)\psi}_{L^1}+2C\int_0^te^{-\lambda(t-s)}\abs{U'(s)}\dd s ,
\end{equation}
where $C,\lambda>0$ depend only on the switching dynamics, and $\lambda$ can be taken equal to its spectral gap $\lambda_1$, which is optimal. When $v=-P'$ the limit is the Boltzmann-type density $\psi\propto\exp\bigl(-\int P'/D\bigr)$.

\emph{Whole real line (Theorem~\ref{thm:B}).} On $\Omega=\R$, with general $v(x)$ and $D(x)>0$, we prove convergence to $U^*\psi$ in $L^1(\R)$ and asymptotic equivalence under two conditions: \emph{confinement}, $\psi\in L^1(\R)$, and \emph{no escape to infinity}, $\int^{\pm\infty}\dd x/(D\psi)=\infty$, which holds for instance whenever $D$ is bounded. This includes the Ornstein--Uhlenbeck dynamics, multi-well potentials growing at infinity, and the heavy-tailed densities produced by state-dependent noise. For the Ornstein--Uhlenbeck dynamics no exponential rate can hold uniformly in $L^1(\R)$ (Proposition~\ref{prop:OUnorate}); but when $\psi$ satisfies a Poincar\'e inequality, as in all these examples, and $g'(U^*)<0$, the convergence is exponential in the weighted norm $(\int f^2/\psi)^{1/2}$ for initial data of finite weighted norm (Corollary~\ref{cor:rateR}).

\emph{General switching dynamics (Theorem~\ref{thm:C}).} The proofs use the switching dynamics only through the Markov semigroup it generates. Section~\ref{sec:general} states the general result, for any Markov semigroup that forgets zero-mass perturbations; it contains the compartmental result and Theorems~\ref{thm:A} and~\ref{thm:B}. We apply it to nonlocal switching, $\int_\Omega[K(x,y)u(y)-K(y,x)u(x)]\dd y$, for which the convergence is exponential under a connectivity condition on the kernel $K$ (Corollary~\ref{cor:nonlocal}), and to no-flux diffusions with drifts that need not be gradients in bounded domains of $\R^d$ (Corollary~\ref{cor:Rd}); there the limit is a stationary density that is not explicit and carries a circulating probability flux, and it still does not depend on the proliferation rates.

\subsection{Idea of the proof, and the analogy with compartmental models}

The proof follows, step by step, the elementary proof of the compartmental result in \cite{Review}, which we summarise. There, $u'=Au+F$ with $F=g(U)Ru$, and:

\begin{enumerate}[label=(\roman*), itemsep=2pt]
\item Because the columns of $A$ sum to zero and all compartments share the factor $g(U)$, $U'=\rho\,g(U)$ with $\rho=\sum_ir_iu_i\ge0$, so $U$ is monotone and $\int_0^\infty\abs{U'}\dd t<\infty$.
\item Because all compartments share the same factor $g(U)$, the vector $F$ has one sign, and so its size equals $\abs{\sum_iF_i}=\abs{U'}$: the reaction term is integrable in time.
\item The deviation $d=u-U\pi$ from the stationary composition $\pi$ solves $d'=Ad+(F-U'\pi)$, a linear equation forced by a zero-sum term of size at most $2\abs{U'}$.
\item The switching dynamics forgets its initial condition: $e^{tA}$ contracts zero-sum vectors exponentially. Hence $d(t)\to0$.
\item Persistent proliferation in some compartment, fed by switching, forces $U\to U^*$.
\end{enumerate}

In the continuum, every ingredient has a counterpart (Table~\ref{tab:dictionary}). The matrix exponential $e^{tA}$ becomes the solution operator $S(t)$ of the linear equation \eqref{eq:linear}, represented by its Green function $\Gamma(x,y,t)$ (the density at $x$ at time $t$ of cells that started at $y$). The fact that the columns of $A$ sum to zero becomes the divergence form of $\mathcal L$ together with the no-flux condition, and it implies $\int\Gamma(x,y,t)\dd x=1$, the counterpart of ``the columns of $e^{tA}$ sum to one''. Positivity of all the entries of $e^{\tau A}$ (a consequence of irreducibility) becomes strict positivity of $\Gamma$ (a consequence of $D>0$ and of the strong maximum principle). The variation-of-constants formula becomes Duhamel's formula, which we use to \emph{define} solutions of \eqref{eq:model}. With these substitutions, steps (i)--(v) go through almost verbatim.

\begin{table}[h]
\centering
\small
\begin{tabular}{@{}lll@{}}
\toprule
 & Compartmental model & Phenotype-structured model \\
\midrule
state & $u\in\R^n$, $U=\sum_iu_i$ & $u\in L^1(\Omega)$, $U=\int_\Omega u\dd x$\\
size & $\abs{u}_1=\sum_i\abs{u_i}$ & $\norm u=\int_\Omega\abs u\dd x$\\
switching generator & transition matrix $A$ & $\mathcal Lu=-(vu)_x+(Du_x)_x$, no flux\\
no creation of cells & columns of $A$ sum to $0$ & divergence form $+$ no-flux condition\\
linear solution operator & $e^{tA}$ & $S(t)$, kernel $\Gamma(x,y,t)$\\
mass conservation & columns of $e^{tA}$ sum to $1$ & $\int_\Omega\Gamma(x,y,t)\dd x=1$\\
connectivity & irreducibility of $A$ (H3) & $D>0$ on $[0,\ell]$ (H3)\\
consequence & all entries of $e^{\tau A}$ $\ge\epsilon>0$ & $\Gamma(x,y,\tau)\ge c>0$\\
mixing constant & $1-n\epsilon$ & $1-\ell c$\\
stationary state & $\pi$, $A\pi=0$ & $\psi\propto e^{\int v/D}$, $\mathcal L\psi=0$\\
persistent proliferation & one compartment $s$ & one sub-interval $J$\\
solution formula & variation of constants & Duhamel (mild solution)\\
\bottomrule
\end{tabular}
\caption{Dictionary between the compartmental proof of \cite{Review} and the proof given here.}\label{tab:dictionary}
\end{table}

\paragraph{The proof in a nutshell.} Formally, the whole argument fits in a few lines (Figure~\ref{fig:mechanism}). Integrating \eqref{eq:model} over $\Omega$, the switching term disappears (no flux), and
\[
U'=\rho\,g(U),\qquad\rho(t)=\int_\Omega r\,u\dd x\ \ge0 ,
\]
so $U$ moves monotonically towards $\Ustar$, and $\int_0^\infty\abs{U'}\dd t=\abs{U_\infty-U_0}<\infty$. The deviation $d=u-U\psi$ from the stationary profile has zero mass and, since $\mathcal L\psi=0$, it solves
\[
\partial_td=\mathcal Ld+\varphi,\qquad\varphi=g(U)\,r\,u-U'\psi,\qquad\int_\Omega\varphi\dd x=0,\qquad\norm{\varphi}_{L^1}\le2\abs{U'} ,
\]
where the bound holds because $g(U)ru$ has one sign, so that its $L^1$ norm is $\abs{\int g(U)ru}=\abs{U'}$. If the switching dynamics contracts zero-mass functions, $\norm{e^{t\mathcal L}w}_{L^1}\le Ce^{-\lambda t}\norm w_{L^1}$, Duhamel's formula gives
\[
\norm{d(t)}_{L^1}\le Ce^{-\lambda t}\norm{d(0)}_{L^1}+2C\int_0^te^{-\lambda(t-s)}\abs{U'(s)}\dd s\ \longrightarrow\ 0,
\]
because $\abs{U'}$ is integrable. Finally, if $U$ stopped short of $\Ustar$, switching would keep feeding the proliferating phenotypes and $U$ would keep growing; so $U\to\Ustar$ and $u\to\Ustar\psi$. What is new in the continuum is technical, not conceptual: the rest of the paper gives a meaning to solutions of \eqref{eq:model} for which Duhamel's formula holds (mild solutions, Section~\ref{sec:nonlinear}), establishes the contraction of zero-mass functions, which was automatic for matrices (Section~\ref{sec:linear}), or a substitute for it on the real line (Section~\ref{sec:R}), and identifies what each step really needs (Section~\ref{sec:general}). We state the results from parabolic PDE theory that we use as black boxes, with references, and derive everything else.

\begin{figure}[tbp]
\centering
\includegraphics[width=\textwidth]{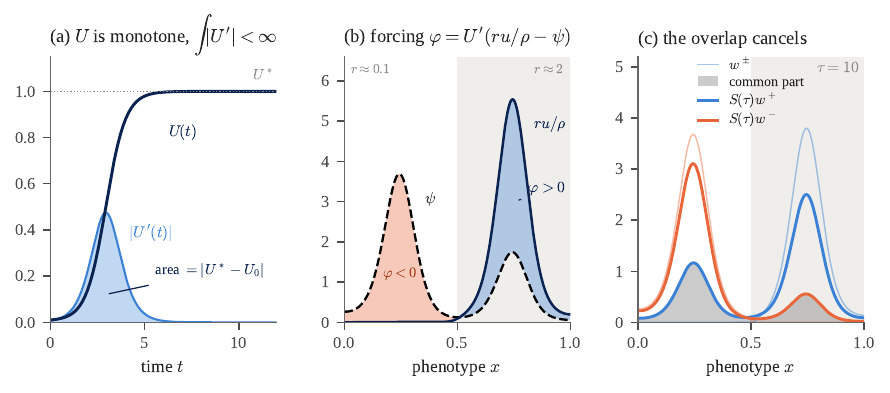}
\caption{The mechanism of the proof, on the example of Section~\ref{subsec:numerics}. (a) Under uniform competition the total population $U$ is monotone, and the total amount of reaction, $\int_0^\infty\abs{U'}\dd t$, equals $\abs{\Ustar-U_0}$. (b) At the time of fastest growth ($t\approx2.9$), new cells are produced with density $ru/\rho$ (solid), biased towards the fast-proliferating phenotypes (shaded). Relative to the stationary density $\psi$ (dashed), the reaction term is the zero-mass forcing $\varphi=U'(ru/\rho-\psi)$, of size at most $2\abs{U'}$. (c) The switching dynamics erases zero-mass perturbations. After a time $\tau$, the evolved positive and negative parts of $w=ru/\rho-\psi$ overlap, and $\norm{S(\tau)w}=\norm w-2\times(\text{grey area})$; here $\norm{S(\tau)w}\approx0.52\,\norm w$ for $\tau=10$. Doeblin's argument (Lemma~\ref{lem:mix}) bounds the grey area from below.}\label{fig:mechanism}
\end{figure}

\paragraph{Relation to other approaches.} Since $\int_0^\infty\abs{U'}\dd t<\infty$, the model \eqref{eq:model} is an integrable perturbation of the linear switching dynamics, and the theory of asymptotically autonomous systems \cite{Markus1956,StraussYorke1967,Thieme1992,MischaikowSmithThieme1995} gives conditions under which such systems inherit the asymptotic behaviour of their limit. That theory, however, needs to know beforehand that the perturbation vanishes, here that $g(U(t))\to0$, and, in infinite dimensions, compactness of the orbits, and it gives no rate. In our argument the integrability of the perturbation follows from the sign structure of the reaction term alone, before anything is known about the limit of $U$, and Duhamel's formula turns it into the explicit estimate \eqref{eq:star-intro}. The other ingredient, the contraction of zero-mass perturbations by the switching dynamics, could also be approached through the general relative entropy inequality of Michel, Mischler and Perthame \cite{MichelMischlerPerthame2005}, the standard tool for the long-time behaviour of linear structured-population equations; we use instead Doeblin's argument and the spectral gap, which give explicit rates, and, on the real line, a Dirichlet-form construction.

\subsection{Organisation}

Section~\ref{sec:setting} states the hypotheses and main results. Section~\ref{sec:tools} recalls Markov semigroups, the continuum analogue of transition matrices. Section~\ref{sec:linear} studies the linear equation on a bounded interval, and proves the mixing estimate. Section~\ref{sec:nonlinear} defines mild solutions of \eqref{eq:model}, and shows that they exist globally, are nonnegative, and that their total population obeys $U'=\rho\,g(U)$. Section~\ref{sec:proofA} proves Theorem~\ref{thm:A}, following the steps of the compartmental proof, and illustrates it numerically. Section~\ref{sec:R} treats the real line, including rates in a weighted norm. Section~\ref{sec:general} states the general result behind these proofs and applies it to nonlocal switching and to bounded domains in $\R^d$, and Section~\ref{sec:disc} discusses the hypotheses and their biological meaning. Routine proofs about integrals of $L^1$-valued functions and about mild solutions are collected in Appendices~\ref{app:integrals} and~\ref{app:mild}.

\emph{A reading guide.} The mechanism is contained in the nutshell above, in the mixing estimate (Section~\ref{subsec:mixing}) and in Section~\ref{sec:proofA}. A reader willing to accept that \eqref{eq:model} has a unique global nonnegative solution can skip Section~\ref{sec:nonlinear} except for Proposition~\ref{prop:global}. On the real line, the construction of Section~\ref{subsec:construction} can be taken as a black box that provides the asymptotic stability \eqref{eq:E}.

\section{Setting and main results}\label{sec:setting}

\subsection{Bounded phenotype interval}

Let $\Omega=(0,\ell)$ with $\ell>0$ and write $\norm f=\int_\Omega\abs{f}\dd x$ for the norm of $\Lone$. We make four assumptions, labelled as their compartmental counterparts in \cite{Review}.

\begin{enumerate}[label=\textbf{(H\arabic*)}, leftmargin=3em, itemsep=4pt]
\item \emph{Common modulation (uniform competition).} The model has the form \eqref{eq:model}: all phenotypes are regulated by the same factor $g(U)$, which depends only on the total population $U$.
\item \emph{Stable carrying capacity.} $g$ is continuously differentiable on $(0,\infty)$ and has a unique zero $\Ustar>0$, with $g>0$ on $(0,\Ustar)$ and $g<0$ on $(\Ustar,\infty)$.
\item \emph{Non-degenerate switching.} $D\in C^2([0,\ell])$ with $D(x)>0$ for all $x\in[0,\ell]$, and $v\in C^1([0,\ell])$. The no-flux boundary condition
\begin{equation}\label{eq:noflux}
D\,\partial_xu-v\,u=0\qquad\text{at }x=0\text{ and }x=\ell
\end{equation}
is imposed.
\item \emph{Bounded and persistent proliferation.} $r:[0,\ell]\times[0,\infty)\to[0,\rmax]$ is continuous, and there are a sub-interval $J\subset\Omega$ of positive length, $\rmin>0$ and $t_r\ge0$ such that $r(x,t)\ge\rmin$ for all $x\in J$ and $t\ge t_r$.
\end{enumerate}

Hypothesis (H3) is the continuum counterpart of connected plasticity: diffusion acts everywhere in phenotype space, so every phenotype can be reached from every other. The smoothness requirements on $D$ and $v$ are made so that classical parabolic theory can be quoted (Theorem~\ref{bb:green}); rough coefficients are discussed in Section~\ref{sec:disc}. Hypothesis (H4) allows quiescent phenotypes ($r=0$ on parts of $\Omega$) and time-dependent proliferation; it only asks that, from some time $t_r$ on, proliferation does not stop on a whole range $J$ of phenotypes. In particular, $r$ may be arbitrary (for instance, reduced by a treatment) on the finite time interval $[0,t_r)$.

The \emph{stationary density} of the switching dynamics is
\begin{equation}\label{eq:psi}
\psi(x)=\frac1Z\exp\Bigl(\int_0^x\frac{v(y)}{D(y)}\dd y\Bigr),\qquad Z=\int_0^\ell\exp\Bigl(\int_0^x\frac{v(y)}{D(y)}\dd y\Bigr)\dd x .
\end{equation}
It is a probability density, strictly positive on $[0,\ell]$, and it satisfies $D\psi'=v\psi$, so that the flux $D\psi'-v\psi$ vanishes identically; hence $\mathcal L\psi=0$ and $\psi$ satisfies \eqref{eq:noflux}.

Solutions of \eqref{eq:model} are understood in the \emph{mild} (Duhamel) sense of Definition~\ref{def:mild}, which is the natural analogue of the variation-of-constants formula; for smooth data, mild solutions are classical (Remark~\ref{rem:classical}). Below, $S(t)$ denotes the solution operator of the linear equation \eqref{eq:linear} with no-flux conditions, constructed in Section~\ref{sec:linear}.

\begin{mainthm}[Bounded phenotype space]\label{thm:A}
Assume \textup{(H1)--(H4)}. Let $u_0\in\Lone$ with $u_0\ge0$ and $U_0:=\int_\Omega u_0\dd x>0$. Then:
\begin{enumerate}[label=(\alph*), itemsep=3pt]
\item \eqref{eq:model} has a unique global mild solution $u\in C([0,\infty);\Lone)$ with $u(0)=u_0$; it is nonnegative.
\item The total population $U(t)$ is continuously differentiable and monotone, and $U(t)\to\Ustar$ as $t\to\infty$.
\item There are constants $C\ge1$ and $\lambda>0$, depending only on $\ell$, $D$ and $v$, such that for all $t\ge0$
\begin{equation}\label{eq:star}
\bigl\|u(t)-U(t)\psi\bigr\|\le Ce^{-\lambda t}\bigl\|u_0-U_0\psi\bigr\|+2C\int_0^te^{-\lambda(t-s)}\abs{U'(s)}\dd s ,
\end{equation}
and the right-hand side tends to zero. Hence $u(t)\to\Ustar\psi$ in $\Lone$. One can take $\lambda=\lambda_1$, the spectral gap of the switching dynamics (Lemma~\ref{lem:gap}), and no larger $\lambda$ is possible in general (Remark~\ref{rem:constants}).
\item (\emph{Asymptotic equivalence}) For every $f\in\Lone$ with $\int_\Omega f\dd x=\Ustar$, the solution $u^*(t)=S(t)f$ of the linear equation satisfies $\norm{u(t)-u^*(t)}\to0$.
\end{enumerate}
\end{mainthm}

In words: the long-term phenotypic distribution $\Ustar\psi$ depends on the switching dynamics ($D$ and $v$) and on the carrying capacity $\Ustar$, but not on the proliferation rates $r$, nor on the initial condition. Estimate \eqref{eq:star} says how fitness is forgotten: differences in proliferation distort the phenotypic composition only while the total population changes ($U'\ne0$), and each distortion is then erased by the switching dynamics at rate $\lambda$, which can be taken to be the spectral gap $\lambda_1$ (see Figure~\ref{fig:numerics}).

\begin{maincor}[Epigenetic landscapes]\label{cor:landscape}
If $v=-P'$ for a potential $P\in C^2([0,\ell])$, the limit in Theorem~\ref{thm:A} is
\[
\Ustar\psi(x)=\frac{\Ustar}{Z}\exp\Bigl(-\int_0^x\frac{P'(y)}{D(y)}\dd y\Bigr),
\]
and for constant $D$ this is the Boltzmann-type density $\Ustar Z^{-1}e^{-P(x)/D}$ (with a different normalising constant $Z$).
\end{maincor}

\begin{maincor}[Exponential rate]\label{cor:rate}
If, in addition, $g'(\Ustar)<0$, then there are $\nu>0$, depending only on $\ell$, $D$, $v$, $g$, $\rmin$ and $J$, and $K>0$, depending also on $u_0$ and $r$, such that
\[
\norm{u(t)-\Ustar\psi}\le K(1+t)e^{-\nu t}\qquad\text{for all }t\ge0 .
\]
\end{maincor}

The hypothesis $g'(\Ustar)<0$ cannot be dropped. For $g(U)=(1-U)^3$, which satisfies (H2) with $\Ustar=1$, and $U_0<1$, the difference $z=1-U$ satisfies $z'=-\rho\,z^3$, where $\rho$ is bounded above, and, for large $t$, bounded below by a positive constant (Step~5(B) in Section~\ref{sec:proofA}). Hence $(1/z^2)'=2\rho$ is bounded above and below by positive constants for large $t$, $z(t)$ decays like $t^{-1/2}$, and $\norm{u(t)-\Ustar\psi}\ge\abs{U(t)-\Ustar}$ decays no faster.

\subsection{Whole real line}

On $\Omega=\R$ we consider the same switching operator $\mathcal Lu=-\partial_x(vu)+\partial_x(D\partial_xu)$, now with coefficients defined on the whole line, and we write again
\begin{equation}\label{eq:psiR}
\psi(x)=\frac1Z\exp\Bigl(\int_0^x\frac{v(y)}{D(y)}\dd y\Bigr),\qquad Z=\int_\R\exp\Bigl(\int_0^x\frac{v(y)}{D(y)}\dd y\Bigr)\dd x .
\end{equation}
Hypothesis (H3) is replaced by the following three conditions.

\begin{enumerate}[label=\textbf{(R\arabic*)}, leftmargin=3em, itemsep=4pt]
\item \emph{Local regularity and non-degeneracy.} $D\in C^1(\R)$ with $D(x)>0$ for all $x$, and $v\in C(\R)$.
\item \emph{Confinement.} $Z<\infty$, so that $\psi$ is a probability density on $\R$.
\item \emph{No escape to infinity.} $\displaystyle\int_0^{+\infty}\frac{\dd x}{D(x)\psi(x)}=\int_{-\infty}^0\frac{\dd x}{D(x)\psi(x)}=+\infty .$
\end{enumerate}

No global bound on $D$ or $v$ is required. Condition (R2) says that the landscape holds the population together. Condition (R3) says that no cell reaches infinitely distant phenotypes in finite time; without it, mass could leak through infinity, and the identity $U'=\rho g(U)$, on which everything rests, would fail. In the language of one-dimensional diffusions, and in the presence of (R2), (R3) says that $\pm\infty$ are not accessible (Feller's test for explosions). Condition (R3) holds, for instance, whenever $D$ is bounded (Lemma~\ref{lem:R3}). Section~\ref{sec:R} constructs the solution operator $S(t)$ of the linear equation under (R1)--(R3), and gives examples: the Ornstein--Uhlenbeck process, multi-well potentials growing at infinity, and state-dependent noise with heavy-tailed stationary densities.

\begin{mainthm}[Whole real line]\label{thm:B}
Let $\Omega=\R$ and assume \textup{(H1)}, \textup{(H2)} and \textup{(R1)--(R3)}. Let $r:\R\times[0,\infty)\to[0,\rmax]$ be continuous, with $r(x,t)\ge\rmin>0$ for all $x$ in some bounded interval $J$ of positive length and all $t\ge t_r$, for some $t_r\ge0$. Let $u_0\in L^1(\R)$ with $u_0\ge0$ and $U_0>0$. Then:
\begin{enumerate}[label=(\alph*), itemsep=2pt]
\item \eqref{eq:model} has a unique global mild solution, which is nonnegative;
\item $U(t)$ is continuously differentiable and monotone, and $U(t)\to\Ustar$;
\item $\norm{u(t)-\Ustar\psi}_{L^1(\R)}\to0$ as $t\to\infty$;
\item $\norm{u(t)-S(t)f}_{L^1(\R)}\to0$ for every $f\in L^1(\R)$ with $\int f=\Ustar$.
\end{enumerate}
\end{mainthm}

As in Corollary~\ref{cor:landscape}, when $v=-P'$ the limit is $\Ustar Z^{-1}\exp\bigl(-\int_0^xP'/D\bigr)$. Theorem~\ref{thm:B} claims no rate. For the Ornstein--Uhlenbeck dynamics this is not a weakness of the method: the linear dynamics itself has no uniform rate of convergence in $L^1(\R)$ (Proposition~\ref{prop:OUnorate}). A rate does hold, however, in the weighted norm
\[
\norm f_\psi:=\Bigl(\int_\R\frac{f^2}{\psi}\dd x\Bigr)^{1/2}\ \ge\ \norm f_{L^1(\R)},
\]
for initial data with $\norm{u_0}_\psi<\infty$, as soon as the stationary density satisfies a Poincar\'e inequality and, as in Corollary~\ref{cor:rate}, $g'(\Ustar)<0$.

\begin{maincor}[Exponential rate on the real line]\label{cor:rateR}
In the setting of Theorem~\ref{thm:B}, assume in addition that $g'(\Ustar)<0$, that $\norm{u_0}_\psi<\infty$, and that there is $\lambda_P>0$ such that
\begin{equation}\label{eq:poincare}
\int_\R\Bigl(w-\int_\R w\,\psi\dd y\Bigr)^2\psi\dd x\le\frac1{\lambda_P}\int_\R D\psi\,(w')^2\dd x\qquad\text{for all }w\in C_c^\infty(\R).
\end{equation}
Then there are $\nu>0$ and $K>0$ such that
\[
\norm{u(t)-\Ustar\psi}_{L^1(\R)}\le\norm{u(t)-\Ustar\psi}_\psi\le K(1+t)e^{-\nu t}\qquad\text{for all }t\ge0 .
\]
\end{maincor}

The Poincar\'e inequality \eqref{eq:poincare} holds for the Ornstein--Uhlenbeck dynamics, for multi-well potentials with constant $D$ and $\liminf_{x\to\pm\infty}\pm P'(x)>0$, and for the heavy-tailed densities produced by state-dependent noise (Section~\ref{subsec:examples}). Since $K$ depends on $\norm{u_0}_\psi$, the rate is not uniform in $L^1(\R)$, in agreement with Proposition~\ref{prop:OUnorate}.

\section{Markov semigroups}\label{sec:tools}

The solution operator $e^{tA}$ of a compartmental switching system has nonnegative entries and columns summing to one. Its continuum counterpart is a Markov semigroup, which we now define; all the properties of the switching dynamics used in the proofs are properties of this object. Integrals of $L^1$-valued functions of time, such as Duhamel's integral $\int_0^tS(t-s)F(s)\dd s$, are Riemann integrals of continuous functions, and the elementary facts we use about them are collected in Appendix~\ref{app:integrals}.

\begin{definition}[Markov semigroup]\label{def:markov}
A family $\{S(t)\}_{t\ge0}$ of bounded linear operators on $\Lone$ is a \emph{Markov semigroup} if:
\begin{enumerate}[label=\textbf{(S\arabic*)}, leftmargin=3em, itemsep=2pt]
\item (positivity) $f\ge0\ \Rightarrow\ S(t)f\ge0$;
\item (mass conservation) $\int_\Omega S(t)f\dd x=\int_\Omega f\dd x$ for all $f\in\Lone$;
\item (semigroup) $S(0)=I$ and $S(t+s)=S(t)S(s)$;
\item (strong continuity at $0$) $\norm{S(t)f-f}\to0$ as $t\downarrow0$, for every $f\in\Lone$.
\end{enumerate}
\end{definition}

For the matrix exponential $e^{tA}$ of a transition-rate matrix these are, respectively: nonnegative entries, columns summing to one, $e^{(t+s)A}=e^{tA}e^{sA}$, and continuity at $t=0$.

\begin{lemma}[Consequences of (S1)--(S4)]\label{lem:markov}
Let $S$ be a Markov semigroup. Then:
\begin{enumerate}[label=(\roman*), itemsep=2pt]
\item $\norm{S(t)f}\le\norm f$ for all $f$, with equality if $f\ge0$ (\emph{contraction}).
\item If $\int_\Omega f=0$ then $\int_\Omega S(t)f=0$: the zero-mass subspace $\Lzero:=\{f\in\Lone:\int_\Omega f\dd x=0\}$ is invariant.
\item $(t,f)\mapsto S(t)f$ is jointly continuous from $[0,\infty)\times\Lone$ to $\Lone$.
\end{enumerate}
\end{lemma}

\begin{proof}
(i) For $f\ge0$, $\norm{S(t)f}=\int S(t)f=\int f=\norm f$ by (S1)--(S2). For general $f$, write $f=f^+-f^-$ with $f^\pm\ge0$ and $\abs f=f^++f^-$. By (S1), $S(t)f^\pm\ge0$, so pointwise a.e.
\[
\abs{S(t)f}=\abs{S(t)f^+-S(t)f^-}\le S(t)f^++S(t)f^-=S(t)\abs f ,
\]
and integrating, $\norm{S(t)f}\le\int S(t)\abs f=\int\abs f=\norm f$. This is the continuum version of the computation $\abs{e^{tA}v}_1\le\sum_{i,j}(e^{tA})_{ij}\abs{v_j}=\abs v_1$ for compartmental models.

(ii) is (S2).

(iii) First, $t\mapsto S(t)f$ is continuous for each $f$. Right continuity: $S(t+h)f-S(t)f=S(t)(S(h)f-f)$ has norm at most $\norm{S(h)f-f}\to0$, by (i) and (S4). Left continuity: for $0<h\le t$, $S(t-h)f-S(t)f=S(t-h)(f-S(h)f)$, with the same bound. Then
\[
\norm{S(t)f-S(t_0)f_0}\le\norm{S(t)(f-f_0)}+\norm{S(t)f_0-S(t_0)f_0}\le\norm{f-f_0}+\norm{S(t)f_0-S(t_0)f_0},
\]
which tends to zero as $(t,f)\to(t_0,f_0)$.
\end{proof}

\section{The linear switching dynamics on a bounded interval}\label{sec:linear}

Throughout this section, (H3) holds and $\Omega=(0,\ell)$.

\subsection{What we use from parabolic theory}

\begin{blackbox}[Green function of the no-flux problem]\label{bb:green}
Assume \textup{(H3)}. There is a continuous function $\Gamma:[0,\ell]\times[0,\ell]\times(0,\infty)\to\R$ with the following properties.
\begin{enumerate}[label=(\alph*), itemsep=3pt]
\item \emph{(Existence and uniqueness of classical solutions.)} For every $f\in C([0,\ell])$, the function
\[
u(x,t)=\int_0^\ell\Gamma(x,y,t)f(y)\dd y\quad(t>0),\qquad u(x,0)=f(x),
\]
is continuous on $[0,\ell]\times[0,\infty)$; $\partial_tu$, $\partial_xu$, $\partial_{xx}u$ exist and are continuous on $(0,\ell)\times(0,\infty)$; $\partial_xu$ extends continuously to $[0,\ell]\times(0,\infty)$; $u$ satisfies $\partial_tu=\mathcal Lu$ in $(0,\ell)\times(0,\infty)$ and the no-flux condition \eqref{eq:noflux} for $t>0$. Moreover, it is the only function with these properties.
\item \emph{(Strict positivity.)} $\Gamma(x,y,t)>0$ for all $x,y\in[0,\ell]$ and $t>0$.
\item \emph{(Chapman--Kolmogorov.)} $\Gamma(x,y,t+s)=\int_0^\ell\Gamma(x,z,t)\,\Gamma(z,y,s)\dd z$ for all $x,y$ and $t,s>0$.
\end{enumerate}
\end{blackbox}

\noindent\emph{Comments and references.} $\Gamma(x,y,t)$ is the density at phenotype $x$ and time $t$ of the population that started as a unit mass at phenotype $y$. Since $v$ is only assumed to be $C^1$, the zeroth-order coefficient of $\mathcal L$ written in non-divergence form, $\mathcal Lu=D\,\partial_{xx}u+(D'-v)\,\partial_xu-v'u$, is merely continuous, and classical Schauder theory, which requires H\"older continuous coefficients, does not apply directly to \eqref{eq:linear}. It applies after the substitution $u=\psi w$, where $\psi\in C^2([0,\ell])$ is the strictly positive stationary density \eqref{eq:psi} (see Lemma~\ref{lem:psi}). Indeed, $D\psi'=v\psi$ gives
\begin{equation}\label{eq:flux-w}
D\,\partial_xu-v\,u=D\psi\,\partial_xw ,
\end{equation}
so that $u$ solves \eqref{eq:linear}--\eqref{eq:noflux} if and only if $w$ solves the Neumann problem
\[
\partial_tw=D\,\partial_{xx}w+(D'+v)\,\partial_xw\quad\text{in }(0,\ell)\times(0,\infty),\qquad\partial_xw=0\quad\text{at }x=0,\ell ,
\]
whose coefficients $D\in C^2$ and $D'+v\in C^1$ are H\"older continuous and which has no zeroth-order term; and, since $\psi\in C^2$, $w$ has the regularity listed in (a) if and only if $u=\psi w$ has it. This Neumann problem is the second initial-boundary value problem of Friedman \cite[Ch.~2, Sec.~5]{Friedman1964}. Existence of classical solutions is proved with single-layer potentials in \cite[Ch.~5, Sec.~3]{Friedman1964}; for smooth initial data that satisfy the boundary condition, existence and uniqueness in H\"older classes is \cite[Ch.~IV, Thm.~5.3]{LSU1968}; and the Green (Neumann) function $\Gamma_N$ of such problems was constructed by It\^o \cite{Ito1957} (see \cite[p.~336]{Friedman1964}). Uniqueness in the class described in (a) follows from the maximum principle \cite[Ch.~2, Thm.~15]{Friedman1964}.
Then $\Gamma(x,y,t)=\psi(x)\Gamma_N(x,y,t)/\psi(y)$. Positivity (b) follows from the \emph{strong maximum principle} in the interior \cite[Ch.~2, Thm.~5]{Friedman1964} and the \emph{boundary point lemma} of Hopf type at the endpoints \cite[Ch.~2, Thm.~14]{Friedman1964}, applied to $w$ (see also \cite{ProtterWeinberger1967}): a nonnegative, nonzero solution cannot vanish at an interior point at a positive time, and it cannot vanish at an endpoint either, because the boundary point lemma, applied to $1-w$, would then give $\partial_xw\ne0$ there, contradicting the Neumann condition. This applies to $\Gamma$ itself: $\Gamma\ge0$, since solutions with nonnegative data are nonnegative, and, for $0<s<t$, $\Gamma(\cdot,y,t)$ is by (c) the value at time $t-s$ of the solution with the continuous initial datum $\Gamma(\cdot,y,s)\ge0$, which is nonzero because it has mass $1$ (Lemma~\ref{lem:mass-lin}); hence $\Gamma(x,y,t)>0$ for all $x$. Property (c) is a consequence of the uniqueness in (a): solving from time $0$ to $s$ and then from $s$ to $t+s$ gives the same result as solving from $0$ to $t+s$. In the variable $w$, \eqref{eq:linear} is the regular Sturm--Liouville problem $\partial_tw=\psi^{-1}\partial_x(D\psi\,\partial_xw)$ with Neumann conditions, which also gives $\Gamma$ as an eigenfunction expansion; its spectral gap is the optimal mixing rate (Lemma~\ref{lem:gap}).

\begin{remark}[Continuum analogue of irreducibility]
Property (b) is the counterpart of Lemma~A.1(ii) of \cite{Review}, namely: for an irreducible transition matrix $A$, \emph{all entries of $e^{\tau A}$ are positive}. There, positivity came from the existence of chains of transitions between any two compartments; here it comes from $D>0$ everywhere, which lets diffusion connect any two phenotypes in any positive time. If $D$ vanished on a sub-interval, cells could not cross it by diffusion, and property (b) could fail.
\end{remark}

\subsection{The solution operator}

For $f\in\Lone$ and $t>0$ define
\begin{equation}\label{eq:S}
\bigl(S(t)f\bigr)(x)=\int_0^\ell\Gamma(x,y,t)f(y)\dd y,\qquad S(0)f=f .
\end{equation}
Since $\Gamma(\cdot,\cdot,t)$ is continuous on the compact square $[0,\ell]^2$, it is bounded, so the integral is finite for every $f\in\Lone$, and $S(t)f$ is a continuous function of $x$.

The next lemma is the continuum counterpart of the ``preliminary fact'' of \cite[Appendix~A]{Review}: there, the columns of $A$ sum to zero because every cell that leaves a compartment enters another one, and consequently the columns of $e^{tA}$ sum to one. Here, the divergence form of $\mathcal L$ together with the no-flux condition play the same role.

\begin{lemma}[Mass conservation]\label{lem:mass-lin}
$\displaystyle\int_0^\ell\Gamma(x,y,t)\dd x=1$ for all $y\in[0,\ell]$ and $t>0$.
\end{lemma}

\begin{proof}
Let $f\in C([0,\ell])$ and let $u$ be the classical solution of Theorem~\ref{bb:green}(a). Write $J_{\rm flux}(x,t)=D(x)\partial_xu(x,t)-v(x)u(x,t)$ for the flux, so that $\partial_tu=\partial_xJ_{\rm flux}$; by Theorem~\ref{bb:green}(a), $J_{\rm flux}$ is continuous on $[0,\ell]\times(0,\infty)$ and vanishes at $x=0,\ell$. Fix $0<t_1<t_2$ and $0<a<b<\ell$. Since $u$ is smooth on $[a,b]\times[t_1,t_2]$, integrating $\partial_tu=\partial_xJ_{\rm flux}$ over this rectangle gives
\[
\int_a^b\bigl(u(x,t_2)-u(x,t_1)\bigr)\dd x=\int_{t_1}^{t_2}\bigl(J_{\rm flux}(b,t)-J_{\rm flux}(a,t)\bigr)\dd t .
\]
Let $a\downarrow0$ and $b\uparrow\ell$. The left-hand side tends to $\int_0^\ell(u(x,t_2)-u(x,t_1))\dd x$, and the right-hand side tends to $\int_{t_1}^{t_2}(J_{\rm flux}(\ell,t)-J_{\rm flux}(0,t))\dd t=0$ by continuity of $J_{\rm flux}$ up to the boundary and by the no-flux condition. Hence $\int_0^\ell u(x,t)\dd x$ is constant for $t>0$; since $u$ is uniformly continuous on $[0,\ell]\times[0,1]$, it tends to $\int_0^\ell f\dd x$ as $t\downarrow0$. Therefore, for every $t>0$,
\[
\int_0^\ell f(y)\Bigl(\int_0^\ell\Gamma(x,y,t)\dd x\Bigr)\dd y=\int_0^\ell u(x,t)\dd x=\int_0^\ell f(y)\dd y,
\]
where we used Fubini's theorem ($\Gamma$ is bounded). Since this holds for every continuous $f$, and $y\mapsto\int_0^\ell\Gamma(x,y,t)\dd x$ is continuous, this function is identically $1$.
\end{proof}

\begin{proposition}[$S$ is a Markov semigroup]\label{prop:S-markov}
Under \textup{(H3)}, the operators \eqref{eq:S} form a Markov semigroup on $\Lone$. For $f\in C([0,\ell])$, $u(x,t)=(S(t)f)(x)$ is the classical solution of \eqref{eq:linear}--\eqref{eq:noflux} with initial datum $f$.
\end{proposition}

\begin{proof}
(S1) follows from $\Gamma>0$. (S2): by Fubini and Lemma~\ref{lem:mass-lin},
\[
\int_\Omega S(t)f\dd x=\int_\Omega f(y)\Bigl(\int_\Omega\Gamma(x,y,t)\dd x\Bigr)\dd y=\int_\Omega f\dd y .
\] (S3): by Chapman--Kolmogorov and Fubini, $S(t)S(s)f=S(t+s)f$ for $t,s>0$. (S4): by (S1)--(S2), each $S(t)$ is an $L^1$ contraction (Lemma~\ref{lem:markov}(i), whose proof uses only (S1)--(S2)). For continuous $f$, $S(t)f\to f$ uniformly on $[0,\ell]$ as $t\downarrow0$, by uniform continuity of the classical solution on $[0,\ell]\times[0,1]$; hence also in $\Lone$. For $f\in\Lone$, choose $f_k$ continuous with $\norm{f-f_k}\to0$; then
\[
\norm{S(t)f-f}\le\norm{S(t)(f-f_k)}+\norm{S(t)f_k-f_k}+\norm{f_k-f}\le2\norm{f-f_k}+\norm{S(t)f_k-f_k},
\]
so $\limsup_{t\downarrow0}\norm{S(t)f-f}\le2\norm{f-f_k}$ for every $k$, which gives (S4).
\end{proof}

\subsection{The stationary density}

\begin{lemma}[Invariance of $\psi$]\label{lem:psi}
The density $\psi$ of \eqref{eq:psi} belongs to $C^2([0,\ell])$, is strictly positive, and satisfies $S(t)\psi=\psi$ for all $t\ge0$. It is the only probability density with this property.
\end{lemma}

\begin{proof}
From \eqref{eq:psi}, $\psi'=(v/D)\psi$. Since $v/D\in C^1([0,\ell])$ by (H3), $\psi\in C^1$ and then $\psi'\in C^1$, so $\psi\in C^2$. The flux $D\psi'-v\psi$ is identically zero, so $\mathcal L\psi=\partial_x(D\psi'-v\psi)=0$, and \eqref{eq:noflux} holds. Hence the time-independent function $u(x,t)=\psi(x)$ has all the properties listed in Theorem~\ref{bb:green}(a) with $f=\psi$, and by the uniqueness stated there, $S(t)\psi=\psi$. Uniqueness among probability densities is proved in Lemma~\ref{lem:mix}(iii) below.
\end{proof}

This is the analogue of the stationary distribution $\pi$ of the compartmental model, with one difference: in the continuum, $\psi$ is given by an explicit formula, because in one dimension the stationarity condition $\mathcal L\psi=0$ with zero flux integrates to the first-order ODE $D\psi'=v\psi$.

\subsection{The switching dynamics forgets its initial state}\label{subsec:mixing}

We now prove the main linear estimate, the counterpart of \cite[Lemma~A.1(iii)]{Review}. The proof is Doeblin's argument, as in the compartmental case: after a time $\tau>0$, both the positive and the negative part of a zero-mass function have spread over the whole phenotype space, and their common part cancels.

\begin{lemma}[Exponential mixing: Doeblin's argument]\label{lem:mix}
Assume \textup{(H3)}, fix $\tau>0$, and let $c:=\min_{x,y\in[0,\ell]}\Gamma(x,y,\tau)$. Then $0<c\le1/\ell$, and with $\bar c:=\min\{c,\,1/(2\ell)\}$,
\begin{enumerate}[label=(\roman*), itemsep=3pt]
\item $\norm{S(\tau)w}\le(1-\ell\bar c)\norm w$ for every $w\in\Lzero$;
\item for all $t\ge0$ and $w\in\Lzero$,
\begin{equation}\label{eq:mix}
\norm{S(t)w}\le C\,e^{-\lambda t}\norm w,\qquad C:=\frac1{1-\ell\bar c},\quad\lambda:=-\frac1\tau\ln(1-\ell\bar c)>0 ;
\end{equation}
\item $\psi$ is the only probability density $\phi$ with $S(t)\phi=\phi$ for all $t$, and $S(t)f\to\bigl(\int_\Omega f\bigr)\psi$ in $\Lone$ for every $f\in\Lone$.
\end{enumerate}
\end{lemma}

\begin{proof}
\emph{The constant $c$.} $\Gamma(\cdot,\cdot,\tau)$ is continuous and strictly positive on the compact square $[0,\ell]^2$ (Theorem~\ref{bb:green}(b)), so its minimum $c$ is attained and positive. By Lemma~\ref{lem:mass-lin}, $1=\int_0^\ell\Gamma(x,y,\tau)\dd x\ge c\,\ell$, so $c\le1/\ell$. We use $\bar c\le c$, which satisfies $\ell\bar c<1$ strictly; replacing $c$ by the smaller $\bar c$ only weakens the lower bound $\Gamma\ge c$, so nothing is lost below. (This is the counterpart of $n\epsilon\le1$ in the compartmental case.)

(i) Let $w\in\Lzero$ and write $w=w^+-w^-$ with $w^\pm=\max\{\pm w,0\}\ge0$. Since $\int w=0$, the two parts carry the same mass:
\[
\int_\Omega w^+\dd x-\int_\Omega w^-\dd x=\int_\Omega w\dd x=0\quad\Longrightarrow\quad\int_\Omega w^+\dd x=\int_\Omega w^-\dd x=\tfrac12\norm w .
\]
Because $\Gamma(x,y,\tau)\ge\bar c$ for all $x,y$ and $w^+\ge0$,
\[
\bigl(S(\tau)w^+\bigr)(x)=\int_0^\ell\Gamma(x,y,\tau)w^+(y)\dd y\ \ge\ \bar c\int_0^\ell w^+(y)\dd y=\tfrac{\bar c}2\norm w\qquad\text{for every }x\in[0,\ell],
\]
and likewise for $w^-$. So the functions
\[
h_+:=S(\tau)w^+-\tfrac{\bar c}2\norm w,\qquad h_-:=S(\tau)w^--\tfrac{\bar c}2\norm w
\]
are nonnegative (we subtract, from each, a constant floor that both exceed at every phenotype). By mass conservation (S2) applied to $w^+$ and to $w^-$, their integrals are
\[
\int_\Omega h_+\dd x=\underbrace{\int_\Omega S(\tau)w^+\dd x}_{=\int w^+=\frac12\norm w}-\ \ell\cdot\tfrac{\bar c}2\norm w=\tfrac12\norm w\,(1-\ell\bar c),
\]
and the same for $h_-$. Since $h_\pm\ge0$, their integrals are their $L^1$ norms. Finally $S(\tau)w=S(\tau)w^+-S(\tau)w^-=h_+-h_-$ (the floors cancel), so
\[
\norm{S(\tau)w}=\norm{h_+-h_-}\le\norm{h_+}+\norm{h_-}=(1-\ell\bar c)\norm w .
\]

(ii) By Lemma~\ref{lem:markov}(ii), $S(\tau)w$ again has zero mass, so (i) can be applied to it, then to $S(\tau)^2w=S(2\tau)w$, and so on: $\norm{S(m\tau)w}\le(1-\ell\bar c)^m\norm w$ for every integer $m\ge0$. For $t\ge0$ write $t=m\tau+s$ with $m=\lfloor t/\tau\rfloor$ and $0\le s<\tau$; by the semigroup property and the contraction property (Lemma~\ref{lem:markov}(i), which needs no zero-mass condition) applied over the remaining time $s$,
\[
\norm{S(t)w}=\norm{S(s)S(m\tau)w}\le\norm{S(m\tau)w}\le(1-\ell\bar c)^m\norm w\le(1-\ell\bar c)^{t/\tau-1}\norm w=Ce^{-\lambda t}\norm w .
\]

(iii) If $\phi$ is a probability density with $S(t)\phi=\phi$, then $\phi-\psi\in\Lzero$ is fixed by every $S(t)$, so $\norm{\phi-\psi}=\norm{S(t)(\phi-\psi)}\le Ce^{-\lambda t}\norm{\phi-\psi}\to0$; hence $\phi=\psi$. For $f\in\Lone$, $f-(\int f)\psi\in\Lzero$ and $S(t)f-(\int f)\psi=S(t)\bigl(f-(\int f)\psi\bigr)$, whose norm tends to zero by (ii).
\end{proof}

Doeblin's argument needs only a lower bound on $\Gamma$ and gives explicit constants, but its rate is conservative (Remark~\ref{rem:constants}). The optimal rate is the spectral gap of the Sturm--Liouville problem mentioned after Theorem~\ref{bb:green}.

\begin{lemma}[Optimal rate: the spectral gap]\label{lem:gap}
Assume \textup{(H3)}, and let
\begin{equation}\label{eq:gap}
\lambda_1:=\inf\Bigl\{\frac{\int_0^\ell D\psi\,(w')^2\dd x}{\int_0^\ell w^2\,\psi\dd x}:\ w\in C^1([0,\ell]),\ \int_0^\ell w\,\psi\dd x=0,\ w\ne0\Bigr\}.
\end{equation}
Then $\lambda_1\ge\bigl(\int_0^\ell\frac{\dd x}{D\psi}\bigr)^{-1}>0$, and there is $C_1\ge1$, depending only on $\ell$, $D$ and $v$, such that
\begin{equation}\label{eq:mix-gap}
\norm{S(t)w}\le C_1e^{-\lambda_1t}\norm w\qquad\text{for all }t\ge0\text{ and }w\in\Lzero .
\end{equation}
The rate $\lambda_1$ is optimal: there is $w\in\Lzero$, $w\ne0$, with $S(t)w=e^{-\lambda_1t}w$ for all $t\ge0$.
\end{lemma}

\begin{proof}
\emph{Positivity of $\lambda_1$.} Let $w\in C^1([0,\ell])$ with $\int w\psi=0$. Since $\psi>0$, $w$ vanishes at some point $x_0$, and by the Cauchy--Schwarz inequality, for every $x\in[0,\ell]$,
\[
w(x)^2=\Bigl(\int_{x_0}^xw'\dd y\Bigr)^2\le\int_0^\ell D\psi\,(w')^2\dd y\,\int_0^\ell\frac{\dd y}{D\psi} .
\]
Multiplying by $\psi(x)$ and integrating over $[0,\ell]$ gives the lower bound on $\lambda_1$, since $\int\psi=1$.

\emph{Energy decay.} Let $f\in C([0,\ell])$ with $\int f=0$, let $u(x,t)=(S(t)f)(x)$ be the classical solution (Proposition~\ref{prop:S-markov}), and let $w=u/\psi$. By \eqref{eq:flux-w}, $\partial_tu=\partial_x(a\,\partial_xw)$ with $a:=D\psi$, and $\partial_xw=0$ at $x=0,\ell$ for $t>0$; moreover $w$ and $\partial_xw$ are continuous on $[0,\ell]\times(0,\infty)$. For $0<t_1<t_2$ and $0<\alpha<\beta<\ell$, integrate the identity
\[
\partial_t\bigl(\psi w^2\bigr)=2w\,\partial_x\bigl(a\,\partial_xw\bigr)=2\,\partial_x\bigl(w\,a\,\partial_xw\bigr)-2a\,(\partial_xw)^2
\]
over $[\alpha,\beta]\times[t_1,t_2]$, and let $\alpha\downarrow0$, $\beta\uparrow\ell$ as in the proof of Lemma~\ref{lem:mass-lin}; the boundary terms vanish by the Neumann condition, and
\[
E(t_2)-E(t_1)=-2\int_{t_1}^{t_2}\int_0^\ell a\,(\partial_xw)^2\dd x\dd t,\qquad E(t):=\int_0^\ell\psi w^2\dd x=\int_0^\ell\frac{u(x,t)^2}{\psi(x)}\dd x .
\]
The inner integral is continuous in $t$, so $E\in C^1((0,\infty))$ with $E'(t)=-2\int_0^\ell a\,(\partial_xw)^2\dd x$. Since $\int_0^\ell w(t)\psi\dd x=\int_0^\ell u(t)\dd x=0$ (Lemma~\ref{lem:mass-lin}) and $w(\cdot,t)\in C^1([0,\ell])$, the definition of $\lambda_1$ gives $E'\le-2\lambda_1E$, hence $E(t)\le e^{-2\lambda_1(t-s)}E(s)$ for $0<s\le t$.

\emph{From $L^2$ to $L^1$.} By the Cauchy--Schwarz inequality and $\int\psi=1$, $\norm g\le\bigl(\int_0^\ell g^2/\psi\dd x\bigr)^{1/2}$ for every $g$, so $\norm{S(t)f}\le E(t)^{1/2}$. Let $\Gamma^*:=\max_{x,y}\Gamma(x,y,1)$ and $\psi_{\min}:=\min\psi>0$. Then $\abs{u(x,1)}\le\Gamma^*\norm f$ for all $x$, so $E(1)\le\ell\,(\Gamma^*)^2\norm f^2/\psi_{\min}$, and for $t\ge1$
\[
\norm{S(t)f}\le e^{-\lambda_1(t-1)}E(1)^{1/2}\le e^{\lambda_1}\Gamma^*\sqrt{\ell/\psi_{\min}}\;e^{-\lambda_1t}\norm f ,
\]
while for $0\le t<1$, $\norm{S(t)f}\le\norm f\le e^{\lambda_1}e^{-\lambda_1t}\norm f$. This is \eqref{eq:mix-gap} for continuous $f$, with $C_1:=e^{\lambda_1}\max\{1,\Gamma^*\sqrt{\ell/\psi_{\min}}\}$. A general $w\in\Lzero$ is the $\Lone$-limit of continuous functions $f_k$, hence also of the continuous zero-mass functions $f_k-(\int f_k)\psi$, and \eqref{eq:mix-gap} passes to the limit because $S(t)$ is continuous on $\Lone$.

\emph{Optimality.} By classical Sturm--Liouville theory \cite{CoddingtonLevinson1955}, the infimum in \eqref{eq:gap} is the smallest positive eigenvalue of the Neumann problem $-\psi^{-1}(D\psi\,e')'=\lambda e$ on $(0,\ell)$, $e'(0)=e'(\ell)=0$, and it is attained at an eigenfunction $e_1\in C^2([0,\ell])$, which is orthogonal to the constants in $L^2(\psi\dd x)$: $\int e_1\psi=0$. By \eqref{eq:flux-w}, $u(x,t)=e^{-\lambda_1t}\psi(x)e_1(x)$ has the properties listed in Theorem~\ref{bb:green}(a), so $S(t)(\psi e_1)=e^{-\lambda_1t}\psi e_1$ by uniqueness, and $w:=\psi e_1\in\Lzero$.
\end{proof}

\begin{remark}[About the constants]\label{rem:constants}
Lemmas~\ref{lem:mix} and~\ref{lem:gap} both give estimates of the form \eqref{eq:mix}. Doeblin's argument uses only a lower bound on $\Gamma$, and it is the argument that carries over from compartmental models, but its rate can be far from optimal: in the example of Section~\ref{subsec:numerics}, $\lambda_1\approx0.066$, whereas Lemma~\ref{lem:mix} gives $\lambda\approx6\cdot10^{-6}$ for $\tau=1$, and about $2.5\cdot10^{-3}$ for the best choice of $\tau$ (near $\tau=10$). The rate $\lambda_1$ cannot be improved in \eqref{eq:star} either: for $u_0=\Ustar\psi+\varepsilon\psi e_1$, which is nonnegative for small $\varepsilon>0$, we have $U\equiv\Ustar$, the reaction term vanishes, and $\norm{u(t)-\Ustar\psi}=\varepsilon e^{-\lambda_1t}\norm{\psi e_1}$. For constant $D$ and $v=-P'$, $\lambda_1$ can be exponentially small in $\operatorname{osc}(P)/D$ when $P$ has several deep wells (the Kramers regime), which means that the time needed to forget the initial state, and hence fitness, can be very long (Figure~\ref{fig:numerics}(g)).
\end{remark}

\section{The nonlinear problem: mild solutions}\label{sec:nonlinear}

In this section $S$ is \emph{any} Markov semigroup on $L^1(\Omega)$, for an interval $\Omega\subseteq\R$ (bounded or not), and $r:\Omega\times[0,\infty)\to[0,\rmax]$ is continuous, with $g$ as in (H2). We do not yet use (H4) or the mixing estimate; this generality lets us use the same results on the real line in Section~\ref{sec:R}.

\subsection{Definition and motivation}

For compartmental models, $u'=Au+F(t)$ is equivalent to the variation-of-constants formula $u(t)=e^{tA}u(0)+\int_0^te^{(t-s)A}F(s)\dd s$. For the PDE \eqref{eq:model}, we take the corresponding formula as the definition of a solution. This is standard in the theory of semilinear evolution equations \cite[Ch.~6]{Pazy1983}, and it has the advantage of requiring no differentiability of $u$; everything we need will be derived from the formula itself.

For $\phi\in\Lone$ with $\int\phi>0$ and $t\ge0$, let
\[
N(t,\phi):=g\Bigl(\int_\Omega\phi\dd x\Bigr)\,r(\cdot,t)\,\phi ,
\]
the reaction term of \eqref{eq:model}, a function in $\Lone$.

\begin{definition}[Mild solution]\label{def:mild}
Let $u_0\in\Lone$ with $\int u_0>0$ and $0<T\le\infty$. A \emph{mild solution} of \eqref{eq:model} on $[0,T)$ is a continuous $u:[0,T)\to\Lone$ with $U(t):=\int_\Omega u(t)\dd x>0$ and
\begin{equation}\label{eq:mild}
u(t)=S(t)u_0+\int_0^tS(t-s)F(s)\dd s,\qquad F(s):=N\bigl(s,u(s)\bigr)=g\bigl(U(s)\bigr)r(\cdot,s)u(s),
\end{equation}
for all $t\in[0,T)$.
\end{definition}

\begin{remark}[Mild and classical solutions]\label{rem:classical}
If $u$ is a classical solution of \eqref{eq:model} with the no-flux condition, then $u$ is a mild solution: for fixed $t$, differentiate $s\mapsto S(t-s)u(s)$ and integrate from $0$ to $t$, exactly as one derives the variation-of-constants formula for ODEs. Conversely, for smooth $u_0$ and $r$, the mild solution is classical, by parabolic regularity \cite[Ch.~6]{Pazy1983}. We never need this converse.
\end{remark}

\subsection{Local existence, uniqueness and positivity}

The results of this subsection are standard; their proofs are given in Appendix~\ref{app:mild}. The first one is a Banach fixed-point argument, in which the only care needed is to keep the total population away from zero, where $g$ may be singular (as for the Gompertz law).

\begin{proposition}[Local existence]\label{prop:local}
Let $0<a$ and $R>0$. There is $T_0=T_0(a,R)>0$, depending only on $a$, $R$, $\rmax$ and $g$, with the following property. For every $t_0\ge0$ and every $\phi_0\in\Lone$ with $\int\phi_0\ge2a$ and $\norm{\phi_0}\le R$, there is a unique continuous $u:[t_0,t_0+T_0]\to\Lone$ with $\int u(t)\ge a$ and
\begin{equation}\label{eq:mild-t0}
u(t)=S(t-t_0)\phi_0+\int_{t_0}^tS(t-s)N\bigl(s,u(s)\bigr)\dd s,\qquad t\in[t_0,t_0+T_0].
\end{equation}
\end{proposition}

Solutions can be glued: if $u$ solves \eqref{eq:mild} on $[0,t_0]$ and $\tilde u$ solves \eqref{eq:mild-t0} with $\phi_0=u(t_0)$, the concatenation solves \eqref{eq:mild} on the union of the intervals. This follows by applying $S(t-t_0)$ to \eqref{eq:mild} at time $t_0$ and using the semigroup property and Lemma~\ref{lem:riemann}(i).

For compartmental models, nonnegativity of solutions follows from the observation that a compartment that reaches zero can only gain cells. For mild solutions there is no pointwise derivative to inspect, so we argue differently: we rewrite the equation so that every term is manifestly nonnegative. The only obstruction is that the reaction coefficient $g(U(t))r(x,t)$ can be negative (when $U>\Ustar$); we remove it by adding and subtracting a large constant $k$, that is, by passing to $e^{kt}u$ (Lemma~\ref{lem:rescale} in Appendix~\ref{app:mild}).

\begin{proposition}[Positivity]\label{prop:positive}
If $u$ is a mild solution on $[0,T]$ with $u_0\ge0$, then $u(t)\ge0$ for all $t\in[0,T]$.
\end{proposition}

\begin{proof}
Let $G(t)=g(U(t))$, which is continuous on $[0,T]$; let $G_0=\max_{[0,T]}\abs G$ and $k=G_0\rmax$. Regarding $G$ as a given function of time, $u$ solves the \emph{linear} equation $u(t)=S(t)u_0+\int_0^tS(t-s)\,q(s)u(s)\dd s$ with $q(x,t)=G(t)r(x,t)$, which is \eqref{eq:lin-q}. By Lemma~\ref{lem:rescale} with $\beta(t)=e^{kt}$, $W(t)=e^{kt}u(t)$ solves \eqref{eq:lin-qc} with coefficient $k+q=k+G(t)r(x,t)\ge k-G_0\rmax=0$. By Lemma~\ref{lem:volterra}, the unique solution of \eqref{eq:lin-qc} is nonnegative; hence $W\ge0$ and $u=e^{-kt}W\ge0$.
\end{proof}

\subsection{Global existence, and the total population}

The next proposition contains Step~1 of the compartmental proof: the total population obeys a scalar ODE, is monotone, and stays between $U_0$ and $\Ustar$. Here it is also what makes the solution global.

\begin{proposition}[Global existence and monotonicity of $U$]\label{prop:global}
Let $u_0\ge0$ with $U_0>0$. Then \eqref{eq:model} has a unique global mild solution $u\in C([0,\infty);\Lone)$; it is nonnegative. Its total population $U(t)=\int_\Omega u(t)\dd x$ is continuously differentiable, with
\begin{equation}\label{eq:Uode}
U'(t)=\rho(t)\,g\bigl(U(t)\bigr),\qquad\rho(t):=\int_\Omega r(x,t)u(x,t)\dd x\ \ge0 ,
\end{equation}
and $U$ is monotone, with values between $U_0$ and $\Ustar$. In particular $U(t)\ge m:=\min\{U_0,\Ustar\}>0$, $\norm{u(t)}=U(t)\le M:=\max\{U_0,\Ustar\}$, $U(t)$ converges to a limit $U_\infty\in[m,M]$, and
\begin{equation}\label{eq:TV}
\int_0^\infty\abs{U'(t)}\dd t=\abs{U_\infty-U_0}<\infty .
\end{equation}
\end{proposition}

\begin{proof}
\emph{The ODE for $U$.} Let $u$ be a mild solution on some interval $[0,T)$ (for instance the local one of Proposition~\ref{prop:local} with $t_0=0$, $\phi_0=u_0$, $a=U_0/2$, $R=\norm{u_0}$). It is nonnegative by Proposition~\ref{prop:positive}. Integrate \eqref{eq:mild} over $\Omega$. By Lemma~\ref{lem:riemann}(i) the mass functional commutes with the time integral, and by mass conservation (S2) it is not changed by $S(t)$ or $S(t-s)$:
\begin{align*}
U(t)&=\int_\Omega S(t)u_0\dd x+\int_0^t\Bigl(\int_\Omega S(t-s)F(s)\dd x\Bigr)\dd s\\
&=U_0+\int_0^t\Bigl(\int_\Omega F(s)\dd x\Bigr)\dd s=U_0+\int_0^tg\bigl(U(s)\bigr)\rho(s)\dd s .
\end{align*}
The integrand $s\mapsto g(U(s))\rho(s)$ is continuous (Lemma~\ref{lem:Fcont} and continuity of the mass functional), so $U$ is $C^1$ and \eqref{eq:Uode} holds. This is exactly the compartmental computation ``summing the equations, the transition terms cancel'': switching does not change the total population.

\emph{Monotonicity.} Regard $\rho$ as a given continuous, nonnegative function of time. Then $U$ solves the scalar ODE $y'=\rho(t)g(y)$, and so does the constant $y\equiv\Ustar$, since $g(\Ustar)=0$. The right-hand side is locally Lipschitz in $y$ ($g$ is $C^1$), so solutions of this ODE are unique and two of them cannot cross. Hence $U(t)-\Ustar$ has a constant sign (or vanishes identically), and so do $g(U(t))$ and $U'(t)=\rho(t)g(U(t))$. Therefore $U$ is monotone and lies between $U_0$ and $\Ustar$; since $u\ge0$, also $\norm{u(t)}=U(t)$.

\emph{Global existence.} Let $a=m/2$ and $R=M$, and let $T_0=T_0(a,R)$ be given by Proposition~\ref{prop:local}. As long as the solution exists, $\int u(t)=U(t)\ge m=2a$ and $\norm{u(t)}\le M=R$, by what we have just shown. So from any time $t_0$ up to which the solution exists, Proposition~\ref{prop:local} with $\phi_0=u(t_0)$ extends it to $[t_0,t_0+T_0]$, with the \emph{same} $T_0$. Gluing, the solution exists on $[0,\infty)$; it is unique by the uniqueness part of Proposition~\ref{prop:local} applied on successive intervals. Finally, $U$ is monotone and bounded, so $U_\infty=\lim_{t\to\infty}U(t)$ exists, and $\int_0^\infty\abs{U'}=\abs{\int_0^\infty U'}=\abs{U_\infty-U_0}$ because $U'$ does not change sign.
\end{proof}

\section{Proof of Theorem~\ref{thm:A}}\label{sec:proofA}

We now follow the steps of the compartmental proof \cite[Appendix~A]{Review}: Steps~1--5 below correspond to (i)--(v) in the introduction, and Step~6 collects the conclusions. Assume (H1)--(H4), let $S$ be the Markov semigroup of Proposition~\ref{prop:S-markov}, let $C\ge1$ and $\lambda>0$ be constants for which the mixing estimate \eqref{eq:mix} holds for all $t\ge0$ and $w\in\Lzero$ (by Lemmas~\ref{lem:mix} and~\ref{lem:gap} one may take $\lambda=\lambda_1$ and $C=C_1$), and let $u$ be the global nonnegative mild solution of Proposition~\ref{prop:global}. Part (a) of Theorem~\ref{thm:A} is Proposition~\ref{prop:global}.

\paragraph{Step 1: the total population is monotone.}
This is Proposition~\ref{prop:global}: $U'=\rho\,g(U)$, $U$ is monotone between $U_0$ and $\Ustar$, $U(t)\ge m>0$, and $\int_0^\infty\abs{U'}\dd t=\abs{U_\infty-U_0}<\infty$.

\paragraph{Step 2: the reaction term has one sign.}
Let $F(t)=g(U(t))\,r(\cdot,t)\,u(t)$. Together with Step~1, this is where uniform competition (H1) is used. Since $r\ge0$ and $u\ge0$, and since the factor $g(U(t))$ is the same for all phenotypes $x$, the function $F(t)$ has the sign of the number $g(U(t))$ at every phenotype. Therefore $\abs{F(t)}=\abs{g(U(t))}\,r\,u$ pointwise, and integrating over $\Omega$,
\begin{equation}\label{eq:F}
\int_\Omega F(t)\dd x=g\bigl(U(t)\bigr)\rho(t)=U'(t),\qquad\norm{F(t)}=\abs{g\bigl(U(t)\bigr)}\rho(t)=\abs{U'(t)} .
\end{equation}
The total size of the reaction term equals the growth rate of the total population, which is integrable over $[0,\infty)$ by \eqref{eq:TV}. We never need to know how fast $g(U(t))$ tends to zero.

\paragraph{Step 3: Duhamel's formula for the deviation.}
Let
\[
d(t):=u(t)-U(t)\,\psi
\]
be the deviation of the solution from the stationary density, scaled by the current population size. Since $\int u(t)=U(t)$ and $\int\psi=1$, $d(t)\in\Lzero$ for every $t$.

In the compartmental case we differentiated $d$ and used $A(U\pi)=U\,A\pi=0$ to find $d'=Ad+(F-U'\pi)$. Formally the same computation works here, since $\mathcal L(U\psi)=U\,\mathcal L\psi=0$ ($U(t)$ is a number and $\psi$ is stationary):
\[
\partial_td=\partial_tu-U'\psi=\mathcal Lu+F-U'\psi=\mathcal Ld+\bigl(F-U'\psi\bigr).
\]
Because $u$ is only a mild solution, we derive the corresponding Duhamel formula directly. By the fundamental theorem of calculus applied to the $C^1$ function $U$ and the fixed function $\psi$, and by the invariance $S(t)\psi=\psi$ (Lemma~\ref{lem:psi}),
\[
U(t)\psi=U_0\psi+\int_0^tU'(s)\,\psi\dd s=S(t)\bigl(U_0\psi\bigr)+\int_0^tS(t-s)\bigl(U'(s)\psi\bigr)\dd s .
\]
Subtracting this identity from the mild formula \eqref{eq:mild} gives
\begin{equation}\label{eq:duhamel-d}
d(t)=S(t)d(0)+\int_0^tS(t-s)\,\varphi(s)\dd s,\qquad\varphi(s):=F(s)-U'(s)\psi .
\end{equation}
Two properties of the forcing $\varphi$ let us apply the mixing estimate \eqref{eq:mix}, which is valid only for zero-mass functions:
\begin{itemize}[itemsep=2pt]
\item $\varphi(s)\in\Lzero$: by \eqref{eq:F}, $\int\varphi(s)=U'(s)-U'(s)\int\psi=0$;
\item its size is controlled by $\abs{U'}$: by the triangle inequality, \eqref{eq:F} and $\norm\psi=1$,
\[
\norm{\varphi(s)}\le\norm{F(s)}+\abs{U'(s)}\,\norm\psi=2\abs{U'(s)} .
\]
\end{itemize}
Take norms in \eqref{eq:duhamel-d}, use $\norm{\int h}\le\int\norm h$ for the integral, and apply \eqref{eq:mix} to $S(t)d(0)$ and to each $S(t-s)\varphi(s)$:
\[
\norm{d(t)}\le\norm{S(t)d(0)}+\int_0^t\norm{S(t-s)\varphi(s)}\dd s\le Ce^{-\lambda t}\norm{d(0)}+\int_0^tCe^{-\lambda(t-s)}\norm{\varphi(s)}\dd s .
\]
With $\norm{\varphi(s)}\le2\abs{U'(s)}$ this is \eqref{eq:star}.

\paragraph{Step 4: the deviation vanishes.}
The first term of \eqref{eq:star} tends to $0$. For the second, let $I(t)=\int_0^te^{-\lambda(t-s)}\abs{U'(s)}\dd s$ and split it at $t/2$:
\[
I(t)=\underbrace{\int_0^{t/2}e^{-\lambda(t-s)}\abs{U'(s)}\dd s}_{I_1(t)}+\underbrace{\int_{t/2}^te^{-\lambda(t-s)}\abs{U'(s)}\dd s}_{I_2(t)} .
\]
On $[0,t/2]$ we have $t-s\ge t/2$, so $e^{-\lambda(t-s)}\le e^{-\lambda t/2}$, a bound independent of $s$; on $[t/2,t]$ we simply use $e^{-\lambda(t-s)}\le1$. Hence, by \eqref{eq:TV},
\[
I_1(t)\le e^{-\lambda t/2}\int_0^\infty\abs{U'}\dd s=e^{-\lambda t/2}\abs{U_\infty-U_0}\to0,\qquad I_2(t)\le\int_{t/2}^\infty\abs{U'(s)}\dd s\to0,
\]
the latter because it is the tail of a convergent integral. Therefore
\begin{equation}\label{eq:d-to-0}
\norm{d(t)}\le\underbrace{Ce^{-\lambda t}\norm{d(0)}}_{\to0}+2C\underbrace{I_1(t)}_{\to0}+2C\underbrace{I_2(t)}_{\to0}\ \longrightarrow\ 0 .
\end{equation}
The three terms vanish for different reasons: the first and second because of the exponential factor against fixed constants, the third because it is the tail of a convergent integral. It is the third mechanism that lets us avoid any assumption on the speed at which $U'$, or $g(U)$, decays. At this point we do not yet know that $U_\infty=\Ustar$.

\paragraph{Step 5: the total population reaches $\Ustar$.}
If $U_0=\Ustar$ then $U\equiv\Ustar$. Suppose $U_0<\Ustar$ (the case $U_0>\Ustar$ is symmetric) and, for contradiction, $U_\infty<\Ustar$. We derive two incompatible statements about the total proliferation $\int_0^\infty\rho\dd t$.

\emph{(A) If $U_\infty<\Ustar$, then $\int_0^\infty\rho\dd t<\infty$.} By Step~1, $U(t)\in[U_0,U_\infty]\subset(0,\Ustar)$ for all $t$. On this compact interval $g$ is continuous and positive, hence bounded below by some $g_{\min}>0$. By \eqref{eq:Uode}, $\rho(t)=U'(t)/g(U(t))\le U'(t)/g_{\min}$, and integrating,
\[
\int_0^\infty\rho(t)\dd t\le\frac{U_\infty-U_0}{g_{\min}}<\infty .
\]

\emph{(B) In any case, $\int_0^\infty\rho\dd t=\infty$.} Since $r\ge0$ and $u\ge0$, restricting the integral defining $\rho$ to the interval $J$ of (H4) can only decrease it, so that for $t\ge t_r$
\[
\rho(t)=\int_\Omega r\,u\dd x\ \ge\ \int_Jr\,u\dd x\ \ge\ \rmin\int_Ju(x,t)\dd x .
\]
To bound $\int_Ju$ from below, use $u=d+U\psi$:
\[
\int_Ju(x,t)\dd x=U(t)\int_J\psi\dd x+\int_Jd(x,t)\dd x\ \ge\ m\,\psi(J)-\norm{d(t)},\qquad\psi(J):=\int_J\psi\dd x>0,
\]
where we used $U(t)\ge m$ (Step~1), $\int_Jd\ge-\int_J\abs d\ge-\norm d$, and $\psi>0$. By Step~4 there is $T\ge t_r$ with $\norm{d(t)}\le\frac12m\psi(J)$ for $t\ge T$. Hence
\[
\rho(t)\ \ge\ \rmin\cdot\tfrac12m\,\psi(J)=:c_0>0\qquad\text{for all }t\ge T,
\]
and $\int_0^\infty\rho\dd t\ge\int_T^\infty c_0\dd t=\infty$.

\emph{Conclusion.} (A) and (B) contradict each other, so $U_\infty=\Ustar$. This proves Theorem~\ref{thm:A}(b). As in the compartmental case, (A) uses only the common modulation and the shape of $g$, while (B) uses only the switching dynamics and persistent proliferation: switching keeps bringing cells into the proliferating range $J$, so their share there never falls much below $\psi(J)$, and this forces the population to keep growing until it reaches its carrying capacity.

\emph{A by-product: selection is integrable in time.} Part (B) did not use the assumption $U_\infty<\Ustar$: for every solution there are $T\ge0$ and $c_0>0$ with $\rho(t)\ge c_0$ for $t\ge T$. Since $\abs{g(U)}=\abs{U'}/\rho$, \eqref{eq:TV} gives
\begin{equation}\label{eq:gint}
\int_0^\infty\abs{g(U(t))}\dd t\le T\max_{[0,T]}\abs{g(U)}+\frac{\abs{U_\infty-U_0}}{c_0}<\infty .
\end{equation}
This is used in Remark~\ref{rem:replicator} and in the proof of Corollary~\ref{cor:rateR}.

\paragraph{Step 6: conclusion.}
By Steps 4 and 5,
\[
\norm{u(t)-\Ustar\psi}\le\norm{d(t)}+\abs{U(t)-\Ustar}\,\norm\psi\ \longrightarrow\ 0,
\]
which, together with \eqref{eq:star} from Step~3, proves Theorem~\ref{thm:A}(c). For (d), let $f\in\Lone$ with $\int f=\Ustar$. Then $f-\Ustar\psi\in\Lzero$ and $S(t)f-\Ustar\psi=S(t)(f-\Ustar\psi)$, so $\norm{S(t)f-\Ustar\psi}\le Ce^{-\lambda t}\norm{f-\Ustar\psi}\to0$, and
\[
\norm{u(t)-S(t)f}\le\norm{u(t)-\Ustar\psi}+\norm{\Ustar\psi-S(t)f}\to0 .\qquad\qed
\]

\subsection{Proofs of the corollaries}

\begin{proof}[Proof of Corollary~\ref{cor:landscape}]
Substitute $v=-P'$ in \eqref{eq:psi}. For constant $D$, $\int_0^xP'/D=(P(x)-P(0))/D$, and the factor $e^{P(0)/D}$ is absorbed into the normalising constant.
\end{proof}

\begin{proof}[Proof of Corollary~\ref{cor:rate}]
Let $\gamma:=-g'(\Ustar)>0$ and $z(t)=U(t)-\Ustar$. By Taylor's theorem there is $\delta\in(0,\Ustar/2]$ such that, for $\abs z\le\delta$, $g(\Ustar+z)$ has the sign of $-z$ and $\tfrac\gamma2\abs z\le\abs{g(\Ustar+z)}\le2\gamma\abs z$. By Theorem~\ref{thm:A}(b),(c) there is $T_1\ge t_r$ such that $\abs{z(t)}\le\delta$ and $\norm{d(t)}\le\frac14\Ustar\psi(J)$ for $t\ge T_1$. For such $t$, $U(t)\ge\Ustar/2$ and, as in Step~5(B),
\[
\rho(t)\ge\rmin\bigl(U(t)\psi(J)-\norm{d(t)}\bigr)\ge\tfrac14\rmin\Ustar\psi(J)=:c_* .
\]
Hence, for $t\ge T_1$, $z$ moves towards $0$ and $\abs{z}'=-\rho\abs{g(U)}\le-c_*\frac\gamma2\abs z$, so
\[
\abs{z(t)}\le\abs{z(T_1)}e^{-\omega(t-T_1)},\qquad\omega:=\frac{c_*\gamma}2=\frac18\,\abs{g'(\Ustar)}\,\rmin\,\Ustar\,\psi(J) .
\]
Moreover $\abs{U'}=\rho\abs{g(U)}\le\rmax M\cdot2\gamma\abs z$ for $t\ge T_1$, and $\abs{U'}$ is bounded on $[0,T_1]$; hence $\abs{U'(t)}\le K_1e^{-\omega t}$ for all $t\ge0$, for some $K_1$. Take $\lambda=\lambda_1$ and $C=C_1$ in \eqref{eq:star} (Lemma~\ref{lem:gap}), and let $\nu=\min\{\lambda_1,\omega\}$, which depends only on $\ell$, $D$, $v$, $g$, $\rmin$ and $J$. Then
\[
\norm{d(t)}\le C_1e^{-\lambda_1t}\norm{d(0)}+2C_1K_1\int_0^te^{-\lambda_1(t-s)}e^{-\omega s}\dd s\le C_1e^{-\nu t}\norm{d(0)}+2C_1K_1\,t\,e^{-\nu t},
\]
since $e^{-\lambda_1(t-s)}e^{-\omega s}\le e^{-\nu(t-s)}e^{-\nu s}=e^{-\nu t}$. Finally $\norm{u-\Ustar\psi}\le\norm d+\abs z$. (If $\lambda_1\ne\omega$, the integral is at most $e^{-\nu t}/\abs{\lambda_1-\omega}$, and the factor $1+t$ can be dropped.)
\end{proof}

\begin{remark}[Time scales]\label{rem:timescale}
Splitting the integral in \eqref{eq:star} at any $T\le t$ gives
\[
\norm{u(t)-U(t)\psi}\le Ce^{-\lambda t}\norm{u_0-U_0\psi}+2Ce^{-\lambda(t-T)}\abs{U(T)-U_0}+2C\abs{\Ustar-U(T)} .
\]
If the population has essentially saturated by time $T$, the distortion caused by fitness differences decays like $e^{-\lambda(t-T)}$ afterwards. The statement ``fitness is irrelevant'' is therefore asymptotic: it holds on time scales long compared with $1/\lambda_1$, the time the switching dynamics needs to forget where cells started. In the example of Section~\ref{subsec:numerics} essentially all growth occurs before $t\approx5$, and the imprint of fitness then decays at the rate $\lambda_1$ (Figure~\ref{fig:numerics}); similarly, in the multi-well numerical example of \cite{Review} essentially all growth occurs before $t\approx10$, and the imprint of the fitness gradient disappears afterwards, on the time scale of the switching dynamics.
\end{remark}

\begin{remark}[Transient treatments]\label{rem:treatment}
A treatment that kills cells in a phenotype-dependent way adds a term $-\delta(x,t)u$ to \eqref{eq:model}, which violates (H1). Suppose that $\delta$ is continuous, $0\le\delta\le\delta_{\max}$, and $\delta(\cdot,t)=0$ for $t\ge t_F$. The results of Section~\ref{sec:nonlinear} hold on $[0,t_F]$, with the same proofs, for the reaction term $\bigl(g(U)r-\delta\bigr)u$: there is a unique mild solution, and it is nonnegative. Its total population now satisfies $U'=g(U)\rho-\int_\Omega\delta u\dd x$, so that $U'\ge-\delta_{\max}U$ while $U\le\Ustar$, and $U'\le0$ while $U\ge\Ustar$; hence $U$ stays between $\min\{U_0,\Ustar\}e^{-\delta_{\max}t}$ and $\max\{U_0,\Ustar\}$, the solution exists on $[0,t_F]$, and $u(t_F)\ge0$ has positive mass. From $t_F$ on, the model is again \eqref{eq:model}, and Theorem~\ref{thm:A}, applied with initial time $t_F$, gives $u(t)\to\Ustar\psi$. As in the compartmental case \cite{Review}, a transient treatment shapes the transient, and with it the regrowth of the population, but not the limit; by \eqref{eq:star}, its imprint on the composition fades at the rate $\lambda_1$ once the population has regrown.
\end{remark}

\begin{remark}[The composition sees a finite amount of selection]\label{rem:replicator}
The composition $p(t):=u(t)/U(t)$ is a probability density. Lemma~\ref{lem:rescale} with $\beta(t)=U_0/U(t)$, for which $\beta'/\beta=-U'/U=-g(U)\,\bar r$ with $\bar r(t):=\int_\Omega r(x,t)p(x,t)\dd x$, shows that $p$ is the mild solution of
\[
\partial_tp=\mathcal Lp+g\bigl(U(t)\bigr)\bigl(r(x,t)-\bar r(t)\bigr)p ,
\]
a replicator--mutator equation in which selection acts with intensity $g(U(t))$. By \eqref{eq:gint}, the total intensity $\int_0^\infty\abs{g(U(t))}\dd t$ is finite: whatever the fitness differences, they act during a finite effective time, after which only the switching dynamics remains. More precisely, since $U'=g(U)\,\bar r\,U$, the selection term has norm $\norm{g(U)(r-\bar r)p}\le\abs{g(U)}\int_\Omega(r+\bar r)p\dd x=2\abs{(\ln U)'}$, and, $U$ being monotone, the total selection that the composition undergoes is at most $2\abs{\ln(\Ustar/U_0)}$: large imprints of fitness require large changes in population size. Under additive competition, $\partial_tu=\mathcal Lu+\bigl(r-d(U)\bigr)u$, with a death rate $d$ that increases with the total population but is the same for all phenotypes (for instance $d(U)=U$, as in Figure~\ref{fig:numerics}), the same computation gives $\partial_tp=\mathcal Lp+(r-\bar r)p$, in which selection acts at full strength forever; for time-independent $r$, $p$ then converges to the normalised principal eigenfunction of $\mathcal L+r$, and the limit depends on $r$ (Figure~\ref{fig:numerics}(d)). This is the continuum counterpart of the compartmental birth--death example discussed in \cite{Review}.
\end{remark}

\subsection{A numerical illustration}\label{subsec:numerics}

Figure~\ref{fig:numerics} shows a simulation on $\Omega=(0,1)$ with constant $D=0.02$ and a tilted double-well landscape, $v=-P'$ with $P(x)=0.03\cos(4\pi x)+0.03\,x$, so that $\psi\propto e^{-P/D}$ puts about $68\%$ of the population in the left well. Proliferation strongly favours the right well, $r(x)=0.1+1.9/\bigl(1+e^{-(x-1/2)/0.02}\bigr)$, and $g(U)=1-U$. The initial density $u_0=0.01\,\psi$ has the stationary composition, so every deviation from $\psi$ is caused by fitness differences. We discretise \eqref{eq:model} on $N=400$ cells of width $h$ by the finite-volume scheme
\[
\frac{\dd u_i}{\dd t}=\frac{\mathcal J_{i+1/2}-\mathcal J_{i-1/2}}{h}+g(U_h)\,r_iu_i,\qquad\mathcal J_{i+1/2}=a_{i+1/2}\,\frac{u_{i+1}/\psi_{i+1}-u_i/\psi_i}{h},
\]
with $\mathcal J_{1/2}=\mathcal J_{N+1/2}=0$, $a=D\psi$ and $U_h=h\sum_iu_i$. This is the discrete counterpart of \eqref{eq:flux-w}: it conserves mass, preserves positivity, and has $(\psi_i)$ as exact stationary state; the semi-discrete system is itself a compartmental model of the type studied in \cite{Giaimo2025,Review}. Time integration uses a BDF method with relative tolerance $10^{-10}$.

During the growth phase ($t\lesssim5$) the composition is pulled towards the fast-proliferating well (Figure~\ref{fig:numerics}(c)), and its $L^1$ distance to $\psi$ reaches $1.26$. Once the population has saturated, this distance decays at the rate $0.0660$ (Figure~\ref{fig:numerics}(f)), which coincides with the spectral gap $\lambda_1\approx0.066$ of the discrete switching generator, as predicted by \eqref{eq:star} with $\lambda=\lambda_1$. For comparison, under the additive competition $\partial_tu=\mathcal Lu+(r-U)u$ of Remark~\ref{rem:replicator}, the composition converges to the principal eigenfunction of $\mathcal L+r$, which is concentrated in the fast-proliferating well (Figure~\ref{fig:numerics}(d)): there, fitness is not forgotten.

Figure~\ref{fig:numerics}(g) shows how the memory time $1/\lambda_1$ depends on the noise intensity, for the same potential. It grows like $e^{\Delta P/D}$, where $\Delta P\approx0.053$ is the height of the barrier seen from the shallower well (the Kramers regime of Remark~\ref{rem:constants}), from about one time unit at $D=0.1$ to about $900$ at $D=0.008$, whereas the growth phase lasts about five time units. When switching is slow, the period during which fitness shapes the composition can thus be much longer than the growth phase that created the imprint.

\begin{figure}[tbp]
\centering
\includegraphics[width=\textwidth]{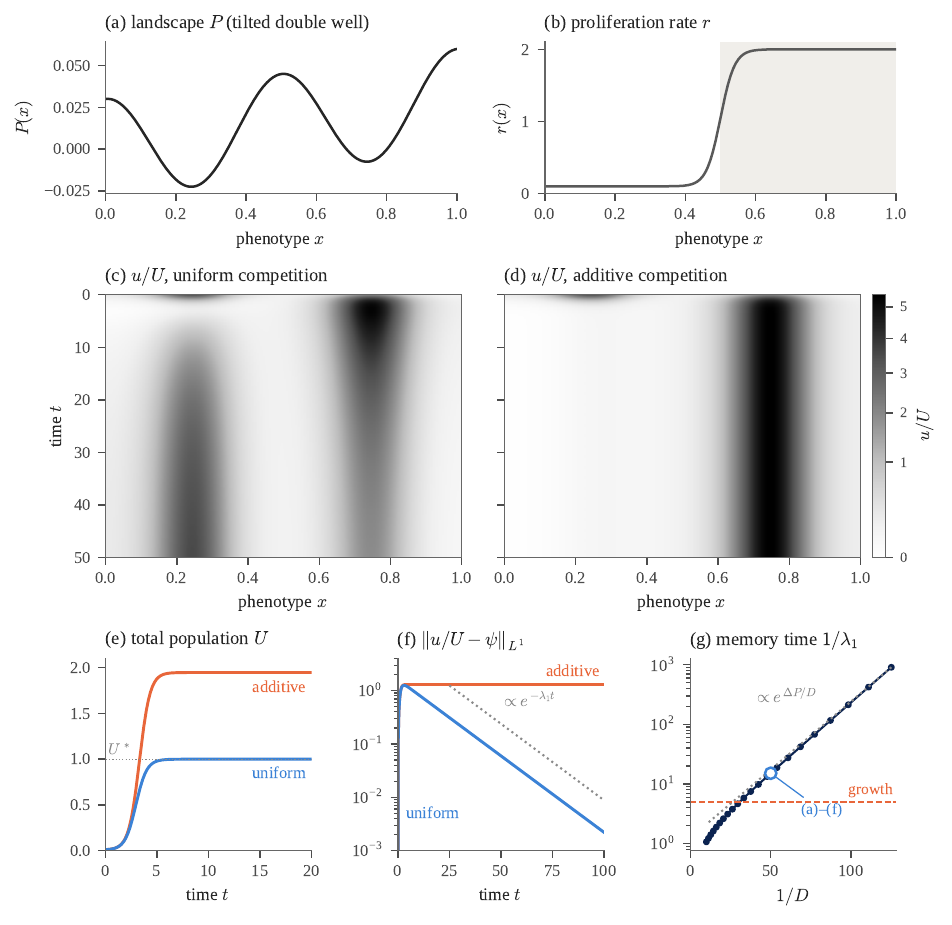}
\caption{Uniform competition erases the imprint of fitness; additive competition does not. Parameters in Section~\ref{subsec:numerics}. (a) The potential $P$, a tilted double well; $v=-P'$. (b) The proliferation rate $r$, fast on the shaded half. (c), (d) Composition $u/U$ as a function of phenotype and time (time runs downward; common nonlinear grey scale), under \eqref{eq:model} with $g(U)=1-U$ (c) and under the additive model $\partial_tu=\mathcal Lu+(r-U)u$ (d). Both start at the stationary composition, $u_0=0.01\,\psi$. Under uniform competition the composition is pulled towards the right well during growth and then returns to $\psi$; under additive competition it stays concentrated in the fast-proliferating well. (e) Total population. (f) $L^1$ distance of the composition to $\psi$: under \eqref{eq:model} it decays at the spectral gap $\lambda_1$ once the population has saturated (dotted: slope $-\lambda_1$), while under additive competition it does not decay. (g) Memory time $1/\lambda_1$ as a function of $1/D$, for the same potential: it grows like $e^{\Delta P/D}$ (dotted), with $\Delta P\approx0.053$ the barrier seen from the shallower well. The circle marks $D=0.02$, used in (a)--(f), and the dashed line the duration of the growth phase.}\label{fig:numerics}
\end{figure}

\section{The whole real line}\label{sec:R}

\subsection{What changes on an unbounded phenotype space}

On $\Omega=\R$, the only step of the proof of Theorem~\ref{thm:A} that fails is the mixing estimate \eqref{eq:mix}. Its proof used a lower bound $\Gamma(x,y,\tau)\ge c>0$ for \emph{all} pairs of phenotypes, which is impossible on $\R$: it would contradict $\int_\R\Gamma(x,y,\tau)\dd x=1$ (and the proof of Lemma~\ref{lem:gap} used $\min\psi>0$). For the Ornstein--Uhlenbeck dynamics, not only the proof but the estimate itself fails (Proposition~\ref{prop:OUnorate}): cells that start far from the bulk of the stationary distribution need an arbitrarily long time to arrive.

What survives is that each \emph{fixed} initial distribution is eventually forgotten. Inspecting the proof of Theorem~\ref{thm:A}, the linear dynamics enters only through three properties: $S$ is a Markov semigroup (Definition~\ref{def:markov}); $S(t)\psi=\psi$ for a strictly positive probability density $\psi$; and the switching dynamics forgets the initial state. On $\R$ the last one is replaced by the \emph{asymptotic stability} of $S$ \cite{LasotaMackey1994}
\begin{equation}\label{eq:E}
\Bigl\|S(t)f-\Bigl(\int_\R f\dd x\Bigr)\psi\Bigr\|_{L^1(\R)}\longrightarrow0\quad(t\to\infty)\qquad\text{for every }f\in L^1(\R),
\end{equation}
and the exponential bound in Step~4 is replaced by the dominated convergence theorem. Section~\ref{subsec:construction} constructs $S$ and proves \eqref{eq:E} under (R1)--(R3); Section~\ref{subsec:examples} gives examples; Section~\ref{subsec:norate} shows that no uniform rate holds for the OU dynamics; Section~\ref{subsec:proofB} proves Theorem~\ref{thm:B}; and Section~\ref{subsec:rateR} proves Corollary~\ref{cor:rateR}. Section~\ref{sec:general} turns this observation into a general theorem.

\subsection{Construction of the switching semigroup}\label{subsec:construction}

\emph{In a nutshell.} In the variable $w=u/\psi$, the switching dynamics becomes $\partial_tw=\psi^{-1}\partial_x(a\,\partial_xw)$, with $a=D\psi$, whose generator is symmetric in $L^2(\psi\dd x)$; spectral theory then gives $w\to\int w\,\psi\dd x$, that is, $u\to(\int u)\psi$. The only delicate point is mass conservation, since mass could escape through infinity. In the \emph{scale coordinate} $\eta=\int_0^x\dd y/a(y)$, the operator becomes $(D\psi^2)^{-1}\partial_{\eta\eta}$, a diffusion without drift, and (R3) says exactly that $\eta(x)\to\pm\infty$ as $x\to\pm\infty$: the ends of the phenotype space remain infinitely far away in this coordinate, and this prevents the leak (Step~(a) of Proposition~\ref{prop:SR}).

\paragraph{Symmetrisation.} The substitution that turns the problem into a symmetric one is the same one mentioned after Theorem~\ref{bb:green}. Assume (R1)--(R2), let $\psi$ be given by \eqref{eq:psiR} and set
\[
a(x):=D(x)\psi(x),\qquad\mu(\dd x):=\psi(x)\dd x .
\]
Since $\psi'=(v/D)\psi$ with $v/D$ continuous, $\psi\in C^1(\R)$, so $a\in C^1(\R)$ and $a>0$; $\mu$ is a probability measure by (R2). For $u=\psi w$, the identity $D\psi'=v\psi$ gives
\begin{equation}\label{eq:symm}
D\,\partial_xu-v\,u=D\psi\,\partial_xw+(D\psi'-v\psi)w=a\,\partial_xw,\qquad\text{hence}\qquad\mathcal Lu=\partial_x\bigl(a\,\partial_xw\bigr).
\end{equation}
So $\partial_tu=\mathcal Lu$ is equivalent to $\partial_tw=\psi^{-1}\partial_x(a\,\partial_xw)=:-\mathcal A_0w$. On $C_c^\infty(\R)\subset L^2(\mu)$, integration by parts gives, with $\langle w,\tilde w\rangle_\mu=\int w\tilde w\,\psi\dd x$,
\begin{equation}\label{eq:form}
\langle\mathcal A_0w,\tilde w\rangle_\mu=-\int_\R\bigl(a\,w'\bigr)'\,\tilde w\dd x=\int_\R a\,w'\,\tilde w'\dd x=:\mathcal E(w,\tilde w),
\end{equation}
so $\mathcal A_0$ is symmetric and nonnegative in $L^2(\mu)$, and $\mathcal E(w,w)=\int a(w')^2$ is the energy of the switching dynamics. We use the following standard facts.

\begin{blackbox}[Friedrichs extension and Dirichlet forms]\label{bb:dirichlet}
Let $a\in C^1(\R)$ and $\psi\in C(\R)$ be strictly positive, with $\int\psi=1$.
\begin{enumerate}[label=(\alph*), itemsep=3pt]
\item The form $\mathcal E$ of \eqref{eq:form}, defined on $C_c^\infty(\R)$, is closable in $L^2(\mu)$. Let $(\mathcal E,\mathcal F)$ be its closure. There is a unique nonnegative self-adjoint operator $\mathcal A$ on $L^2(\mu)$ (the Friedrichs extension of $\mathcal A_0$) such that $w\in D(\mathcal A)$ and $\mathcal Aw=h$ if and only if $w\in\mathcal F$ and $\mathcal E(w,\tilde w)=\langle h,\tilde w\rangle_\mu$ for all $\tilde w\in\mathcal F$.
\item $(\mathcal E,\mathcal F)$ is a Dirichlet form. Consequently the self-adjoint contraction semigroup $T(t)=e^{-t\mathcal A}$ on $L^2(\mu)$ is \emph{Markovian}: $0\le w\le1$ implies $0\le T(t)w\le1$; in particular $T(t)$ preserves positivity.
\item Every $w\in\mathcal F$ has a locally absolutely continuous representative with $w'\in L^2_{\rm loc}$, and $\mathcal E(w,w)=\int a\,(w')^2\dd x$. Every Lipschitz function with compact support belongs to $\mathcal F$.
\end{enumerate}
\end{blackbox}

\noindent\emph{Comments and references.} (a) is the Friedrichs extension theorem, or equivalently the correspondence between closed nonnegative forms and nonnegative self-adjoint operators \cite{Kato1995,ReedSimon1975}. (b) is the Beurling--Deny criterion: the form does not increase under the unit contraction $w\mapsto(0\vee w)\wedge1$, since $\bigl((0\vee w)\wedge1\bigr)'=w'\mathbf 1_{\{0<w<1\}}$ \cite{FukushimaOshimaTakeda2011}. (c) holds because on each compact set $K$ the coefficients $a$ and $\psi$ are bounded below by positive constants, so the norm $\bigl(\mathcal E(w,w)+\norm w_{L^2(\mu)}^2\bigr)^{1/2}$ controls the $H^1(K)$ norm, and Lipschitz functions with compact support are limits of their mollifications in this norm. Nothing in this theorem uses (R3).

\begin{proposition}[The switching semigroup on $\R$]\label{prop:SR}
Assume \textup{(R1)--(R3)}, and for $f\in L^1(\R)$ with $\int f^2/\psi\dd x<\infty$ put
\[
S(t)f:=\psi\,T(t)\bigl(f/\psi\bigr).
\]
Then $S(t)$ extends uniquely to a Markov semigroup on $L^1(\R)$ such that:
\begin{enumerate}[label=(\roman*), itemsep=2pt]
\item $S(t)\psi=\psi$ for all $t\ge0$;
\item $S$ is asymptotically stable: \eqref{eq:E} holds;
\item $S$ solves the Fokker--Planck equation weakly: for $f\in L^1(\R)$ and $\zeta\in C_c^\infty(\R)$,
\[
\int_\R\bigl(S(t)f\bigr)\zeta\dd x=\int_\R f\zeta\dd x+\int_0^t\int_\R\bigl(S(s)f\bigr)\,\mathcal L^*\zeta\dd x\dd s,\qquad\mathcal L^*\zeta:=(D\zeta')'+v\zeta' .
\]
\end{enumerate}
\end{proposition}

\begin{proof}
\emph{Step (a): no loss of mass. This is where (R3) is used.} We show that the constant function $\mathbf 1$ belongs to $\mathcal F$ and $\mathcal E(\mathbf 1,\mathbf 1)=0$. Let $\eta(x)=\int_0^x\dd y/a(y)$ be the \emph{scale function}. It is increasing, and by (R3) it maps $\R$ onto $\R$. For $0<n<N$, let $\phi_{n,N}:\R\to[0,1]$ be the piecewise linear function equal to $1$ on $[\eta(-n),\eta(n)]$, to $0$ outside $[\eta(-N),\eta(N)]$, and linear in between, and set $w_{n,N}(x)=\phi_{n,N}(\eta(x))$. This is a Lipschitz function with compact support, so $w_{n,N}\in\mathcal F$ by Theorem~\ref{bb:dirichlet}(c). Since $w_{n,N}'=\phi_{n,N}'(\eta(x))/a(x)$, the change of variables $\xi=\eta(x)$, $\dd\xi=\dd x/a$, gives
\[
\mathcal E(w_{n,N},w_{n,N})=\int_\R a\Bigl(\frac{\phi_{n,N}'(\eta(x))}{a(x)}\Bigr)^2\dd x=\int_\R\phi_{n,N}'(\xi)^2\dd\xi=\frac1{\eta(N)-\eta(n)}+\frac1{\eta(-n)-\eta(-N)} .
\]
By (R3), $\eta(\pm N)\to\pm\infty$, so we can choose $N=N_n$ with $\mathcal E(w_n,w_n)\le1/n$, where $w_n:=w_{n,N_n}$. Then $0\le w_n\le1$ and $w_n\to1$ pointwise, so $w_n\to\mathbf 1$ in $L^2(\mu)$ by dominated convergence ($\mu$ is finite). Hence $(w_n)$ is a Cauchy sequence for the norm $\bigl(\mathcal E(\cdot,\cdot)+\norm\cdot_{L^2(\mu)}^2\bigr)^{1/2}$, because $\mathcal E(w_n-w_m,w_n-w_m)\le2\mathcal E(w_n,w_n)+2\mathcal E(w_m,w_m)$. Since the form is closed, $\mathbf 1\in\mathcal F$ and $\mathcal E(\mathbf 1,\mathbf 1)=\lim\mathcal E(w_n,w_n)=0$. By the Cauchy--Schwarz inequality for $\mathcal E$, $\mathcal E(\mathbf 1,\tilde w)=0$ for every $\tilde w\in\mathcal F$, so Theorem~\ref{bb:dirichlet}(a) gives $\mathbf 1\in D(\mathcal A)$ with $\mathcal A\mathbf 1=0$, and therefore
\[
T(t)\mathbf 1=e^{-t\mathcal A}\mathbf 1=\mathbf 1\qquad\text{for all }t\ge0 .
\]
This is the continuum counterpart of the columns of $e^{tA}$ summing to one. Without (R3), the cut-offs could not be made to have small energy, and in general mass is lost through infinity.

\emph{Step (b): a Markov semigroup on $L^1$.} Let $w\in L^2(\mu)$. By positivity, $\abs{T(t)w}\le T(t)\abs w$ (as in the proof of Lemma~\ref{lem:markov}(i)); by self-adjointness and Step~(a),
\[
\int_\R T(t)\abs w\dd\mu=\langle T(t)\abs w,\mathbf 1\rangle_\mu=\langle\abs w,T(t)\mathbf 1\rangle_\mu=\int_\R\abs w\dd\mu,\qquad\int_\R T(t)w\dd\mu=\int_\R w\dd\mu .
\]
So $T(t)$ is a contraction for the norm of $L^1(\mu)$ and preserves $\int\cdot\dd\mu$. Since $\mu$ is finite, $L^2(\mu)$ is dense in $L^1(\mu)$, and $T(t)$ extends uniquely to a positive, mass-preserving contraction on $L^1(\mu)$; the semigroup property passes to the extension. For $w\in L^2(\mu)$, $\norm{T(t)w-w}_{L^1(\mu)}\le\norm{T(t)w-w}_{L^2(\mu)}\to0$ as $t\downarrow0$ (Cauchy--Schwarz, $\mu(\R)=1$), and by density and contractivity this holds for every $w\in L^1(\mu)$. Finally, $f\mapsto f/\psi$ is an isometry from $L^1(\dd x)$ onto $L^1(\mu)$ that preserves positivity and maps $\int f\dd x$ to $\int(f/\psi)\dd\mu$. Conjugating by it, $S(t)f=\psi\,T(t)(f/\psi)$ is a Markov semigroup on $L^1(\R)$, and $S(t)\psi=\psi\,T(t)\mathbf 1=\psi$, which is (i).

\emph{Step (c): the switching dynamics forgets its initial state.} If $\mathcal Aw=0$, then $\int a\,(w')^2\dd x=\mathcal E(w,w)=\langle\mathcal Aw,w\rangle_\mu=0$, so $w'=0$ a.e.\ (because $a>0$) and $w$ is constant (Theorem~\ref{bb:dirichlet}(c)). Thus $\ker\mathcal A=\operatorname{span}\{\mathbf 1\}$, and the orthogonal projection onto it is $P_0w=\langle w,\mathbf 1\rangle_\mu\mathbf 1=\bigl(\int w\dd\mu\bigr)\mathbf 1$. Let $E_w$ be the spectral measure of $\mathcal A$ associated with $w$, a finite measure on $[0,\infty)$ whose atom at $0$ is $\norm{P_0w}^2$. By the spectral theorem \cite{ReedSimon1980},
\[
\norm{T(t)w-P_0w}_{L^2(\mu)}^2=\int_{(0,\infty)}e^{-2t\xi}\,E_w(\dd\xi)\longrightarrow0\qquad(t\to\infty),
\]
by dominated convergence (the integrand is at most $1$ and tends to $0$ for every $\xi>0$). Now let $f\in L^1(\R)$ with $\int f^2/\psi<\infty$, and $w=f/\psi\in L^2(\mu)$. Since $\int w\dd\mu=\int f\dd x$, the isometry of Step~(b) and Cauchy--Schwarz give
\[
\Bigl\|S(t)f-\Bigl(\int f\Bigr)\psi\Bigr\|_{L^1(\R)}=\Bigl\|T(t)w-\int w\dd\mu\Bigr\|_{L^1(\mu)}\le\norm{T(t)w-P_0w}_{L^2(\mu)}\longrightarrow0 .
\]
Such $f$ are dense in $L^1(\R)$ (they include the bounded functions with compact support, since $\psi$ is bounded below on compact sets). For general $f\in L^1(\R)$ and $f_k$ as above with $\norm{f-f_k}\to0$,
\begin{align*}
\Bigl\|S(t)f-\Bigl(\int f\Bigr)\psi\Bigr\|&\le\norm{S(t)(f-f_k)}+\Bigl\|S(t)f_k-\Bigl(\int f_k\Bigr)\psi\Bigr\|+\Bigl|\int f_k-\int f\Bigr|\\
&\le2\norm{f-f_k}+\Bigl\|S(t)f_k-\Bigl(\int f_k\Bigr)\psi\Bigr\|,
\end{align*}
so the $\limsup$ is at most $2\norm{f-f_k}$ for every $k$. This proves (ii).

\emph{Step (d): weak formulation.} Let $w_0\in L^2(\mu)$ and $w(t)=T(t)w_0$. For $t>0$, $w(t)\in D(\mathcal A)$ and $\frac{\dd}{\dd t}w=-\mathcal Aw$ in $L^2(\mu)$ (spectral theorem). For $\zeta\in C_c^\infty(\R)\subset D(\mathcal A_0)\subset D(\mathcal A)$, symmetry and \eqref{eq:form} give
\[
\frac{\dd}{\dd t}\langle w,\zeta\rangle_\mu=-\langle\mathcal Aw,\zeta\rangle_\mu=-\langle w,\mathcal A_0\zeta\rangle_\mu=\int_\R w\,(a\zeta')'\dd x=\int_\R(\psi w)\,\frac{(D\psi\zeta')'}{\psi}\dd x .
\]
With $u=\psi w$, $\langle w,\zeta\rangle_\mu=\int u\zeta\dd x$ and $(D\psi\zeta')'/\psi=(D\zeta')'+D\zeta'\psi'/\psi=(D\zeta')'+v\zeta'=\mathcal L^*\zeta$. Integrating in time, and letting $t\downarrow0$ using strong continuity, gives (iii) for $f=\psi w_0$. Both sides of (iii) are continuous in $f\in L^1(\R)$, because $\zeta$ and $\mathcal L^*\zeta$ are bounded and $S(s)$ is a contraction; so (iii) extends to all $f$ by density.
\end{proof}

\begin{remark}[Why this is the right semigroup]\label{rem:unique}
Identity \eqref{eq:symm} shows that the flux $D\partial_xu-vu$ equals $a\,\partial_xw$, so the construction above is the Fokker--Planck dynamics with no additional condition imposed at $\pm\infty$. Condition (R3) is exactly what makes this unambiguous: when both ends are inaccessible, the operator $\mathcal A_0$ has only one Markovian self-adjoint realisation in $L^2(\mu)$, and no boundary condition at infinity can, or needs to, be imposed; see \cite{Eberle1999,FukushimaOshimaTakeda2011}. When (R3) fails, cells may reach $\pm\infty$ in finite time, and one must decide what happens to them there (absorption, reflection, \dots), which changes the model.
\end{remark}

\begin{remark}[Rates in weighted norms]\label{rem:poincare}
If the Poincar\'e inequality \eqref{eq:poincare} holds, then $\mathcal A$ has a spectral gap $\lambda_P$, and the argument of Step~(c) gives $\norm{S(t)f-(\int f)\psi}_{L^1}\le e^{-\lambda_Pt}\norm{f-(\int f)\psi}_\psi$ (Lemma~\ref{lem:H}): an exponential rate, but only for initial data with $\norm f_\psi<\infty$. For the OU dynamics, $\lambda_P=\theta$. Proposition~\ref{prop:OUnorate} shows that such a rate cannot be uniform in the $L^1$ norm. Section~\ref{subsec:rateR} propagates the weighted norm along the nonlinear solution and proves Corollary~\ref{cor:rateR}.
\end{remark}

\subsection{Examples}\label{subsec:examples}

We first give two practical criteria, for (R3) and for the Poincar\'e inequality \eqref{eq:poincare}, and then examples. Figure~\ref{fig:examples} shows Theorem~\ref{thm:B} at work on examples~(1) and~(3) below.

\begin{lemma}[A sufficient condition for (R3)]\label{lem:R3}
If \textup{(R1)--(R2)} hold and $D\le D_{\max}$ on $\R$, then \textup{(R3)} holds.
\end{lemma}

\begin{proof}
By the Cauchy--Schwarz inequality, for every $n$,
\[
1=\Bigl(\int_n^{n+1}1\dd x\Bigr)^2\le\int_n^{n+1}\psi\dd x\int_n^{n+1}\frac{\dd x}{\psi},\qquad\text{so}\qquad\int_n^{n+1}\frac{\dd x}{D\psi}\ge\frac1{D_{\max}\int_n^{n+1}\psi\dd x} .
\]
Since $\psi$ is integrable, $\int_n^{n+1}\psi\to0$ as $n\to\infty$, so the terms on the right tend to $+\infty$ and $\int_0^\infty\dd x/(D\psi)=\sum_{n\ge0}\int_n^{n+1}\dd x/(D\psi)=\infty$. The same argument applies at $-\infty$.
\end{proof}

\begin{lemma}[A criterion for the Poincar\'e inequality]\label{lem:muckenhoupt}
Assume \textup{(R1)--(R2)} and let $a=D\psi$. If
\[
B_+:=\sup_{x>0}\Bigl(\int_x^\infty\psi\dd y\Bigr)\Bigl(\int_0^x\frac{\dd y}{a(y)}\Bigr)<\infty\qquad\text{and}\qquad B_-:=\sup_{x<0}\Bigl(\int_{-\infty}^x\psi\dd y\Bigr)\Bigl(\int_x^0\frac{\dd y}{a(y)}\Bigr)<\infty ,
\]
then \eqref{eq:poincare} holds with $\lambda_P=1/(4\max\{B_+,B_-\})$.
\end{lemma}

\begin{proof}
Let $w\in C_c^\infty(\R)$. Since the mean $\int w\psi$ minimises $c\mapsto\int(w-c)^2\psi$, the left-hand side of \eqref{eq:poincare} is at most $\int_\R(w-w(0))^2\psi\dd x$. For $x>0$, $\abs{w(x)-w(0)}\le\int_0^x\abs{w'}\dd y$, and Muckenhoupt's weighted Hardy inequality \cite{Muckenhoupt1972}, applied to $\abs{w'}$, gives
\[
\int_0^\infty\Bigl(\int_0^x\abs{w'}\dd y\Bigr)^2\psi(x)\dd x\le4B_+\int_0^\infty a\,(w')^2\dd x ;
\]
the same argument on $(-\infty,0)$ gives the corresponding bound with $B_-$. Adding the two inequalities proves the claim.
\end{proof}

\begin{enumerate}[label=(\arabic*), itemsep=4pt]
\item \emph{Ornstein--Uhlenbeck.} $v(x)=-\theta x$ ($\theta>0$) and constant $D$: $\psi$ is the centred Gaussian density with variance $\sigma^2=D/\theta$, (R1)--(R2) hold, and (R3) holds by Lemma~\ref{lem:R3}. The Poincar\'e inequality \eqref{eq:poincare} holds with $\lambda_P=\theta$, by the Gaussian Poincar\'e inequality \cite{BakryGentilLedoux2014}.
\item \emph{Multi-well landscapes.} $v=-P'$ with $P\in C^1(\R)$ and constant $D$: $\psi\propto e^{-P/D}$, and (R1)--(R3) hold as soon as $e^{-P/D}$ is integrable, for instance if $P(x)\ge\alpha\abs x-\beta$ for some $\alpha>0$. This covers rugged landscapes on the whole line with any number of wells, provided the potential grows at infinity (the multi-well examples of \cite{Review} are posed on a bounded interval, and are covered by Theorem~\ref{thm:A}). If, moreover, $P'(x)\ge\alpha$ for $x\ge x_0$ and $P'(x)\le-\alpha$ for $x\le-x_0$, for some $\alpha,x_0>0$, then \eqref{eq:poincare} holds by Lemma~\ref{lem:muckenhoupt}: for $x\ge x_0$, $\psi(y)\le\psi(x)e^{-\alpha(y-x)/D}$ for $y\ge x$ and $\psi(y)\ge\psi(x)e^{\alpha(x-y)/D}$ for $x_0\le y\le x$, so that $\int_x^\infty\psi\dd y\le D\psi(x)/\alpha$ and $\int_{x_0}^x\dd y/(D\psi)\le1/(\alpha\psi(x))$, and $B_+<\infty$ follows; similarly $B_-<\infty$.
\item \emph{State-dependent noise and heavy tails.} $v(x)=-\theta x$ and $D(x)=D_0(1+x^2)$, an example of \cite{Review}. Then $\int_0^xv/D=-\frac\theta{2D_0}\ln(1+x^2)$, so $\psi\propto(1+x^2)^{-\theta/(2D_0)}$, which is integrable if and only if $\theta>D_0$. Here $D$ is unbounded and Lemma~\ref{lem:R3} does not apply, but (R3) can be checked directly: $1/(D\psi)\propto(1+x^2)^{\theta/(2D_0)-1}$, whose integral diverges at $\pm\infty$ if and only if $\theta\ge D_0$. So every case in which the stationary density exists is covered, including power-law tails. The Poincar\'e inequality holds as well: with $\beta:=\theta/D_0>1$, as $x\to+\infty$,
\[
\int_x^\infty\psi\dd y\sim\frac{x^{1-\beta}}{Z(\beta-1)},\qquad\int_0^x\frac{\dd y}{D\psi}\sim\frac{Z\,x^{\beta-1}}{D_0(\beta-1)},
\]
so the product in Lemma~\ref{lem:muckenhoupt} tends to $1/(D_0(\beta-1)^2)$ and $B_+<\infty$; by symmetry, $B_-=B_+$. (The example in \cite{Review} uses the divergence form, as here. In the It\^o convention, in which the Fokker--Planck operator is $-\partial_x(bu)+\partial_{xx}(Du)$ with drift $b(x)=-\theta x$, one has $v=b-D'=-(\theta+2D_0)x$ and $\psi\propto(1+x^2)^{-1-\theta/(2D_0)}$, which is integrable for every $\theta>0$; (R3) and \eqref{eq:poincare} hold by the same computations, with $\beta=2+\theta/D_0$.)

\item \emph{Non-examples.} For pure diffusion ($v\equiv0$, constant $D$) or for drifts that push cells outwards, (R2) fails: there is no stationary density, and the population spreads indefinitely (see Remark~\ref{rem:heat}).
\end{enumerate}

\begin{figure}[tbp]
\centering
\includegraphics[width=\textwidth]{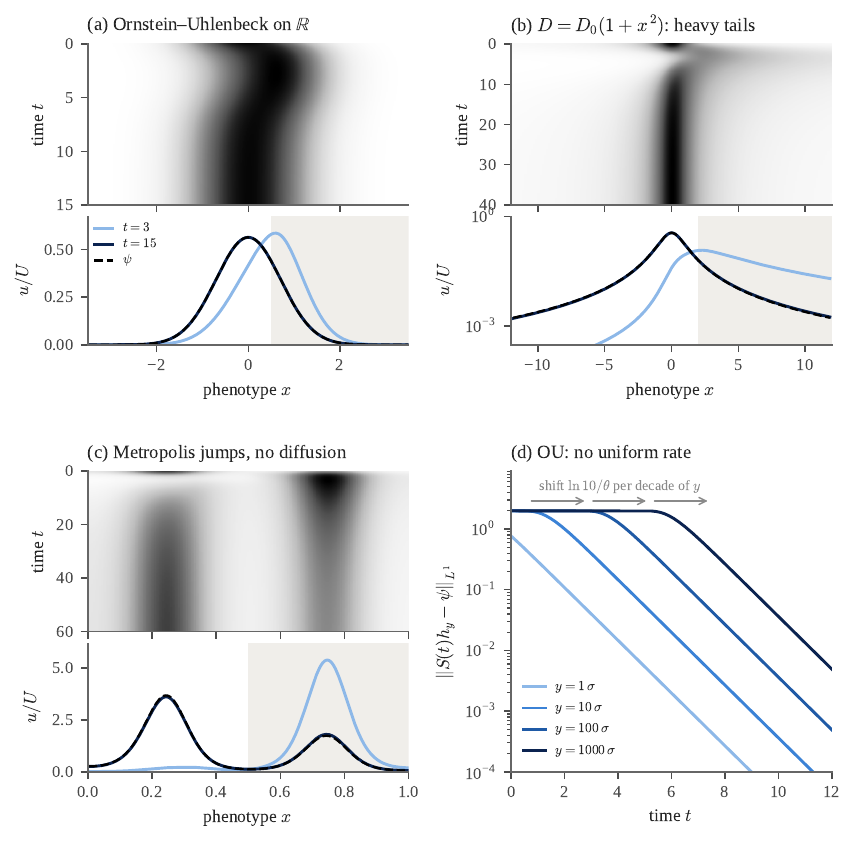}
\caption{Beyond the bounded interval and beyond diffusion. (a)--(c) Composition $u/U$ under \eqref{eq:model}, with $g(U)=1-U$ and $u_0=0.01\,\psi$, so that every deviation from $\psi$ is caused by fitness differences; proliferation is fast ($r\approx2$ instead of $0.1$) on the shaded range. Top: $u/U$ as a function of phenotype and time (time runs downward). Bottom: $u/U$ at the time of maximal distortion (light), at the final time ($t=15$, $40$ and $60$; dark), and $\psi$ (dashed). (a) Ornstein--Uhlenbeck dynamics, $v=-\theta x$, $\theta=1$, $D=0.5$ (Theorem~\ref{thm:B}, example~(1)). (b) State-dependent noise, $v=-\theta x$, $D=D_0(1+x^2)$, $\theta=0.65$, $D_0=0.3$, as in the heavy-tailed example of \cite{Review}; the stationary density has power-law tails, $\psi\propto(1+x^2)^{-\theta/(2D_0)}$ (logarithmic scale; example~(3)). (c) Nonlocal switching by Metropolis jumps on $(0,1)$, without diffusion: $K(x,y)=k\,q(x-y)\min\{1,\pi(x)/\pi(y)\}$, with $q$ a centred Gaussian of standard deviation $0.2$, $k=1$, and $\pi$ the Boltzmann-type density $\psi$ of Section~\ref{subsec:numerics} (Corollary~\ref{cor:nonlocal} and Remark~\ref{rem:nonlocal-examples}(2)). The population converges to the same density as in Figure~\ref{fig:numerics}, reaching the other well by jumps rather than by crossing the barrier. (d) No uniform rate for the Ornstein--Uhlenbeck dynamics (Proposition~\ref{prop:OUnorate}): $L^1$ distance to $\psi$ of the linear evolution of $h_y=\psi(\cdot-y)$, a copy of $\psi$ shifted by $y$ ($\sigma^2=D/\theta$), computed exactly from the Mehler formula \eqref{eq:mehler}, with $\theta=1$ and $D=0.5$ as in (a). Each decade of $y$ delays convergence by $\ln10/\theta$. In (a) and (b), discretisation as in Section~\ref{subsec:numerics}, with the line truncated to $[-8,8]$ in (a), and to $[-1000,1000]$ with a non-uniform grid in (b); in (c), the kernel is evaluated at the midpoints of $200$ cells of $(0,1)$.}\label{fig:examples}
\end{figure}

\begin{remark}[A half-line]\label{rem:halfline}
The construction of Section~\ref{subsec:construction}, and with it Theorem~\ref{thm:B}, carries over to $\Omega=(0,\infty)$ with the no-flux condition at $0$, as for the exponential distribution in \cite{Review}. Assume (R1)--(R2) on $[0,\infty)$, and (R3) at $+\infty$ only. Define the form $\mathcal E$ on the smooth functions with compact support in $[0,\infty)$, which need not vanish at $0$. Theorem~\ref{bb:dirichlet} holds with the same justification, except that the operator in (a) is now the one associated with the closed form, rather than the Friedrichs extension of $\mathcal A_0$: integrating by parts, $\mathcal E(w,\tilde w)=\langle h,\tilde w\rangle_\mu$ for all $\tilde w$ forces $a(0)w'(0)=0$, which by \eqref{eq:symm} is the no-flux condition for $u=\psi w$. In the proof of Proposition~\ref{prop:SR}, the scale function maps $[0,\infty)$ onto $[0,\infty)$, the cut-offs of Step~(a) are needed at $+\infty$ only, and the test functions in (iii) must satisfy $\zeta'(0)=0$; the rest is unchanged. For instance, for $v\equiv-\alpha<0$ and constant $D$, $\psi(x)=(\alpha/D)e^{-\alpha x/D}$, (R1)--(R3) hold, and the proof of Lemma~\ref{lem:muckenhoupt} on $(0,\infty)$ gives \eqref{eq:poincare} with $\lambda_P=1/(4B_+)$, since $B_+=\sup_{x>0}(1-e^{-\alpha x/D})D/\alpha^2=D/\alpha^2$.
\end{remark}

\subsection{No uniform rate for the Ornstein--Uhlenbeck dynamics}\label{subsec:norate}

For the OU dynamics ($v(x)=-\theta x$, constant $D$, $\sigma^2=D/\theta$), the semigroup of Proposition~\ref{prop:SR} is given explicitly by the Mehler kernel \cite{BakryGentilLedoux2014,Risken1989}:
\begin{equation}\label{eq:mehler}
\bigl(S(t)f\bigr)(x)=\int_\R G_{\sigma_t^2}\bigl(x-ye^{-\theta t}\bigr)f(y)\dd y,\qquad\sigma_t^2:=\sigma^2\bigl(1-e^{-2\theta t}\bigr),
\end{equation}
where $G_s(z)=(2\pi s)^{-1/2}e^{-z^2/(2s)}$ is the centred Gaussian density with variance $s$. In words: a cell that starts at $y$ is, at time $t$, normally distributed with mean $ye^{-\theta t}$ (the drift pulls it towards $0$) and variance $\sigma_t^2$ (diffusion spreads it). We also use a standard fact about Gaussians: if $X\sim\mathcal N(\mu_1,s_1)$ and, given $X$, $Y\sim\mathcal N(\alpha X,s_2)$, then $Y\sim\mathcal N(\alpha\mu_1,\alpha^2s_1+s_2)$.

\begin{proposition}[No uniform rate on $\R$ for the OU dynamics]\label{prop:OUnorate}
For the OU semigroup and every $t>0$,
\[
\sup\Bigl\{\frac{\norm{S(t)w}}{\norm w}:\ w\in L^1(\R),\ \int w=0,\ w\ne0\Bigr\}=1 .
\]
Consequently no estimate of the form \eqref{eq:mix} with $\lambda>0$ can hold.
\end{proposition}

\begin{proof}
The supremum is at most $1$ by contraction. For $y\in\R$ let $h_y(x)=\psi(x-y)$, the stationary Gaussian shifted to be centred at $y$, and $w_y=h_y-\psi$, so that $\int w_y=0$ and $\norm{w_y}\le2$. By the Gaussian fact, $S(t)h_y$ is the density of $\mathcal N\bigl(ye^{-\theta t},\ e^{-2\theta t}\sigma^2+\sigma_t^2\bigr)=\mathcal N(ye^{-\theta t},\sigma^2)$, while $S(t)\psi=\psi$. So $\norm{S(t)w_y}$ is the $L^1$ distance between two Gaussian densities with the same variance and means $ye^{-\theta t}$ apart. As $\abs y\to\infty$ with $t$ fixed, this distance tends to $2$, since the two densities become essentially disjoint, while $\norm{w_y}\le2$. Hence $\norm{S(t)w_y}/\norm{w_y}\to1$. If \eqref{eq:mix} held, this ratio would be at most $Ce^{-\lambda t}<1$ for large $t$.
\end{proof}

In other words, a cell that starts at a very distant phenotype $y$ needs a time of order $\theta^{-1}\ln\abs y$ just to reach the bulk of the distribution; convergence holds for every initial distribution, but not uniformly (Figure~\ref{fig:examples}(d)). More generally, for one-dimensional diffusions a uniform rate in $L^1$ (that is, uniform ergodicity, also called strong ergodicity in \cite{Mao2002,Chen2005}) holds if and only if $\pm\infty$ are entrance boundaries, which under (R1)--(R3) means
\[
\int_0^\infty\frac1{a(y)}\Bigl(\int_y^\infty\psi\dd z\Bigr)\dd y<\infty\qquad\text{and}\qquad\int_{-\infty}^0\frac1{a(y)}\Bigl(\int_{-\infty}^y\psi\dd z\Bigr)\dd y<\infty ,
\]
see \cite{Mao2002}, and \cite{Chen2005} for the wider picture.
The first integral is the supremum over $y>0$ of the mean time needed to reach $0$ from $y$. For the OU dynamics its integrand behaves like $1/(\theta y)$, which reproduces the time $\theta^{-1}\ln\abs y$ found above. The integrals also diverge for the heavy-tailed example~(3) of Section~\ref{subsec:examples} and, with constant $D$, for drifts such as $v(x)=-x\ln(e+\abs x)$, which grow faster than linearly; they converge for strongly confining drifts such as $v(x)=-x^3$, for which cells return from arbitrarily far away within a bounded time. In that case the argument of Theorem~\ref{thm:A} would give a rate in $L^1(\R)$ as well; we do not pursue this.

\subsection{Proof of Theorem~\ref{thm:B}}\label{subsec:proofB}

The results of Section~\ref{sec:nonlinear} hold for any Markov semigroup on an interval, and the semigroup $S$ of Proposition~\ref{prop:SR} is one. Hence the solution exists globally, is nonnegative (part (a)), and Step~1 holds: $U$ is monotone between $U_0$ and $\Ustar$, $U\ge m$, and $\int_0^\infty\abs{U'}<\infty$. Step~2 is unchanged. Step~3 is unchanged up to, but not including, the application of \eqref{eq:mix}, since it used only $S(t)\psi=\psi$ (Proposition~\ref{prop:SR}(i)): we still have
\[
d(t)=S(t)d(0)+\int_0^tS(t-s)\varphi(s)\dd s,\qquad d(0),\varphi(s)\in L^1_0(\R),\qquad\norm{\varphi(s)}\le2\abs{U'(s)} .
\]

\paragraph{Step 4 (without a rate).} By \eqref{eq:E} applied to the zero-mass function $d(0)$, $\norm{S(t)d(0)}\to0$. For the integral,
\[
\Bigl\|\int_0^tS(t-s)\varphi(s)\dd s\Bigr\|\le\int_0^\infty\mathbf 1_{\{s<t\}}\,\norm{S(t-s)\varphi(s)}\dd s .
\]
The right-hand side is an ordinary integral of a nonnegative scalar function of $s$. For each fixed $s$, the integrand tends to $0$ as $t\to\infty$, by \eqref{eq:E} applied to the zero-mass function $\varphi(s)$. It is dominated, uniformly in $t$, by $\norm{\varphi(s)}\le2\abs{U'(s)}$, which is integrable on $[0,\infty)$ by Step~1. By dominated convergence, the integral tends to $0$. Hence $\norm{d(t)}\to0$.

\paragraph{Steps 5 and 6.} Step~5 used only $\norm{d(t)}\to0$, $U\ge m$ and $\psi(J)=\int_J\psi>0$, which holds because $\psi>0$. It goes through unchanged, together with the by-product \eqref{eq:gint}, and gives $U(t)\to\Ustar$, which is part (b). Step~6 gives $\norm{u(t)-\Ustar\psi}\to0$, which is part (c), and, for $\int f=\Ustar$, $\norm{S(t)f-\Ustar\psi}\to0$ by \eqref{eq:E}, which gives part (d). \qed

\begin{remark}[Without confinement]\label{rem:heat}
Consider pure diffusion on $\R$, $v\equiv0$ and constant $D$, for which (R2) fails. Then $S(t)$ is convolution with the heat kernel, and it is classical that $\norm{S(t)w}\to0$ for every $w\in L^1(\R)$ with $\int w=0$ \cite{EscobedoZuazua1991}. If $r\ge\rmin>0$ everywhere, Step~5 is immediate ($\rho\ge\rmin U\ge\rmin m$), and the argument above still applies once $U(t)\psi$ is replaced by $U(t)S(t)h$, for a fixed probability density $h$. Indeed, $U(t)S(t)h=S(t)(U_0h)+\int_0^tS(t-s)\bigl(U'(s)S(s)h\bigr)\dd s$, so $d(t):=u(t)-U(t)S(t)h$ satisfies
\[
d(t)=S(t)(u_0-U_0h)+\int_0^tS(t-s)\bigl(F(s)-U'(s)S(s)h\bigr)\dd s ,
\]
where the forcing has zero mass and norm at most $2\abs{U'(s)}$, and Step~4 (without a rate) gives $\norm{d(t)}\to0$. Hence $U\to\Ustar$ and
\[
\norm{u(t)-S(t)f}\le\norm{d(t)}+\abs{U(t)-\Ustar}+\norm{S(t)(\Ustar h-f)}\longrightarrow0
\]
for every $f\in L^1(\R)$ with $\int f=\Ustar$. The population spreads indefinitely, so there is no limiting profile, but it becomes indistinguishable from a solution of the heat equation: fitness is still asymptotically irrelevant, even without an equilibrium to converge to.
\end{remark}

\subsection{A rate in a weighted norm: proof of Corollary~\ref{cor:rateR}}\label{subsec:rateR}

Let $H:=\{f\in L^1(\R):\norm f_\psi<\infty\}$, where $\norm f_\psi=(\int f^2/\psi\dd x)^{1/2}$ as in Section~\ref{sec:setting}. This is the Hilbert space $L^2(\R;\psi^{-1}\dd x)$, and by the Cauchy--Schwarz inequality and $\int\psi=1$,
\begin{equation}\label{eq:L1H}
\norm f_{L^1(\R)}\le\norm f_\psi\quad(f\in H),\qquad\norm\psi_\psi=1 .
\end{equation}

\begin{lemma}[The switching semigroup on $H$]\label{lem:H}
Assume \textup{(R1)--(R3)}. Then $S(t)H\subset H$, $\norm{S(t)f}_\psi\le\norm f_\psi$, and $t\mapsto S(t)f$ is continuous in $H$ for every $f\in H$. If \eqref{eq:poincare} holds, then
\begin{equation}\label{eq:gapR}
\norm{S(t)f}_\psi\le e^{-\lambda_Pt}\norm f_\psi\qquad\text{for all }t\ge0\text{ and }f\in H\text{ with }\int_\R f\dd x=0 .
\end{equation}
\end{lemma}

\begin{proof}
The map $f\mapsto f/\psi$ is an isometry from $H$ onto $L^2(\mu)$, and $S(t)f=\psi\,T(t)(f/\psi)$ for $f\in H$ (Proposition~\ref{prop:SR}); so the first claims follow from the corresponding properties of the self-adjoint contraction semigroup $T(t)$. Both sides of \eqref{eq:poincare} are continuous for the norm $\bigl(\mathcal E(w,w)+\norm w^2_{L^2(\mu)}\bigr)^{1/2}$, so \eqref{eq:poincare} extends from $C_c^\infty(\R)$ to $\mathcal F$: $\mathcal E(w,w)\ge\lambda_P\norm w^2_{L^2(\mu)}$ for every $w\in\mathcal F$ with $\langle w,\mathbf 1\rangle_\mu=0$. Since $\mathcal A\mathbf 1=0$ and $\mathcal A$ is self-adjoint, the orthogonal complement $\mathbf 1^\perp$ of $\mathbf 1$ in $L^2(\mu)$ is invariant under $\mathcal A$ and $T(t)$, and, by Theorem~\ref{bb:dirichlet}(a), $\langle\mathcal Aw,w\rangle_\mu=\mathcal E(w,w)\ge\lambda_P\norm w^2_{L^2(\mu)}$ for $w\in D(\mathcal A)\cap\mathbf 1^\perp$. Hence the restriction of $\mathcal A$ to $\mathbf 1^\perp$ has its spectrum in $[\lambda_P,\infty)$, and $\norm{T(t)w}_{L^2(\mu)}\le e^{-\lambda_Pt}\norm w_{L^2(\mu)}$ for $w\in\mathbf 1^\perp$, by the spectral theorem. For $f\in H$ with $\int f=0$, $w=f/\psi$ satisfies $\langle w,\mathbf 1\rangle_\mu=\int f\dd x=0$, which gives \eqref{eq:gapR}.
\end{proof}

\begin{proof}[Proof of Corollary~\ref{cor:rateR}]
\emph{The solution stays in $H$.} Fix $T>0$. With $G(t)=g(U(t))$ fixed, $u$ is the unique continuous $L^1$-valued solution on $[0,T]$ of the linear equation \eqref{eq:lin-q} with $q=Gr$ (Lemma~\ref{lem:volterra}). By Lemma~\ref{lem:H}, the proof of Lemma~\ref{lem:volterra} applies verbatim in $H$ (the continuity of $t\mapsto q(\cdot,t)w(t)$ in $H$ follows by dominated convergence, as in Lemma~\ref{lem:Fcont}) and yields a solution in $C([0,T];H)$; by \eqref{eq:L1H} it is also a continuous $L^1$-valued solution, so it coincides with $u$. Taking $\norm\cdot_\psi$ in \eqref{eq:mild},
\[
\norm{u(t)}_\psi\le\norm{u_0}_\psi+\rmax\int_0^t\abs{G(s)}\,\norm{u(s)}_\psi\dd s ,
\]
and Gronwall's inequality and \eqref{eq:gint}, which holds on $\R$ as well (Section~\ref{subsec:proofB}), give
\[
\norm{u(t)}_\psi\le K_2:=\norm{u_0}_\psi\exp\Bigl(\rmax\int_0^\infty\abs{g(U(s))}\dd s\Bigr)<\infty\qquad\text{for all }t\ge0 .
\]

\emph{The deviation.} As in Step~3, $d(t)=u(t)-U(t)\psi$ satisfies \eqref{eq:duhamel-d}. Now $d(0)$ and $\varphi(s)=F(s)-U'(s)\psi$ belong to $H$ and have zero mass, and, since $\abs{U'}=\abs{g(U)}\rho\le\abs{g(U)}\rmax U$ and $U=\norm u_{L^1}\le\norm u_\psi\le K_2$,
\[
\norm{\varphi(s)}_\psi\le\rmax\abs{g(U(s))}\,\norm{u(s)}_\psi+\abs{U'(s)}\,\norm\psi_\psi\le2\rmax K_2\abs{g(U(s))} .
\]
By \eqref{eq:gapR},
\begin{equation}\label{eq:starR}
\norm{d(t)}_\psi\le e^{-\lambda_Pt}\norm{d(0)}_\psi+2\rmax K_2\int_0^te^{-\lambda_P(t-s)}\abs{g(U(s))}\dd s .
\end{equation}
By \eqref{eq:gint} and the argument of Step~4, the right-hand side tends to zero; this part does not use $g'(\Ustar)<0$.

\emph{The rate.} The proof of Corollary~\ref{cor:rate} used only Theorem~\ref{thm:A}(b),(c) and the lower bound on $\rho$ of Step~5(B), which hold on $\R$ by Theorem~\ref{thm:B} and Section~\ref{subsec:proofB}. It gives $\omega>0$ and $K_3$ such that $\abs{z(t)}=\abs{U(t)-\Ustar}\le K_3e^{-\omega t}$ and $\abs{g(U(t))}\le K_3e^{-\omega t}$ for all $t\ge0$ (recall that $\abs{g(U)}\le2\abs{g'(\Ustar)}\abs z$ once $\abs z\le\delta$). Inserting this in \eqref{eq:starR} gives, with $\nu=\min\{\lambda_P,\omega\}$, $\norm{d(t)}_\psi\le K(1+t)e^{-\nu t}$ for some $K$, and the claim follows from $\norm{u(t)-\Ustar\psi}_\psi\le\norm{d(t)}_\psi+\abs{z(t)}$ and \eqref{eq:L1H}.
\end{proof}

\section{A general framework}\label{sec:general}

\subsection{From compartments to Markov semigroups}

We arrived at Theorems~\ref{thm:A} and~\ref{thm:B} by transferring, step by step, the elementary proof of the compartmental result \cite[Appendix~A]{Review} to the PDE \eqref{eq:model}. Along the way it became clear that the argument never uses the form of the switching operator: Sections~\ref{sec:tools} and~\ref{sec:nonlinear} are written for an arbitrary Markov semigroup, and Sections~\ref{sec:proofA} and~\ref{sec:R} use the switching dynamics only through the semigroup $S$ it generates, namely through its positivity and mass conservation, an invariant density, and the fact that $S$ forgets zero-mass perturbations (see the beginning of Section~\ref{sec:R}). We preferred to give the proof in the concrete setting of advection--diffusion, where every step has a transparent meaning and a compartmental counterpart. In this section we state the general result that the same proof gives. Nothing is proved twice: we indicate which parts of Sections~\ref{sec:tools}--\ref{sec:R} apply verbatim, and prove only the statements that need an adaptation. As an application, we treat switching by nonlocal jumps (Section~\ref{subsec:nonlocal}).

\subsection{Setting and main result}

Let $(E,\mathfrak m)$ be a $\sigma$-finite measure space, the phenotype space. We write $L^1=L^1(E,\mathfrak m)$, $\norm\cdot$ for its norm, $\int f$ for $\int_Ef\dd\mathfrak m$, and $L^1_0=\{f\in L^1:\int f=0\}$. Markov semigroups on $L^1$ are defined as in Definition~\ref{def:markov}. The model is
\begin{equation}\label{eq:model-abstract}
u(t)=S(t)u_0+\int_0^tS(t-s)\,g\bigl(U(s)\bigr)r(\cdot,s)\,u(s)\dd s,\qquad U(t)=\int u(t),
\end{equation}
that is, Definition~\ref{def:mild} with a general Markov semigroup $S$ in place of the advection--diffusion semigroup. Here $g$ satisfies (H2), and $r:E\times[0,\infty)\to[0,\rmax]$ is measurable, with $t\mapsto r(x,t)$ continuous for $\mathfrak m$-a.e.\ $x$. The switching dynamics enters only through the following two conditions.

\begin{enumerate}[label=\textbf{(M\arabic*)}, leftmargin=3em, itemsep=4pt]
\item \emph{Forgetting.} $\norm{S(t)w}\to0$ as $t\to\infty$, for every $w\in L^1_0$.
\item \emph{Persistent proliferation.} There are a probability density $h\in L^1$, a constant $\delta>0$ and a measurable set $\mathcal T\subset[0,\infty)$ of infinite Lebesgue measure such that $\int r(\cdot,t)\,S(t)h\ge\delta$ for all $t\in\mathcal T$.
\end{enumerate}

Condition (M1) says that the switching dynamics forgets its initial state; when $S$ has an invariant density, it is the \emph{asymptotic stability} of $S$, for which criteria are given in \cite{LasotaMackey1994,PichorRudnicki2000}. Under (M1), the choice of $h$ in (M2) is immaterial for large times, since $\norm{S(t)h-S(t)h'}\to0$ for any two probability densities $h,h'$. If $S$ has an invariant probability density $\psi$, that is, $S(t)\psi=\psi$ for all $t\ge0$, then (M1) is equivalent to \eqref{eq:E}, because $S(t)f-(\int f)\psi=S(t)\bigl(f-(\int f)\psi\bigr)$, and $\psi$ is unique. One can then take $h=\psi$ in (M2), which asks that the mean proliferation rate at the stationary composition, $\int r(\cdot,t)\psi$, be bounded below on a set of times of infinite measure (see also Section~\ref{sec:disc}). When $\psi>0$ a.e., we use the weighted norm $\norm f_\psi=(\int f^2/\psi)^{1/2}$ of Section~\ref{sec:setting}.

\begin{mainthm}[General switching dynamics]\label{thm:C}
Let $S$ be a Markov semigroup on $L^1$, and let $g$ and $r$ be as above. Assume \textup{(M1)} and \textup{(M2)}, and let $u_0\in L^1$ with $u_0\ge0$ and $U_0=\int u_0>0$. Then:
\begin{enumerate}[label=(\alph*), itemsep=2pt]
\item \eqref{eq:model-abstract} has a unique global mild solution $u\in C([0,\infty);L^1)$, and it is nonnegative;
\item $U$ is continuously differentiable and monotone, and $U(t)\to\Ustar$;
\item $\norm{u(t)-S(t)f}\to0$ for every $f\in L^1$ with $\int f=\Ustar$.
\end{enumerate}
If, in addition, $S$ has an invariant probability density $\psi$, then:
\begin{enumerate}[label=(\alph*), resume, itemsep=2pt]
\item $\norm{u(t)-\Ustar\psi}\to0$;
\item if $\norm{S(t)w}\le Ce^{-\lambda t}\norm w$ for all $t\ge0$ and $w\in L^1_0$, then \eqref{eq:star} holds; if, moreover, $g'(\Ustar)<0$ and $\mathcal T$ contains a half-line, then $\norm{u(t)-\Ustar\psi}\le K(1+t)e^{-\nu t}$ for all $t\ge0$, for some $K,\nu>0$;
\item if $\psi>0$ a.e., $g'(\Ustar)<0$, $\mathcal T$ contains a half-line, $\norm{u_0}_\psi<\infty$, and there are $C_P\ge1$ and $\lambda_P>0$ such that $\norm{S(t)f}_\psi\le C_Pe^{-\lambda_Pt}\norm f_\psi$ for all $t\ge0$ and all $f\in L^1_0$ with $\norm f_\psi<\infty$, then $\norm{u(t)-\Ustar\psi}\le\norm{u(t)-\Ustar\psi}_\psi\le K(1+t)e^{-\nu t}$ for all $t\ge0$, for some $K,\nu>0$.
\end{enumerate}
\end{mainthm}

Theorems~\ref{thm:A} and~\ref{thm:B}, Corollaries~\ref{cor:rate} and~\ref{cor:rateR}, and Remark~\ref{rem:heat} are special cases (Section~\ref{subsec:examples-general}). In part~(f), when $S$ is self-adjoint in the weighted space, as in Section~\ref{sec:R}, the decay assumption with $C_P=1$ amounts to a spectral gap of the generator, that is, to the Poincar\'e inequality \eqref{eq:poincare}; for non-reversible switching dynamics it has to be verified by other means, for instance by hypocoercivity methods \cite{Villani2009,DolbeaultMouhotSchmeiser2015}.

\subsection{Proof of Theorem~\ref{thm:C}}\label{subsec:proofC}

\emph{What applies verbatim.} Section~\ref{sec:tools} and Appendix~\ref{app:integrals} use only the facts that $L^1$ is a Banach lattice and that $S$ satisfies (S1)--(S4), and Section~\ref{sec:nonlinear} and Appendix~\ref{app:mild} were written for an arbitrary Markov semigroup; all of them apply with $\Lone$ replaced by $L^1$. The continuity of $r$ entered only in Lemma~\ref{lem:Fcont}, whose dominated convergence argument needs only the continuity of $t\mapsto r(x,t)$ for a.e.\ $x$. This gives (a), and also Proposition~\ref{prop:global}, that is, Step~1 of Section~\ref{sec:proofA}: $U$ is $C^1$ and monotone, $U\ge m=\min\{U_0,\Ustar\}$, and $\int_0^\infty\abs{U'}<\infty$. Step~2 is unchanged. Steps~3--6 used the invariant density $\psi$; without it, we replace $U(t)\psi$ by $U(t)S(t)h$, as in Remark~\ref{rem:heat}.

\emph{Steps 3--4: the deviation from the linear flow vanishes.} Let $h$ be the density in (M2), and $d(t):=u(t)-U(t)S(t)h$. By the fundamental theorem of calculus and the semigroup property, $U(t)S(t)h=S(t)(U_0h)+\int_0^tS(t-s)\bigl(U'(s)S(s)h\bigr)\dd s$; subtracting this identity from \eqref{eq:model-abstract} gives
\begin{equation}\label{eq:duhamel-dh}
d(t)=S(t)(u_0-U_0h)+\int_0^tS(t-s)\,\varphi_h(s)\dd s,\qquad\varphi_h(s):=F(s)-U'(s)S(s)h .
\end{equation}
By \eqref{eq:F} and $\int S(s)h=1$, $\varphi_h(s)\in L^1_0$ and $\norm{\varphi_h(s)}\le2\abs{U'(s)}$, while $u_0-U_0h\in L^1_0$. The dominated convergence argument of Step~4 in Section~\ref{subsec:proofB}, with (M1) in place of \eqref{eq:E}, gives $\norm{d(t)}\to0$.

\emph{Step 5: $U(t)\to\Ustar$.} Part~(A) of Step~5 is unchanged. For part~(B), since $r\ge0$, $u(t)=U(t)S(t)h+d(t)$ and $U\ge m$,
\[
\rho(t)=\int r(\cdot,t)u(t)\ \ge\ m\int r(\cdot,t)S(t)h-\rmax\norm{d(t)}\ \ge\ m\delta-\rmax\norm{d(t)}\qquad(t\in\mathcal T).
\]
If $T$ is such that $\rmax\norm{d(t)}\le m\delta/2$ for $t\ge T$, then $\rho\ge m\delta/2$ on $\mathcal T\cap[T,\infty)$, a set of infinite measure, so $\int_0^\infty\rho\dd t=\infty$. As in Section~\ref{sec:proofA}, this contradicts part~(A) unless $U_\infty=\Ustar$, which proves (b).

\emph{Step 6.} Let $f\in L^1$ with $\int f=\Ustar$. Then $\Ustar h-f\in L^1_0$ and $\norm{S(t)h}=1$, so
\[
\norm{u(t)-S(t)f}\le\norm{d(t)}+\abs{U(t)-\Ustar}+\norm{S(t)(\Ustar h-f)}\longrightarrow0
\]
by Steps~3--5 and (M1). This is (c), and (d) is (c) with $f=\Ustar\psi$.

\emph{Rates.} Let $\psi$ be invariant. With $h=\psi$, \eqref{eq:duhamel-dh} is \eqref{eq:duhamel-d}, and the derivation of \eqref{eq:star} in Step~3 used only the invariance of $\psi$ and the mixing estimate; this is the first part of (e). If $\mathcal T\supset[t_*,\infty)$, then, by (M1), $\int r(\cdot,t)\psi\ge\int r(\cdot,t)S(t)h-\rmax\norm{S(t)(h-\psi)}\ge\delta/2$ for all large $t$. Hence, with $d=u-U\psi$ and $U(t)\to\Ustar$, $\rho(t)\ge U(t)\delta/2-\rmax\norm{d(t)}\ge\Ustar\delta/8$ for all large $t$. With this lower bound in place of $\rho\ge c_*$, the proof of Corollary~\ref{cor:rate} gives the second part of (e); the same lower bound gives \eqref{eq:gint}. For (f), the proof of Corollary~\ref{cor:rateR} in Section~\ref{subsec:rateR} applies verbatim, with Theorem~\ref{thm:C}(b),(d) in place of Theorem~\ref{thm:B}, with Lemma~\ref{lem:weighted} below and the decay assumed in (f) in place of Lemma~\ref{lem:H}, and with an extra factor $C_P$ in \eqref{eq:starR}.

The only new ingredient is the following lemma, which replaces the self-adjointness used in Lemma~\ref{lem:H}.

\begin{lemma}[The switching semigroup in the weighted space]\label{lem:weighted}
Let $S$ be a Markov semigroup on $L^1$ with an invariant probability density $\psi>0$ a.e., and let $H:=\{f\in L^1:\norm f_\psi<\infty\}$. Then $\norm f\le\norm f_\psi$ for $f\in H$; $S(t)H\subset H$ and $\norm{S(t)f}_\psi\le\norm f_\psi$ for $f\in H$; and $t\mapsto S(t)f$ is continuous in $H$ for every $f\in H$.
\end{lemma}

The heart of the lemma is the pointwise inequality $(S(t)f)^2\le\psi\,S(t)(f^2/\psi)$. When $S(t)$ has a kernel $\Gamma$, it is the Cauchy--Schwarz inequality $\bigl(\int\Gamma f\bigr)^2\le\int\Gamma\psi\cdot\int\Gamma f^2/\psi$, together with $\int\Gamma\psi=\psi$; the proof below uses only (S1), (S2) and the invariance of $\psi$.

\begin{proof}
The first inequality is the Cauchy--Schwarz inequality, as in \eqref{eq:L1H}. Let $f\in H$ and $t\ge0$. For every $\theta\in\R$, $(f-\theta\psi)^2/\psi=f^2/\psi-2\theta f+\theta^2\psi$ is a nonnegative function in $L^1$, so, by (S1) and $S(t)\psi=\psi$,
\[
S(t)\bigl(f^2/\psi\bigr)-2\theta\,S(t)f+\theta^2\psi\ \ge\ 0\qquad\text{a.e.}
\]
This holds a.e.\ simultaneously for all rational $\theta$, hence, by continuity in $\theta$, for all real $\theta$, and the discriminant of this quadratic polynomial in $\theta$ is nonpositive: $(S(t)f)^2\le\psi\,S(t)(f^2/\psi)$ a.e. Dividing by $\psi$, integrating and using (S2), $\norm{S(t)f}_\psi^2\le\int f^2/\psi=\norm f_\psi^2$. For the continuity, $H$ is a Hilbert space for the scalar product $\langle f,g\rangle_\psi=\int fg/\psi$, in which the functions $g$ with $g/\psi\in L^\infty$ are dense. For such $g$, $\langle S(t)f,g\rangle_\psi\to\langle f,g\rangle_\psi$ as $t\downarrow0$, because $S(t)f\to f$ in $L^1$. Since $\norm{S(t)f}_\psi\le\norm f_\psi$, it follows that $S(t)f\to f$ weakly in $H$, and then $\norm f_\psi\le\liminf_{t\downarrow0}\norm{S(t)f}_\psi\le\norm f_\psi$ gives convergence in norm. Continuity at $t>0$ follows as in Lemma~\ref{lem:markov}(iii).
\end{proof}

\subsection{Doeblin's condition and first examples}\label{subsec:examples-general}

The next lemma is Lemma~\ref{lem:mix} in abstract form. It provides (M1), with an exponential rate, together with an invariant density. Doeblin's condition and its extension by Harris to unbounded state spaces are classical tools for the ergodicity of Markov processes; see \cite{HairerMattingly2011} for a short proof of Harris' theorem in this spirit, and \cite{CanizoMischler2023} for its version for stochastic semigroups on $L^1$.

\begin{lemma}[Doeblin's condition]\label{lem:doeblin}
Let $S$ be a Markov semigroup on $L^1$, and suppose that there are $\tau>0$, $c>0$ and a probability density $\phi$ such that
\begin{equation}\label{eq:doeblin}
S(\tau)f\ \ge\ c\Bigl(\int f\Bigr)\phi\qquad\text{for every }f\in L^1\text{ with }f\ge0 .
\end{equation}
Then $c\le1$ and, with $\bar c:=\min\{c,1/2\}$:
\begin{enumerate}[label=(\roman*), itemsep=2pt]
\item $\norm{S(t)w}\le Ce^{-\lambda t}\norm w$ for all $t\ge0$ and $w\in L^1_0$, where $C=1/(1-\bar c)$ and $\lambda=-\tau^{-1}\ln(1-\bar c)\ge\bar c/\tau$;
\item $S$ has a unique invariant probability density $\psi$, and $\psi\ge c\,\phi$.
\end{enumerate}
\end{lemma}

\begin{proof}
Integrating \eqref{eq:doeblin} and using (S2) gives $c\le1$. Part~(i) is proved exactly as Lemma~\ref{lem:mix}(i)--(ii): the positive and negative parts of $w\in L^1_0$ satisfy $S(\tau)w^\pm\ge\frac{\bar c}2\norm w\,\phi$, the floor $\frac{\bar c}2\norm w\,\phi$ has integral $\frac{\bar c}2\norm w$, and therefore $\norm{S(\tau)w}\le(1-\bar c)\norm w$; iterating gives (i). (Lemma~\ref{lem:mix} is the case $\phi=1/\ell$, with $\ell\bar c$ in the role of $\bar c$.) For (ii), the set of probability densities is closed in $L^1$ and invariant under $S(\tau)$, and $\norm{S(\tau)f-S(\tau)f'}\le(1-\bar c)\norm{f-f'}$ for any two of them, by (i). By the Banach fixed-point theorem, $S(\tau)$ has a unique fixed point $\psi$ among probability densities. For $t\ge0$, $S(t)\psi$ is a probability density and $S(\tau)S(t)\psi=S(t)S(\tau)\psi=S(t)\psi$, so $S(t)\psi=\psi$; conversely, every invariant probability density is a fixed point of $S(\tau)$. Finally, $\psi=S(\tau)\psi\ge c\,\phi$.
\end{proof}

\paragraph{Compartmental models.} Let $E=\{1,\dots,n\}$ with the counting measure, so that $L^1=\R^n$ with the norm $\abs\cdot_1$, and let $S(t)=e^{tA}$, where $A$ is a transition-rate matrix, with nonnegative off-diagonal entries and columns summing to zero. Then $S$ is a Markov semigroup: $e^{tA}=e^{-ta}e^{t(A+aI)}$ has nonnegative entries if $a\ge\max_i\abs{A_{ii}}$, and $(1,\dots,1)A=0$ gives mass conservation. If $A$ is irreducible, all entries of $e^{\tau A}$ are bounded below by some $\epsilon>0$ \cite[Lemma~A.1]{Review}, which is \eqref{eq:doeblin} with $\phi=(1/n,\dots,1/n)$ and $c=n\epsilon$. Lemma~\ref{lem:doeblin} gives a stationary distribution $\pi\ge\epsilon$ and an exponential rate, and (M2) holds with $h=\pi$, $\delta=\rmin\pi_s$ and $\mathcal T=[t_r,\infty)$ if some compartment $s$ has $r_s(t)\ge\rmin$ for all $t\ge t_r$. Irreducibility can be weakened: if the states contain a single closed communicating class $\mathcal C$, reachable from every state, then all entries of $e^{\tau A}$ in the rows of $\mathcal C$ are positive, which is \eqref{eq:doeblin} with $\phi$ uniform on $\mathcal C$; the stationary distribution vanishes outside $\mathcal C$, and (M2) holds if the persistently proliferating compartment belongs to $\mathcal C$. Theorem~\ref{thm:C} thus contains the compartmental result of \cite[Appendix~A]{Review} (see also \cite{Giaimo2025}), including the estimate \eqref{eq:star}. (In \cite{Review} the rates are allowed to be piecewise continuous in time; see Section~\ref{sec:disc}.)

\paragraph{Advection--diffusion.} On a bounded interval, (M1) with an exponential rate is Lemma~\ref{lem:mix} or Lemma~\ref{lem:gap}, and (M2) follows from (H4), with $h=\psi$, $\delta=\rmin\psi(J)$ and $\mathcal T=[t_r,\infty)$; Theorem~\ref{thm:C} gives Theorem~\ref{thm:A} and Corollary~\ref{cor:rate}. On $\R$, (M1) is Proposition~\ref{prop:SR}(ii), and Theorem~\ref{thm:C} gives Theorem~\ref{thm:B} and, by part~(f) and Lemma~\ref{lem:H}, Corollary~\ref{cor:rateR}. Without confinement (Remark~\ref{rem:heat}) there is no invariant density, (M1) holds by \cite{EscobedoZuazua1991}, and (M2) holds for every $h$ if $r\ge\rmin$ everywhere; parts (a)--(c) of Theorem~\ref{thm:C} are the content of that remark. Theorem~\ref{thm:C} applies in the same way to any other switching dynamics once (M1) is established for its semigroup; we treat nonlocal switching in Section~\ref{subsec:nonlocal}, and no-flux diffusions in bounded domains of $\R^d$ in Section~\ref{subsec:Rd}.

\subsection{Nonlocal switching}\label{subsec:nonlocal}

Phenotypic switching can also occur by jumps rather than by small continuous changes. Let $\Omega\subseteq\R^d$ be measurable, with the Lebesgue measure, and consider
\begin{equation}\label{eq:nonlocal}
\partial_tu(x,t)=g\bigl(U(t)\bigr)\,r(x,t)\,u(x,t)+\int_\Omega\bigl[K(x,y)\,u(y,t)-K(y,x)\,u(x,t)\bigr]\dd y ,
\end{equation}
where $K:\Omega\times\Omega\to[0,\infty)$ is measurable and $K(x,y)$ is the rate density at which cells of phenotype $y$ switch to phenotype $x$. The switching term is the continuum counterpart of a transition-rate matrix: cells of phenotype $x$ are gained from every other phenotype, and lost at the total rate
\[
\kappa(x):=\int_\Omega K(y,x)\dd y .
\]
We assume:
\begin{enumerate}[label=\textbf{(N\arabic*)}, leftmargin=3em, itemsep=4pt]
\item \emph{Bounded switching rate.} $\kappa(x)\le\kappa_{\max}$ for a.e.\ $x\in\Omega$.
\item \emph{Connectivity.} There are an integer $n\ge1$, a constant $k_0>0$ and a probability density $\phi$ on $\Omega$ such that
\[
\sum_{j=1}^nK_j(x,y)\ \ge\ k_0\,\phi(x)\qquad\text{for a.e. }(x,y)\in\Omega\times\Omega,
\]
where $K_1=K$ and $K_{j+1}(x,y)=\int_\Omega K(x,z)K_j(z,y)\dd z$.
\end{enumerate}
Condition (N2) is the counterpart of irreducibility: $K_j(x,y)$ is the rate density of going from $y$ to $x$ through a chain of $j$ jumps, and (N2) asks that, within $n$ jumps, every phenotype can reach the region where $\phi>0$. Kernels of this type describe mutations in continuum-of-alleles and selection--mutation models \cite{Burger2000,Perthame2007}.

\begin{proposition}[The nonlocal switching semigroup]\label{prop:nonlocal}
Assume \textup{(N1)}, and let $\mathcal Ku:=\int_\Omega K(\cdot,y)u(y)\dd y-\kappa u$. Then $\mathcal K$ is a bounded operator on $L^1(\Omega)$, and $S(t)=e^{t\mathcal K}$ is a Markov semigroup. For $u_0\in L^1(\Omega)$ and $F\in C([0,T];L^1(\Omega))$, $u(t)=S(t)u_0+\int_0^tS(t-s)F(s)\dd s$ for $t\in[0,T]$ if and only if $u\in C^1([0,T];L^1(\Omega))$, $u'=\mathcal Ku+F$ and $u(0)=u_0$. If \textup{(N2)} also holds, then, for every $\tau\in(0,1]$, $S$ satisfies Doeblin's condition \eqref{eq:doeblin} with this $\phi$ and
\[
c=k_0\,e^{-\tau\kappa_{\max}}\,\frac{\tau^n}{n!} .
\]
\end{proposition}

\begin{proof}
Let $Bu:=\int_\Omega K(\cdot,y)u(y)\dd y$. By Tonelli's theorem and (N1),
\[
\norm{Bu}\le\int_\Omega\int_\Omega K(x,y)\abs{u(y)}\dd y\dd x=\int_\Omega\kappa\abs u\dd y\le\kappa_{\max}\norm u ,
\]
so $B$, and hence $\mathcal K=B-\kappa$, is bounded, and $S(t)=\sum_{j\ge0}t^j\mathcal K^j/j!$ is a group of bounded operators, continuous in operator norm; this gives (S3) and (S4). By Fubini's theorem, $\int Bu=\int\kappa u$, so $\int\mathcal Ku=0$ for every $u$, and $\int S(t)u=\int u$ follows from the series; this is (S2). For (S1), write $S(t)=e^{-t\kappa_{\max}}e^{tM}$ with $M:=B+(\kappa_{\max}-\kappa)$. The operator $M$ is positive, since $B$ has a nonnegative kernel and $\kappa_{\max}-\kappa\ge0$, so every term of $e^{tM}=\sum_jt^jM^j/j!$ is positive. The equivalence between mild and $C^1$ solutions holds because $S$ is a group, differentiable in operator norm: if $u$ is mild, then $u(t)=S(t)\bigl(u_0+\int_0^tS(-s)F(s)\dd s\bigr)$ is $C^1$ with $u'=\mathcal Ku+F$; conversely, if $u$ is a $C^1$ solution, then $\frac{\dd}{\dd s}\bigl(S(t-s)u(s)\bigr)=S(t-s)F(s)$, and integrating over $[0,t]$ gives the mild formula. Finally, assume (N2), and let $f\ge0$ and $\tau\in(0,1]$. Since $M\ge B\ge0$, we have $M^j\ge B^j$, and $B^j$ is the integral operator with kernel $K_j$. Since $\tau^j/j!\ge\tau^n/n!$ for $1\le j\le n$,
\[
S(\tau)f\ \ge\ e^{-\tau\kappa_{\max}}\sum_{j=1}^n\frac{\tau^j}{j!}B^jf\ \ge\ e^{-\tau\kappa_{\max}}\frac{\tau^n}{n!}\int_\Omega\sum_{j=1}^nK_j(\cdot,y)f(y)\dd y\ \ge\ e^{-\tau\kappa_{\max}}\frac{\tau^n}{n!}\,k_0\Bigl(\int f\Bigr)\phi .
\]
\end{proof}

\begin{maincor}[Nonlocal switching]\label{cor:nonlocal}
Assume \textup{(H2)}, \textup{(N1)} and \textup{(N2)}. Let $r:\Omega\times[0,\infty)\to[0,\rmax]$ be measurable, with $t\mapsto r(x,t)$ continuous for a.e.\ $x$, and $r(x,t)\ge\rmin>0$ for all $t\ge t_r$ (for some $t_r\ge0$) and all $x$ in a set $J$ with $\int_J\phi>0$. Let $\psi$ be the unique probability density such that
\[
\int_\Omega K(x,y)\psi(y)\dd y=\kappa(x)\psi(x)\qquad\text{for a.e. }x\in\Omega .
\]
Then, for every $u_0\in L^1(\Omega)$ with $u_0\ge0$ and $U_0>0$, \eqref{eq:nonlocal} has a unique global solution $u\in C^1([0,\infty);L^1(\Omega))$. It is nonnegative, $U$ is monotone, $U(t)\to\Ustar$, $\norm{u(t)-\Ustar\psi}_{L^1}\to0$, and $\norm{u(t)-S(t)f}_{L^1}\to0$ for every $f\in L^1(\Omega)$ with $\int f=\Ustar$. Moreover, \eqref{eq:star} holds with
\[
C=\frac1{1-\bar c},\qquad\lambda=-\frac1\tau\ln(1-\bar c),\qquad\bar c=\min\Bigl\{\frac12,\ k_0e^{-\tau\kappa_{\max}}\frac{\tau^n}{n!}\Bigr\},
\]
for any $\tau\in(0,1]$, and, if $g'(\Ustar)<0$, the convergence to $\Ustar\psi$ is exponential.
\end{maincor}

\begin{proof}
By Proposition~\ref{prop:nonlocal} and Lemma~\ref{lem:doeblin}, $S$ is a Markov semigroup with a unique invariant probability density $\psi$, and it satisfies (M1) with the stated constants. Since the generator $\mathcal K$ is bounded, a probability density is invariant if and only if $\mathcal K\psi=0$, which is the equation in the statement. Moreover $\psi\ge c\,\phi$, so $\int r(\cdot,t)\psi\ge\rmin\int_J\psi\ge\rmin c\int_J\phi>0$ for all $t\ge t_r$; this is (M2) with $h=\psi$ and $\mathcal T=[t_r,\infty)$. Theorem~\ref{thm:C}(a)--(e) gives the conclusions, and the mild solution is $C^1$ by Proposition~\ref{prop:nonlocal}, since the reaction term is continuous in time (Lemma~\ref{lem:Fcont}).
\end{proof}

\begin{remark}[Examples of nonlocal switching]\label{rem:nonlocal-examples}
In general the stationary density $\psi$ is not explicit, but the conclusion does not depend on it: whatever $\psi$ is, it does not depend on $r$. Three examples:
\begin{enumerate}[label=(\arabic*), itemsep=3pt]
\item \emph{Independent jumps.} $K(x,y)=k(y)\phi(x)$, with $k$ measurable and $0<k_{\min}\le k\le k_{\max}$: a cell of phenotype $y$ switches at rate $k(y)$, and its new phenotype is drawn from $\phi$, independently of $y$. Then $\kappa=k$, (N1)--(N2) hold with $n=1$ and $k_0=k_{\min}$, and the stationarity equation $k\psi=\phi\int k\psi$ gives $\psi=(\phi/k)/\int(\phi/k)$: each phenotype is represented in proportion to the time that cells spend in it. This example allows unbounded $\Omega$, for instance $\Omega=\R^d$, and, in contrast with the Ornstein--Uhlenbeck dynamics (Proposition~\ref{prop:OUnorate}), the rate is uniform, because a cell returns from an arbitrarily distant phenotype in a single jump.
\item \emph{Detailed balance.} If $K(x,y)\pi(y)=K(y,x)\pi(x)$ for a probability density $\pi$, then $\mathcal K\pi=0$, so $\psi=\pi$. For instance, let $\Omega$ be bounded, let $\pi$ be a probability density with $0<\pi_{\min}\le\pi\le\pi_{\max}$ on $\Omega$, and let $q\ge0$ be even and integrable on $\R^d$, with $q\ge q_0>0$ on $\{x-y:x,y\in\Omega\}$. The Metropolis kernel $K(x,y)=q(x-y)\min\{1,\pi(x)/\pi(y)\}$ satisfies $K(x,y)\pi(y)=q(x-y)\min\{\pi(x),\pi(y)\}=K(y,x)\pi(x)$, $\kappa\le\norm q_{L^1}$ and $K\ge q_0\pi_{\min}/\pi_{\max}$, so (N1)--(N2) hold with $n=1$, $\phi=1/\abs\Omega$ and $k_0=\abs\Omega q_0\pi_{\min}/\pi_{\max}$. With $\pi\propto e^{-P/D}$, jumps in the landscape $P$ lead to the same Boltzmann-type limit as the local switching dynamics with $v=-P'$ and constant $D$ (Corollary~\ref{cor:landscape} and Figure~\ref{fig:examples}(c)); for a symmetric kernel, $K(x,y)=K(y,x)$, $\psi$ is uniform.
\item \emph{Short jumps.} If $\Omega=(0,\ell)$ and $K(x,y)\ge k_1>0$ whenever $\abs{x-y}<\ell_0$, then (N2) holds with any integer $n>2\ell/\ell_0$, $\phi=1/\ell$ and $k_0=\ell\,k_1^n\min\{\ell_0/4,\ell\}^{n-1}$. Indeed, for $x,y\in\Omega$ let $p_i=x+i(y-x)/n$, so that $\abs{p_i-p_{i-1}}<\ell_0/2$. If $z_0=x$, $z_n=y$ and $z_i\in I_i:=(p_i-\ell_0/4,p_i+\ell_0/4)\cap\Omega$ for $1\le i\le n-1$, then $\abs{z_i-z_{i-1}}<\ell_0$ for all $i$, and each $I_i$ has length at least $\min\{\ell_0/4,\ell\}$; hence $K_n(x,y)\ge k_1^n\min\{\ell_0/4,\ell\}^{n-1}$. As for irreducible matrices, connectivity comes from chains of transitions.
\end{enumerate}
\end{remark}

\begin{remark}[Local and nonlocal switching combined]\label{rem:unified}
The unified model of \cite{Review} combines both mechanisms. On $\Omega=(0,\ell)$, its switching generator is $\mathcal L+\varepsilon\mathcal K$, with $\mathcal L$ as in (H3), $\mathcal K$ as in Proposition~\ref{prop:nonlocal} under (N1), and $\varepsilon\ge0$. Since $\varepsilon\mathcal K$ is a bounded perturbation of the generator of the semigroup $S$ of Section~\ref{sec:linear}, the sum generates a semigroup $S_\varepsilon$ on $\Lone$, given by the Dyson--Phillips series \cite[Section~3.1]{Pazy1983}. Write $\varepsilon\mathcal K=\varepsilon M-\varepsilon\kappa_{\max}$, with the positive operator $M$ of the proof of Proposition~\ref{prop:nonlocal}. Then every term of the Dyson--Phillips series of $e^{\varepsilon\kappa_{\max}t}S_\varepsilon(t)$ is a positive operator, and the first one is $S(t)$; hence $S_\varepsilon(t)\ge e^{-\varepsilon\kappa_{\max}t}S(t)$, and, with $c$ as in Lemma~\ref{lem:mix}, $S_\varepsilon(\tau)f\ge e^{-\varepsilon\kappa_{\max}\tau}c\int f$ for $f\ge0$. Mass is conserved, because $\int\mathcal Ku=0$. So $S_\varepsilon$ is a Markov semigroup that satisfies Doeblin's condition \eqref{eq:doeblin} with $\phi=1/\ell$, with invariant density $\psi_\varepsilon\ge e^{-\varepsilon\kappa_{\max}\tau}c$ (Lemma~\ref{lem:doeblin}). Hence (M2) follows from (H4), and Theorem~\ref{thm:C}(a)--(e) applies, without (N2): under (H1), (H2) and (H4), the solution converges to $\Ustar\psi_\varepsilon$, which again does not depend on $r$. If the jumps satisfy detailed balance with respect to $\psi$, as for the Metropolis kernel of Remark~\ref{rem:nonlocal-examples}(2) with $\pi=\psi$, then $\mathcal K\psi=0$ and $\psi_\varepsilon=\psi$ for every $\varepsilon$: the jumps then affect how fast fitness is forgotten, but not the limit.
\end{remark}

\subsection{Bounded domains in \texorpdfstring{$\R^d$}{Rd} and non-gradient drifts}\label{subsec:Rd}

In one dimension, the stationary density of an advection--diffusion equation is explicit, $\psi\propto\exp(\int v/D)$, and its flux vanishes. In higher dimension this is no longer so: drifts that are not gradients, relative to the diffusion, break detailed balance and sustain circulating probability fluxes, and the stationary density is in general not explicit. Theorem~\ref{thm:C} applies nonetheless, as we now show for no-flux diffusions in bounded domains.

Let $\Omega\subset\R^d$ be a bounded domain (open and connected) with boundary of class $C^{2,\alpha}$, $0<\alpha<1$, and outward unit normal $n$. Let $D:\bar\Omega\to\R^{d\times d}$ be symmetric, with entries in $C^{1,\alpha}(\bar\Omega)$ and $\xi\cdot D(x)\xi\ge D_{\min}\abs\xi^2$ for some $D_{\min}>0$ and all $x\in\bar\Omega$, $\xi\in\R^d$, and let $v\in C^{1,\alpha}(\bar\Omega;\R^d)$. The switching dynamics is
\begin{equation}\label{eq:Rd}
\partial_tu=\mathcal Lu:=\nabla\cdot\bigl(D\nabla u-v\,u\bigr)\quad\text{in }\Omega,\qquad\bigl(D\nabla u-v\,u\bigr)\cdot n=0\quad\text{on }\partial\Omega .
\end{equation}

\begin{blackbox}[Green function of the no-flux problem in $\Omega$]\label{bb:greenRd}
Under these assumptions there is a continuous function $\Gamma:\bar\Omega\times\bar\Omega\times(0,\infty)\to\R$ with properties (a)--(c) of Theorem~\ref{bb:green}, with $[0,\ell]$ replaced by $\bar\Omega$, $\partial_xu$ and $\partial_{xx}u$ by the first and second derivatives of $u$, and the no-flux condition by the one in \eqref{eq:Rd}.
\end{blackbox}

\noindent\emph{Comments and references.} In non-divergence form, $\mathcal Lu=D:\nabla^2u+(\nabla\cdot D-v)\cdot\nabla u-(\nabla\cdot v)\,u$, where $(\nabla\cdot D)_i=\sum_j\partial_jD_{ij}$; all coefficients are H\"older continuous, and the boundary condition, $D\nabla u\cdot n=(v\cdot n)\,u$, prescribes the conormal derivative $D\nabla u\cdot n$, a derivative in the direction $Dn$, which is oblique since $Dn\cdot n\ge D_{\min}>0$, together with a zeroth-order term. In the terminology of \cite[Ch.~2, Sec.~5]{Friedman1964}, where the zeroth-order term is allowed, this is the second initial-boundary value problem; in \cite[Ch.~IV]{LSU1968} it is the problem with directional derivative (5.4). Existence of classical solutions is proved with single-layer potentials in \cite[Ch.~5, Sec.~3]{Friedman1964} (see also \cite[Ch.~6, Sec.~5]{Friedman1964}); for smooth initial data that satisfy the boundary condition, existence and uniqueness in H\"older classes is \cite[Ch.~IV, Thm.~5.3]{LSU1968}, whose hypotheses, with $l=\alpha$, are those made above; and the Green function was constructed by It\^o \cite{Ito1957} (see \cite[p.~336]{Friedman1964}). Uniqueness in the class described in Theorem~\ref{bb:green}(a) follows from \cite[Ch.~2, Thm.~15]{Friedman1964}, after the substitution $u=e^{\chi}z$, with $\chi\in C^2(\bar\Omega)$ and $\nabla\chi=Kn$ on $\partial\Omega$ for a large constant $K$, which gives the zeroth-order term of the boundary condition the sign required there.
Positivity follows as in the one-dimensional case, from \cite[Ch.~2, Thms.~5 and~14]{Friedman1964}: a nonnegative, nonzero solution cannot vanish at an interior point at a positive time, and if it vanished at a boundary point, the boundary point lemma would give $D\nabla u\cdot n<0$ there, whereas the boundary condition gives $D\nabla u\cdot n=(v\cdot n)\,u=0$. The zeroth-order coefficient $-\nabla\cdot v$, which has no sign, is harmless: the strong maximum principle of \cite[Ch.~2, Thm.~5]{Friedman1964} does not require one, and the boundary point lemma is applied to $e^{-kt}u$, which satisfies an equation whose zeroth-order coefficient is nonpositive for $k$ large. Chapman--Kolmogorov follows from uniqueness.

With Theorem~\ref{bb:greenRd} in place of Theorem~\ref{bb:green}, the proofs of Lemma~\ref{lem:mass-lin} and Proposition~\ref{prop:S-markov} apply with the obvious changes (the integration over $[a,b]$ being replaced by the divergence theorem on smooth subdomains exhausting $\Omega$), and the operators $S(t)f=\int_\Omega\Gamma(\cdot,y,t)f(y)\dd y$ form a Markov semigroup on $\Lone$. Since $\Gamma(\cdot,\cdot,\tau)$ is continuous and positive on the compact set $\bar\Omega\times\bar\Omega$, $c:=\min\Gamma(\cdot,\cdot,\tau)>0$, and $S(\tau)f\ge c\int f=c\abs\Omega\,\bigl(\int f\bigr)\abs\Omega^{-1}$ for $f\ge0$: this is Doeblin's condition \eqref{eq:doeblin} with $\phi=1/\abs\Omega$ and with the constant $c\abs\Omega$ in place of $c$.

\begin{maincor}[Bounded domains in $\R^d$]\label{cor:Rd}
Assume \textup{(H2)} and the above conditions on $\Omega$, $D$ and $v$, and let $r:\bar\Omega\times[0,\infty)\to[0,\rmax]$ be continuous, with $r(x,t)\ge\rmin>0$ for all $x$ in a measurable set $J\subset\Omega$ of positive measure and all $t\ge t_r$, for some $t_r\ge0$. Then:
\begin{enumerate}[label=(\roman*), itemsep=2pt]
\item the switching dynamics \eqref{eq:Rd} has a unique invariant probability density $\psi$, which is continuous and bounded below by a positive constant on $\bar\Omega$;
\item for every $u_0\in\Lone$ with $u_0\ge0$ and $U_0>0$, the mild solution of \eqref{eq:model-abstract} with this switching dynamics is global and nonnegative, $U$ is monotone and $U(t)\to\Ustar$, $\norm{u(t)-\Ustar\psi}\to0$, and $\norm{u(t)-S(t)f}\to0$ for every $f\in\Lone$ with $\int f=\Ustar$; estimate \eqref{eq:star} holds with the constants of Lemma~\ref{lem:doeblin}, computed with $\bar c=\min\{c\abs\Omega,1/2\}$, and, if $g'(\Ustar)<0$, the convergence to $\Ustar\psi$ is exponential;
\item the stationary flux $\mathbf j_\psi:=D\nabla\psi-v\psi$ is a continuous, divergence-free vector field on $\Omega$, tangent to $\partial\Omega$. It vanishes identically if and only if $D^{-1}v=\nabla\Phi$ for some function $\Phi$, and then $\psi\propto e^{\Phi}$; otherwise the stationary state carries a circulating probability flux.
\end{enumerate}
\end{maincor}

\begin{proof}
By the discussion above and Lemma~\ref{lem:doeblin}, applied with the constant $c\abs\Omega$ and $\phi=1/\abs\Omega$, $S$ satisfies (M1), with an exponential rate, and has a unique invariant probability density $\psi$, with $\psi\ge c\abs\Omega\,\phi=c$. Moreover $\psi=S(1)\psi$ is continuous on $\bar\Omega$, because $\Gamma(\cdot,\cdot,1)$ is. This proves (i), and (M2) holds with $h=\psi$, $\delta=\rmin c\abs J$ and $\mathcal T=[t_r,\infty)$. Theorem~\ref{thm:C}(a)--(e) gives (ii). For (iii), the function $u(x,t)=(S(t)\psi)(x)=\psi(x)$ is, for $t>0$, the classical solution of Theorem~\ref{bb:greenRd}(a) with initial datum $\psi$. Hence $\psi\in C^2(\Omega)$, $\nabla\psi$ is continuous up to $\partial\Omega$, $\nabla\cdot\mathbf j_\psi=\mathcal L\psi=\partial_tu=0$ in $\Omega$, and $\mathbf j_\psi\cdot n=0$ on $\partial\Omega$. If $\mathbf j_\psi\equiv0$, then $D^{-1}v=\nabla\ln\psi$. Conversely, if $D^{-1}v=\nabla\Phi$, then $\Phi\in C^{2,\alpha}(\bar\Omega)$, and $\tilde\psi:=e^\Phi/\int_\Omega e^\Phi$ satisfies $D\nabla\tilde\psi-v\tilde\psi=\tilde\psi\,(D\nabla\Phi-v)=0$. The time-independent function $\tilde\psi$ therefore has the properties listed in Theorem~\ref{bb:greenRd}(a), so $S(t)\tilde\psi=\tilde\psi$ by uniqueness, and $\tilde\psi=\psi$ by (i).
\end{proof}

In one dimension, $D^{-1}v=v/D$ is always a derivative, and Corollary~\ref{cor:Rd} recovers the explicit density of Theorem~\ref{thm:A}. In higher dimension, $\psi$ is in general not explicit, but, as for nonlocal switching, the conclusion does not depend on it: whatever $\psi$ is, it does not depend on the proliferation rates.

\section{Discussion}\label{sec:disc}

\paragraph{The role of each hypothesis.}
Each hypothesis enters the proof at identifiable places, and each is needed.

\emph{Uniform competition (H1)} enters in Steps~1 and~2: because every phenotype is regulated by the same factor $g(U)$, the total population obeys $U'=\rho\,g(U)$ with $\rho\ge0$, so it is monotone and $\int_0^\infty\abs{U'}<\infty$, and the reaction term has a single sign, so its size is the growth rate of the total population. If a death term acted on part of phenotype space outside $g$ (for instance, the effect of a drug on sensitive phenotypes), or if the modulation depended on the phenotype, the reaction term would have regions of both signs, its size would no longer be controlled by $\abs{U'}$, and fitness would shape the long-term distribution. The same happens for the additive form $r(x)-d(U)$ used in much of the phenotype-structured literature (Section~\ref{sec:intro}), in which competition is uniform but does not multiply the proliferation rate (Remark~\ref{rem:replicator} and Figure~\ref{fig:numerics}). In the notation of nonlocal competition models, where the net growth rate is $a(x)-\int b(x,y)u(y)\dd y$ \cite{DesvillettesJabinMischlerRaoul2008,JabinRaoul2011}, the logistic case of (H1) is $a=r$ with the kernel $b(x,y)=r(x)/K$, which depends on the phenotype that suffers competition but not on the one that exerts it, while additive competition with $d(U)=U/K$ is the constant kernel $b\equiv1/K$. For compartmental models, \cite{Review} shows that even a mild relaxation of (H1), in which the competitive pressure exerted by a cell depends on its phenotype, can make the equilibrium unstable and produce sustained oscillations; its continuum analogue is the kernel $b(x,y)=r(x)w(y)$. Which kernels other than $b(x,y)=r(x)/K$ still lead to the stationary distribution of the switching dynamics is an open question. Biologically, (H1) makes all phenotypes selectively neutral at saturation, since the net growth rate $g(\Ustar)\,r(x,t)$ vanishes for every $x$: differences in proliferation rates are forgotten, whereas differences in carrying capacity, as in $g=g(U,x)$, are not. The structure also requires switching to proceed independently of proliferation. If phenotypic changes occur at division, the switching operator depends on $r$, and fitness re-enters: through the additive class if density regulation acts on death, or because switching stops together with division at saturation, so that the composition freezes, if it multiplies division.

\emph{Non-degenerate switching (H3)} enters through the mixing estimate (Lemmas~\ref{lem:mix} and~\ref{lem:gap}). Without switching ($D\equiv0$, $v\equiv0$), \eqref{eq:model} reduces to $\partial_tu=g(U)r(x,t)u$, whose solution $u(x,t)=u_0(x)\exp\bigl(\int_0^tr(x,s)g(U(s))\dd s\bigr)$ retains a trace of the fitness landscape $r$ and of the initial condition forever; this is the counterpart of the compartmental example $A=0$ in \cite{Review}.

On the real line, the role of (H3) is shared by (R1)--(R3). Condition (R1) gives connectivity, as $D>0$ does on an interval. Condition (R3) guarantees that switching conserves the number of cells (Step~(a) of Proposition~\ref{prop:SR}), which is what makes $U'=\rho\,g(U)$ true. Condition (R2) provides a limiting profile; without it the conclusion changes form but survives (Remark~\ref{rem:heat}).

\emph{Persistent proliferation (H4)} enters only in Step~5, and only through $\int_0^\infty\rho\dd t=\infty$; Steps~1--4 do not use it. It can therefore be weakened. Writing $u=U\psi+d$ as in Step~5(B), $\rho(t)\ge m\int_\Omega r(x,t)\psi(x)\dd x-\rmax\norm{d(t)}$, and $\norm{d(t)}\to0$ by Step~4. Hence Theorems~\ref{thm:A} and~\ref{thm:B} remain true, with Step~5(B) modified accordingly, if there is $\delta>0$ such that the mean proliferation rate at the stationary composition satisfies $\int_\Omega r(x,t)\psi(x)\dd x\ge\delta$ for all $t$ in a set $\mathcal T$ of infinite measure (this is condition (M2) of Section~\ref{sec:general}, with $h=\psi$). This allows, for instance, proliferation that is switched on and off periodically, or a proliferating range of phenotypes that moves in time. Corollary~\ref{cor:rate} remains true if, moreover, $\liminf_{t\to\infty}t^{-1}\abs{\mathcal T\cap[0,t]}>0$, with a rate that also depends on $\delta$ and on this density. If $r\equiv0$ after some time, $U$ stops at a limit $U_\infty$, in general different from $\Ustar$; if it is different, no linear solution with mass $\Ustar$ can be approached, since two functions whose masses differ by $\abs{U_\infty-\Ustar}$ are at least that far apart in $L^1$.

\paragraph{Rough coefficients and other extensions.}
We assumed $D\in C^2$ and $v\in C^1$ to be able to quote classical parabolic theory. The proof of Theorem~\ref{thm:A} uses only the properties of $S$ established in Section~\ref{sec:linear}: Markov semigroup, invariant density $\psi$, and strict positivity $\Gamma(x,y,\tau)\ge c>0$. For bounded measurable $D\ge D_{\min}>0$ and $v$, the semigroup can be constructed with the Dirichlet form of Section~\ref{sec:R}, and positivity can be obtained from Aronson's Gaussian lower bound \cite{Aronson1968}: the substitutions $w=u/\psi$ and $y=\int_0^x\psi$ turn \eqref{eq:linear}--\eqref{eq:noflux} into $\partial_tw=\partial_y(D\psi^2\,\partial_yw)$ on $(0,1)$ with Neumann conditions, and an even reflection across the endpoints followed by periodic extension reduces this problem to a divergence-form equation on $\R$ with bounded measurable coefficients bounded away from zero; by the method of images, the Neumann kernel dominates the kernel on $\R$, to which Aronson's bound applies. We expect Theorem~\ref{thm:A} to extend to that setting along these lines: by Theorem~\ref{thm:C}, it suffices to construct the switching semigroup and establish (M1). The same remark applies to higher-dimensional phenotype spaces, which Corollary~\ref{cor:Rd} covers for smooth coefficients and bounded domains; unbounded domains in $\R^d$ would require a substitute for the one-dimensional construction of Section~\ref{sec:R}, which relied on the explicit stationary density. Continuity of $r$ in time can be relaxed to piecewise continuity, as in hypothesis (H4) of \cite{Review}, when the discontinuity times are the same for all phenotypes and finitely many in every bounded interval: one solves on each interval of continuity and glues the solutions (Section~\ref{sec:nonlinear}); $U$ is then continuous and piecewise $C^1$, it is still monotone, since $U-\Ustar$ cannot change sign on any interval of continuity, and the rest of the proof is unchanged. Time-dependent switching dynamics are not covered by Theorem~\ref{thm:C}, whose proof uses a fixed semigroup and, for the limit, a fixed invariant density; we leave them for future work.

\paragraph{Relation to the compartmental result.}
The proof shows that the mechanism behind the discrete result of \cite{Giaimo2025} is structural and survives the continuum limit: \emph{sign-definiteness} of the reaction term (from uniform competition) plus \emph{mixing} of the switching dynamics. Theorem~\ref{thm:C} makes this precise: it contains the compartmental result and Theorems~\ref{thm:A} and~\ref{thm:B} as special cases, and it applies to switching mechanisms of a different nature, such as nonlocal jumps. We did not pass to the continuum limit from the compartmental theorem, which would require the constants $C,\lambda$ of the discrete mixing estimate to be uniform in the number of compartments; we worked directly with the PDE. The numerical evidence in \cite{Review}, obtained by discretising the PDE into many compartments, is consistent with this: both the discrete and the continuum statements hold, with the same limit.

\paragraph{Biological reading.}
Under uniform competition, the phenotypic distribution of a population at carrying capacity is the stationary distribution of its switching dynamics. For a switching dynamics driven by an epigenetic potential $P$ and a noise intensity $D$, this is the Boltzmann-type density $\propto\exp(-\int P'/D)$, whatever the proliferation advantage of each phenotype. Fitness differences leave only a transient imprint, which fades at the mixing rate of the switching dynamics, its spectral gap $\lambda_1$; when phenotypic transitions are slow, as in landscapes with high barriers and low noise, this transient can nevertheless be long. In selection--mutation models with additive competition, the population concentrates, as phenotypic changes become small, on the phenotypes favoured by selection \cite{Perthame2007,LorzMirrahimiPerthame2011}. Under uniform competition with constant $D$, it concentrates instead, as $D$ decreases, on the global minima of the potential $P$, whatever their fitness; but when $P$ has several wells, the time $1/\lambda_1$ needed to forget fitness grows exponentially in $1/D$ (Figure~\ref{fig:numerics}(g)), so that the imprint of selection acquired during growth can dominate over long, although finite, times. The same conclusion holds when phenotypes change by jumps rather than by small continuous changes (Corollary~\ref{cor:nonlocal}); the limit is then the stationary density of the jump dynamics, for instance $\propto\phi/k$ when cells leave phenotype $x$ at rate $k(x)$ and land according to $\phi$. It also holds in multidimensional phenotype spaces with drifts that are not gradients (Corollary~\ref{cor:Rd}), where the limit is not of Boltzmann form and carries a circulating probability flux: the irrelevance of fitness does not rely on detailed balance.

\section*{Code availability}
The Python code that generates Figures~1--3 is available at \url{https://github.com/arturfassoni/uniform-competition-continuum}.

\section*{Acknowledgements}
Funding: This work was supported by CAPES and CNPq (Brazil), the Alexander von Humboldt Foundation (Germany), and partially by FAPEMIG (Brazil).

\section*{Declaration of generative AI and AI-assisted technologies in the manuscript preparation process}
The general strategy of the proofs was conceived by the author. During the preparation of this work the author used Claude (Anthropic) in order to search and verify parts of the scientific literature, to assist with technical steps of the proofs, to assist with writing the code that generates the figures, and to assist with translation and language editing. After using this tool, the author reviewed and edited the content as needed and takes full responsibility for the content of the published article.

\appendix

\section{Integrals of \texorpdfstring{$L^1$}{L1}-valued functions}\label{app:integrals}

We need to integrate functions of time with values in $\Lone$, as in Duhamel's formula $\int_0^tS(t-s)F(s)\dd s$. All such functions will be continuous in time, so the Riemann integral suffices: Riemann sums of a continuous $h:[a,b]\to\Lone$ converge in $\Lone$, because $\Lone$ is complete and $h$ is uniformly continuous (see, e.g., \cite{Lang1993}). The usual rules hold, with the same proofs as for real-valued functions: linearity; $\norm{\int_a^bh}\le\int_a^b\norm h$; the fundamental theorem of calculus in the form $\phi(b)f-\phi(a)f=\int_a^b\phi'(s)f\dd s$ for $\phi\in C^1$ and fixed $f$; interchange of iterated integrals over a triangle; and, since a.e.\ nonnegative functions form a closed subset of $\Lone$, $h\ge0$ implies $\int_a^bh\ge0$. Two further facts are used repeatedly.

\begin{lemma}\label{lem:riemann}
Let $h:[a,b]\to\Lone$ be continuous.
\begin{enumerate}[label=(\roman*), itemsep=2pt]
\item If $T:\Lone\to Y$ is a bounded linear operator into a Banach space $Y$, then $T\int_a^bh(s)\dd s=\int_a^bTh(s)\dd s$. In particular, for the mass functional $f\mapsto\int_\Omega f\dd x$,
\[
\int_\Omega\Bigl(\int_a^bh(s)\dd s\Bigr)\dd x=\int_a^b\Bigl(\int_\Omega h(s)\dd x\Bigr)\dd s .
\]
\item If $H:[a,b]\times[a,b]\to\Lone$ is continuous, then $t\mapsto\int_a^tH(s,t)\dd s$ is continuous on $[a,b]$.
\end{enumerate}
\end{lemma}

\begin{proof}
(i) holds for Riemann sums and passes to the limit by continuity of $T$. (ii) $H$ is uniformly continuous on the compact square; write $\int_a^{t'}H(s,t')\dd s-\int_a^tH(s,t)\dd s=\int_a^t[H(s,t')-H(s,t)]\dd s+\int_t^{t'}H(s,t')\dd s$ and bound both terms.
\end{proof}

\begin{lemma}[Duhamel integrals are continuous]\label{lem:duhamel-cont}
Let $S$ be a Markov semigroup and $h:[t_0,T]\to\Lone$ continuous. Then $t\mapsto\int_{t_0}^tS(t-s)h(s)\dd s$ is continuous on $[t_0,T]$, and its norm is at most $\int_{t_0}^t\norm{h(s)}\dd s$.
\end{lemma}

\begin{proof}
The map $(s,t)\mapsto S(\max(t-s,0))h(s)$ is continuous on $[t_0,T]^2$ by Lemma~\ref{lem:markov}(iii), and it equals $S(t-s)h(s)$ for $s\le t$. Apply Lemma~\ref{lem:riemann}(ii), the bound $\norm{\int h}\le\int\norm h$, and Lemma~\ref{lem:markov}(i).
\end{proof}

\section{Proofs for Section~\ref{sec:nonlinear}}\label{app:mild}

Throughout this appendix, $S$ is a Markov semigroup on $L^1(\Omega)$, and $r$ and $g$ are as in Section~\ref{sec:nonlinear}.

\begin{lemma}[Continuity of the reaction term]\label{lem:Fcont}
If $u:[a,b]\to\Lone$ is continuous with $\int u(t)>0$, then $F(t)=N(t,u(t))$ is continuous from $[a,b]$ to $\Lone$.
\end{lemma}

\begin{proof}
$t\mapsto U(t)=\int u(t)$ is continuous, hence so is $t\mapsto g(U(t))$. For $t,t_0\in[a,b]$,
\[
\norm{r(t)u(t)-r(t_0)u(t_0)}\le\rmax\norm{u(t)-u(t_0)}+\int_\Omega\abs{r(x,t)-r(x,t_0)}\,\abs{u(x,t_0)}\dd x .
\]
The first term tends to $0$ as $t\to t_0$. In the second, the integrand tends to $0$ pointwise (continuity of $r$) and is bounded by $2\rmax\abs{u(x,t_0)}$, an integrable function, so it tends to $0$ by dominated convergence.
\end{proof}

\begin{proof}[Proof of Proposition~\ref{prop:local}]
\emph{Bounds on the nonlinearity.} Let $Y:=\{\phi\in\Lone:\ \int\phi\ge a,\ \norm\phi\le R+a\}$. For $\phi\in Y$, $a\le\int\phi\le\norm\phi\le R+a$. Let $G:=\max_{[a,R+a]}\abs g$ and let $\Lambda$ be a Lipschitz constant of $g$ on $[a,R+a]$ (it exists since $g$ is $C^1$). For $\phi,\phi_1,\phi_2\in Y$ and all $t$:
\[
\norm{N(t,\phi)}\le G\rmax(R+a)=:B,
\]
and, writing $N(t,\phi_1)-N(t,\phi_2)=g(\textstyle\int\phi_1)\,r(t)(\phi_1-\phi_2)+\bigl(g(\int\phi_1)-g(\int\phi_2)\bigr)r(t)\phi_2$ and using $\abs{\int\phi_1-\int\phi_2}\le\norm{\phi_1-\phi_2}$,
\[
\norm{N(t,\phi_1)-N(t,\phi_2)}\le\rmax\bigl(G+\Lambda(R+a)\bigr)\norm{\phi_1-\phi_2}=:L_N\norm{\phi_1-\phi_2}.
\]

\emph{A complete metric space.} Let $T_0:=\min\{a/B,\,1/(2L_N)\}$, $I=[t_0,t_0+T_0]$, and let $X$ be the set of continuous $u:I\to\Lone$ with $\sup_{t\in I}\norm{u(t)-S(t-t_0)\phi_0}\le a$, with the distance $\sup_{t\in I}\norm{u_1(t)-u_2(t)}$. It is complete (a closed subset of the Banach space $C(I;\Lone)$). If $u\in X$, then $u(t)\in Y$ for all $t\in I$: indeed, by mass conservation and contraction,
\[
\int u(t)\ge\int S(t-t_0)\phi_0-\norm{u(t)-S(t-t_0)\phi_0}\ge2a-a=a,\qquad\norm{u(t)}\le\norm{\phi_0}+a\le R+a .
\]

\emph{A contraction.} For $u\in X$ let $\Phi(u)(t)$ be the right-hand side of \eqref{eq:mild-t0}. By Lemmas~\ref{lem:Fcont} and~\ref{lem:duhamel-cont}, $\Phi(u)$ is continuous, and
\[
\norm{\Phi(u)(t)-S(t-t_0)\phi_0}\le\int_{t_0}^t\norm{N(s,u(s))}\dd s\le T_0B\le a,
\]
so $\Phi$ maps $X$ into $X$. For $u_1,u_2\in X$,
\[
\norm{\Phi(u_1)(t)-\Phi(u_2)(t)}\le\int_{t_0}^t\norm{N(s,u_1(s))-N(s,u_2(s))}\dd s\le T_0L_N\sup_I\norm{u_1-u_2}\le\tfrac12\sup_I\norm{u_1-u_2} .
\]
By the Banach fixed-point theorem, $\Phi$ has a unique fixed point in $X$. For uniqueness among all continuous solutions of \eqref{eq:mild-t0} on $I$ with $\int u\ge a$, let $u_1,u_2$ be two of them. Their ranges are compact, hence contained in $\{\phi:\int\phi\ge a,\ \norm\phi\le R'\}$ for some $R'\ge R+a$, and the computation above, with $R'$ in place of $R+a$, gives a Lipschitz constant $L'$ of $N(t,\cdot)$ on this set. Hence $e(t):=\norm{u_1(t)-u_2(t)}$ satisfies $e(t)\le L'\int_{t_0}^te(s)\dd s$, and Gronwall's inequality gives $e\equiv0$.
\end{proof}

\begin{lemma}[Rescaling]\label{lem:rescale}
Let $q:\Omega\times[0,T]\to\R$ be bounded and such that $t\mapsto q(\cdot,t)w(t)$ is continuous in $\Lone$ whenever $w$ is, let $\beta\in C^1([0,T])$ be positive with $\beta(0)=1$, and let $w:[0,T]\to\Lone$ be continuous. Then $w$ satisfies
\begin{equation}\label{eq:lin-q}
w(t)=S(t)w_0+\int_0^tS(t-s)\,q(s)w(s)\dd s
\end{equation}
if and only if $W(t):=\beta(t)w(t)$ satisfies
\begin{equation}\label{eq:lin-qc}
W(t)=S(t)w_0+\int_0^tS(t-s)\Bigl(\frac{\beta'(s)}{\beta(s)}+q(s)\Bigr)W(s)\dd s .
\end{equation}
In particular, for $\beta(t)=e^{kt}$ with $k\in\R$, the coefficient in \eqref{eq:lin-qc} is $k+q$.
\end{lemma}

\begin{proof}
Assume \eqref{eq:lin-q}. Substitute \eqref{eq:lin-q} for $w(s)$ inside the term $\int_0^tS(t-s)\frac{\beta'(s)}{\beta(s)}W(s)\dd s=\int_0^t\beta'(s)S(t-s)w(s)\dd s$, and use $S(t-s)S(s)=S(t)$, $S(t-s)S(s-\sigma)=S(t-\sigma)$, and Fubini on the triangle $0\le\sigma\le s\le t$:
\[
\int_0^t\beta'(s)S(t-s)w(s)\dd s=\Bigl(\int_0^t\beta'(s)\dd s\Bigr)S(t)w_0+\int_0^t\Bigl(\int_\sigma^t\beta'(s)\dd s\Bigr)S(t-\sigma)q(\sigma)w(\sigma)\dd\sigma .
\]
Since $\int_0^t\beta'=\beta(t)-1$ and $\int_\sigma^t\beta'=\beta(t)-\beta(\sigma)$, the right-hand side of \eqref{eq:lin-qc} equals
\[
S(t)w_0+(\beta(t)-1)S(t)w_0+\int_0^t\bigl(\beta(t)-\beta(\sigma)\bigr)S(t-\sigma)q(\sigma)w(\sigma)\dd\sigma+\int_0^t\beta(\sigma)S(t-\sigma)q(\sigma)w(\sigma)\dd\sigma ,
\]
which is $\beta(t)\bigl(S(t)w_0+\int_0^tS(t-\sigma)q(\sigma)w(\sigma)\dd\sigma\bigr)=\beta(t)w(t)=W(t)$. The converse is the same computation with $\beta$ replaced by $1/\beta$, whose logarithmic derivative is $-\beta'/\beta$, and with the roles of \eqref{eq:lin-q} and \eqref{eq:lin-qc} exchanged.
\end{proof}

\begin{lemma}[Linear Volterra equations]\label{lem:volterra}
Let $q$ be as in Lemma~\ref{lem:rescale} with $\abs q\le Q$. Then \eqref{eq:lin-q} has exactly one continuous solution on $[0,T]$, and it is the uniform limit of the Picard iterates $w^{(0)}(t)=S(t)w_0$, $w^{(k+1)}(t)=S(t)w_0+\int_0^tS(t-s)q(s)w^{(k)}(s)\dd s$. If $q\ge0$ and $w_0\ge0$, the solution is nonnegative.
\end{lemma}

\begin{proof}
By contraction, $\norm{w^{(k+1)}(t)-w^{(k)}(t)}\le Q\int_0^t\norm{w^{(k)}(s)-w^{(k-1)}(s)}\dd s$, so by induction $\norm{w^{(k+1)}(t)-w^{(k)}(t)}\le\frac{(Qt)^k}{k!}\sup_{[0,T]}\norm{w^{(1)}-w^{(0)}}$. The series $\sum_k(QT)^k/k!$ converges, so the iterates converge uniformly to a continuous solution. If $w,\tilde w$ are two solutions, $e(t)=\norm{w(t)-\tilde w(t)}$ satisfies $e(t)\le Q\int_0^te(s)\dd s$, and Gronwall's inequality gives $e\equiv0$. If $q\ge0$ and $w_0\ge0$, every iterate is nonnegative by (S1) (integrals of nonnegative functions are nonnegative), and so is their $L^1$ limit.
\end{proof}


\end{document}